\documentclass[letterpaper,11pt,reqno]{amsart}
\usepackage[foot]{amsaddr}
\RequirePackage[utf8]{inputenc}
\usepackage[portrait,margin=2.5cm]{geometry}
\usepackage{latexsym,amssymb,amsmath,amsfonts,amsthm,mathrsfs,hyperref}
\hypersetup{hypertexnames=false,hidelinks,pdftitle={Cutoff for random walks on nilpotent groups},pdfauthor={Jonathan Hermon, Xiangying Huang, Justin Salez}}
\IfFileExists{dsfont.sty}{\usepackage{dsfont}}{\let\mathds\mathbf}
\usepackage{enumitem}
\usepackage{bm}
\makeatletter
\renewcommand{\subsubsection}{
  \@startsection{subsubsection}{3}{\z@}
  {.5\linespacing\@plus.7\linespacing}
  {-.5em}
  {\normalfont\itshape}
}
\makeatother

  \makeatletter
\@namedef{subjclassname@2020}{
  \textup{2020} Mathematics Subject Classification}
\makeatother

\newtheorem{thm}{Theorem}[section]

\newtheorem{definition}[thm]{Definition}
\newtheorem{lemma}[thm]{Lemma}
\newtheorem{fact}[thm]{Fact}
\newtheorem{remark}[thm]{Remark}
\newtheorem*{remark*}{Remark}
\newtheorem{prop}[thm]{Proposition}
\newtheorem{corollary}[thm]{Corollary}
\newtheorem{conj}{Conjecture}
\newtheorem{exam}[thm]{Example}

\def\beq{ \begin{equation} }
\def\eeq{ \end{equation} }
\def\mn{\medskip\noindent}

\def\ep{\epsilon}

\def\square{\vcenter{\vbox{\hrule height .4pt
  \hbox{\vrule width .4pt height 5pt \kern 5pt
        \vrule width .4pt} \hrule height .4pt}}}

\def\RR{\mathbb{R}}

\def\ZZ{\mathbb{Z}}

\def\EE{\mathcal{E}}

\def\LL{L}

\def\ep{\varepsilon}
\def\Cay{\mathrm{Cay}}

\def\dist{\mathrm{dist}}

\def\Cay{\mathrm{Cay}}
\def\Diam{\mathrm{Diam}}
\def\P{\mathbb{P}}
\def\E{\mathbb{E}}

\def\id{\mathrm{id}}
\def\1{\mathds{1}}
\def\0{\textbf{0}}
\def\ab{\mathrm{ab}}
\def\Unif{\mathrm{Unif}}

\def\rel{\mathrm{rel}}

\def\TV{\mathrm{TV}}
\def\mix{\mathrm{mix}}

\def\sym{\mathrm{sym}}

\def\supp{\mathrm{supp}}
\def\ev{\mathrm{e.v.}}

\newcommand{\qrel}[1]{\quad\mathrel{#1}\quad}
\newcommand{\crel}[1]{\mathrel{#1}}

\newcommand{\Quad}[1]{\quad\text{#1}\quad}

\newcommand{\inprod}[1]{\langle #1 \rangle}

\usepackage{xcolor}

\usepackage{graphicx}
\begin{document}

\title{Cutoff for random walks on nilpotent groups}
\date{}
\author[Hermon]{Jonathan Hermon$^{\ast}$}
\address{\normalfont$^{\ast}$The University of British Columbia, Department of Mathematics, 1984 Mathematics Road, Vancouver, BC V6T 1Z2, Canada. E-mail: \href{mailto:jhermon@math.ubc.ca}{\texttt{jhermon@math.ubc.ca}}. Supported by NSERC grants.}
\author[Huang]{Xiangying Huang$^{\dagger}$}
\address{\normalfont$^{\dagger}$Department of Statistics and Operations Research, University of North Carolina at Chapel Hill, 305 Hanes Hall, Chapel Hill, NC, USA. E-mail: \href{mailto:zoehuang@unc.edu}{\texttt{zoehuang@unc.edu}}.}
\author[Salez]{Justin Salez$^{\ddagger}$}
\address{\normalfont$^{\ddagger}$CEREMADE, Universit\'e Paris-Dauphine--PSL, Place du Mar\'echal de Lattre de Tassigny, 75775 Paris Cedex 16, France. E-mail: \href{mailto:justin.salez@dauphine.psl.eu}{\texttt{justin.salez@dauphine.psl.eu}}.}

\begin{abstract}
We study continuous-time random walks on finite nilpotent groups, allowing non-symmetric and nonuniform increment measures. Under a simple degree condition, sharp up to an absolute constant factor, we prove that total variation cutoff occurs if and only if the relaxation time is of smaller order than the mixing time. We establish a quantitative reduction to the Abelianization that preserves relaxation times and compares mixing profiles. Combining this reduction with a Gaussian approximation and the third-named author's cutoff criterion for diffusions yields the cutoff characterization.
\end{abstract}
\maketitle
\enlargethispage{4pt}
\setcounter{footnote}{0}

\section{Introduction}

In this work, we prove several results concerning random walks on Cayley graphs of finite nilpotent groups, which are valid for arbitrary generating sets and, more generally, for weighted and possibly non-reversible random walks. In particular, our main result, Theorem~\ref{thm: cutoff reduction to Abelianization}, shows that, under the explicit degree condition \eqref{eq:general_DC} below, total variation cutoff is equivalent to the product condition.\footnote{\label{fn:submission_history}The present paper and a companion paper on random Cayley graphs supersede our 2024 arXiv submission \cite{hermon2024cutoff}, which has been split into these two papers. That submission contained a weaker version of Theorem~\ref{thm: mixing and spectral comparison with Abelianization}, with a degree condition that was not sharp up to an absolute constant factor. The main addition relative to that submission is the cutoff characterization in Theorem~\ref{thm: cutoff reduction to Abelianization}. The short proof of this characterization was presented by the first-named author at the CMI--HIMR Summer School on Symmetry and Randomness, held in Bristol on 15--19 July 2024.} Our second main result, Theorem~\ref{thm: mixing and spectral comparison with Abelianization}, quantitatively reduces mixing and relaxation to those of the walk projected onto the Abelianization, as summarized in \eqref{e:trelG=trelGab} and \eqref{eq:intro_TV_reduction}. This provides both conceptual and quantitative refinements of the classical work of Diaconis and Saloff-Coste \cite{diaconis1994moderate} on random walks on Cayley graphs of moderate growth. Our results also elucidate the dependence---and, in some cases, the lack of dependence---of the spectral gap and the mixing time on the choice of generators.

We begin by introducing the notation needed to formulate our main
results. Throughout the paper, unless otherwise specified, $G$ is
a finite nilpotent group and $\mu$ is a probability measure whose
support $S:=\supp(\mu)$ generates $G$. Write $G_\ab:=G/[G,G]$ for the Abelianization of $G$, and let
$\mu_\ab$ denote the projected measure on $G_\ab$, defined by
\[
\mu_\ab(a)\qrel{:=}\sum_{x\in a}\mu(x),
\qquad a\in G_\ab.
\]
Set $S_\ab:=\supp(\mu_\ab)$, and use the same subscript to denote
the projection of other measures and sets.
Define the symmetrization of $\mu$ by
\begin{equation}\label{eq:def_symmetrization}
\mu_\sym(g)\qrel{:=}\frac{\mu(g)+\mu(g^{-1})}{2},
\qquad g\in G,
\end{equation}
and set $S_\sym:=\supp(\mu_\sym)=S\cup S^{-1}$.
We use the same subscript for the symmetrization of other measures
and sets. Projection and symmetrization commute; in particular, $\mu_{\ab,\sym}$ is the projection of $\mu_\sym$ and $S_{\ab,\sym}=\supp(\mu_{\ab,\sym})$. Throughout the paper, let
\begin{equation}\label{eq:def_mu_star}
m\qrel{:=}|S_{\ab,\sym}|
\qquad\text{and}\qquad
\mu_*\qrel{:=}\min_{a\in S_{\ab,\sym}}\mu_{\ab,\sym}(a).
\end{equation}

Let $L$ denote the nilpotency step of $G$.
Our main results are stated under the degree
condition
\begin{equation}\tag{DC}
\label{eq:general_DC}
m
\qrel{\leq}
\frac{\log|G_\ab|}
{5\max\{\log\log|G|,L\}+\log(1/(m\mu_*))}.
\end{equation}
As discussed in more detail below, \eqref{eq:general_DC} is sharp up to an absolute constant factor.
The overarching theme of this work is that, under \eqref{eq:general_DC}, the
mixing behavior of a random walk on a finite nilpotent group is governed, in
a strong quantitative sense, by that of its projection to the Abelianization.

\medskip\noindent
\textbf{Nilpotent-to-Abelian reduction.}
Theorem~\ref{thm: mixing and spectral comparison with Abelianization}
makes the above principle precise. Under
\eqref{eq:general_DC}, it establishes the identity
\begin{equation}\label{e:trelG=trelGab}
t_{\mathrm{rel}}(\mu)\qrel{=}t_{\mathrm{rel}}(\mu_\ab)
\end{equation}
for all sufficiently large $|G|$. Here, for a possibly non-symmetric measure $\mu$, the relaxation
time $t_{\mathrm{rel}}(\mu)$ is defined as the inverse spectral
gap of the symmetrized walk driven by $\mu_\sym$.
This identity follows from a more general quantitative
bound on relaxation times in
Theorem~\ref{thm:unconditional_relaxation_bounds},
which holds without assuming \eqref{eq:general_DC}.

Equation \eqref{e:trelG=trelGab} has a similar flavor to Aldous' spectral gap
conjecture for the interchange process, resolved by Caputo, Liggett, and
Richthammer \cite{CLR}. Their result implies that, for possibly weighted
Cayley graphs of $\mathcal{S}_n$ generated by transpositions, the spectral
gap is attained by the standard $(n-1)$-dimensional representation. 
In our nilpotent setting, \eqref{e:trelG=trelGab} yields
an analogous conclusion: under \eqref{eq:general_DC},
the spectral gap of the symmetrized walk on $G$ is attained
by a one-dimensional representation of $G$. 
More broadly, low-dimensional
irreducible representations often play a distinguished role in determining
the spectral gap and mixing behavior of random walks on finite groups; see,
for example, \cite{DS,Lacoin,nestoridi2021limit,EvitaSam}. For broad classes of conjugacy-invariant walks on $\mathcal S_n$, the cutoff time is governed by the low-dimensional representations indexed by Young diagrams whose first row has length $n-O(1)$; see \cite{DS,teyssier2020limit,EvitaSam,oleskertaylor2025sharp,teyssier2026profiles}.

The irreducible representations of $G$ that factor through its Abelianization are precisely the one-dimensional representations. Our reduction thus shows that, under \eqref{eq:general_DC}, the asymptotic mixing behavior is governed by these representations. More precisely, Theorem
\ref{thm: mixing and spectral comparison with Abelianization} asserts that,
under \eqref{eq:general_DC}, for every fixed $\varepsilon\in(0,1)$ and
$0<\delta<\min\{\varepsilon,1-\varepsilon\}$, for all sufficiently large $|G|$,
\begin{equation}\label{eq:intro_TV_reduction}
t_{\mathrm{mix}}^{\TV}(\mu_\ab,\varepsilon)
\qrel{\leq}
t_{\mathrm{mix}}^{\TV}(\mu,\varepsilon)
\qrel{\leq}
t_{\mathrm{mix}}^{\TV}(\mu_\ab,\varepsilon-\delta).
\end{equation}
Analogous comparisons hold for the $\ell^p$ and separation mixing times. Thus, the mixing
profile of the walk on $G$ is controlled from both sides by that of its
projection to $G_\ab$, up to an arbitrarily small shift in $\varepsilon$.

To the best of our knowledge, the only nilpotent groups for which
results in the general spirit of
Theorem~\ref{thm: mixing and spectral comparison with Abelianization}
had previously been established are the groups of $d\times d$ unit upper
triangular matrices over $\mathbb Z_q$ (see \S~\ref{s:related UPT}).
The only non-nilpotent group for which we are aware of such results is
the symmetric group, for which results of this kind were discussed above.
In both cases, the results were restricted to specific choices of generators.

Theorem~\ref{thm: cutoff reduction to Abelianization} gives a
cutoff characterization as a further consequence of the reduction
principle: under \eqref{eq:general_DC}, cutoff for the original
walks on nilpotent groups, cutoff for their Abelian projections,
and the corresponding product conditions are equivalent.
The proof combines our reduction theorem with a local central
limit theorem on the Abelianization and the cutoff criterion
for diffusions of non-negative curvature in \cite{salez2025cutoff}.
Notably, this characterization applies to general increment
measures, without requiring reversibility.

\medskip\noindent
\textbf{Optimality of \eqref{eq:general_DC}.}
The degree condition \eqref{eq:general_DC} is optimal up to
an absolute constant factor for the nilpotent-to-Abelian
reduction in Theorem~\ref{thm: mixing and spectral comparison with Abelianization}
and the cutoff equivalence in
Theorem~\ref{thm: cutoff reduction to Abelianization}.
As demonstrated in Example~\ref{example:constant_factor_sharpness}
below, relaxing \eqref{eq:general_DC} by only a constant
factor can invalidate the relaxation-time identity
\eqref{e:trelG=trelGab}, the mixing-time comparison
\eqref{eq:intro_TV_reduction}, and the cutoff equivalence,
even for simple random walks with symmetric support.

As discussed following Conjecture~\ref{op:nilpotent_product_condition},
\eqref{eq:general_DC} is nearly optimal even for the implication from the
product condition to cutoff in the Abelian setting: there are symmetric
simple random walks on cyclic groups for which this implication fails
while $|S|/\log|G|$ diverges arbitrarily slowly.

More generally, replacing $\log|G_\ab|$ in \eqref{eq:general_DC} by
$(\log|G_\ab|)^{1+\varepsilon}$ fails to ensure the mixing-time comparison
\eqref{eq:intro_TV_reduction} for most choices of generators whenever $\varepsilon$ exceeds a threshold
determined by $\log_{|G_\ab|}|G|$.
For canonical families of nilpotent groups with bounded step and rank,
this threshold can be arbitrarily small; see
Examples~\ref{matrix_walk} and~\ref{example:Free nilpotent group of rank r and step L}.

Weaker versions of
Theorems~\ref{thm: mixing and spectral comparison with Abelianization}
and~\ref{thm: cutoff reduction to Abelianization}
can be obtained under a more restrictive degree condition,
with substantially simpler proofs. See
Fact~\ref{lem:normal_subgroup_comparison} and the discussion
following it for this simpler approach.
Much of the technical work in
Sections~\ref{sec:geometry_Cayley}
and~\ref{reduction_abelianization} is devoted to
refining this approach and establishing the degree condition
\eqref{eq:general_DC}, which is optimal up to an
absolute constant factor.

\medskip\noindent
\textbf{Applications of our nilpotent-to-Abelian reduction.}
An application suggested to us by P.\ Diaconis concerns
the dependence of mixing times on the choice of generators.
Under \eqref{eq:general_DC},
Theorem~\ref{thm: mixing and spectral comparison with Abelianization}
answers a question of Diaconis and Saloff-Coste
\cite[Remark~2, p.~24]{diaconis1994moderate}:
for arbitrary increment measures $\mu$,
the relaxation time and, in the sense of
\eqref{eq:intro_TV_reduction}, the mixing times
depend on $\mu$ only through $\mu_\ab$.
This is the most one can hope for in general, since even
for Abelian groups, mixing and relaxation times may depend
drastically on the choice of generators. 
Corollary~\ref{thm:independence of choice of generators}
further identifies a setting in which this dependence
essentially disappears: simple random walks on symmetric
generating sets of minimum cardinality when
$G_\ab\cong\mathbb Z_q^r$.

Our reduction also extends geometric bounds on mixing
times from Abelian to nilpotent groups. 
Theorem~\ref{thm:Geometric bounds on mixing for Abelian}
determines the mixing time on finite Abelian groups up to
an $O(m)$ factor in terms of the square of an
$\ell^2$-diameter associated with a suitable weighted
``Euclidean norm'' on the group.
Corollary~\ref{thm: Geometric bounds on mixing for nilpotent}
extends these bounds to nilpotent groups under
\eqref{eq:general_DC}, providing the first such geometric
mixing-time bounds with no dependence on the growth
exponent $d$; see the discussion of moderate growth
in \S~\ref{s:related moderate growth}. This application was also suggested to us by P.\ Diaconis.

Theorem~\ref{thm: mixing and spectral comparison with Abelianization}
also plays a role in the work of the first- and second-named
authors \cite{hermon2024cutoff}, which establishes cutoff
for random walks on random Cayley graphs of finite
nilpotent groups with uniformly bounded step and rank.
The cutoff result proved there resolves and extends the
Aldous--Diaconis conjecture \cite{aldous1985shuffling}
for this class of groups. The conjecture has motivated
substantial work and was previously resolved in full generality
only for Abelian groups; see \cite{hermon2021cutoff}
and the references therein.
See \S~\ref{sec:applications} for applications of
Theorems~\ref{thm: mixing and spectral comparison with Abelianization}
and~\ref{thm: cutoff reduction to Abelianization}
in this setting.

 \section{Statements of our results}

 \subsection{Definitions and notation}\label{setting&notation}

The notation introduced above applies also to statements of results and remarks. The lower central series is defined by $G_1:=G$ and
$G_{i+1}:=[G_i,G]=\langle[x,y]:x\in G_i,\ y\in G\rangle$. Thus
\[
G=G_1\trianglerighteq G_2\trianglerighteq\cdots
\trianglerighteq G_L\trianglerighteq G_{L+1}=\{\id\},
\]
where $G_L$ is nontrivial. In particular, $G_2=[G,G]$ and $G_\ab=G/G_2$; the step $L$ is also called the nilpotency class.
A generating set is symmetric if it is closed under inversion. A measure $\mu$ is symmetric if $\mu(g)=\mu(g^{-1})$
for every $g\in G$, or equivalently, if the associated
random walk is reversible.
\enlargethispage{2pt}
\begin{definition}[Cayley graphs and associated notation]\label{def:weighted_Cayley_graph}
The \emph{Cayley digraph} $\Cay(G,S)$ has vertex set $G$ and
directed edge set
$
E_S:=\{(g,gs):g\in G,\ s\in S\}.
$
Its weighted version $\Cay(G,\mu)$ assigns weight $\mu(s)$
to each edge $(g,gs)$. When $S$ or $\mu$ is symmetric,
respectively, we regard the corresponding graph as undirected.

We write $\Gamma(\nu):=\Cay(H,\nu)$
 for a probability measure $\nu$ on a group $H$, with the ambient group understood from context. In particular, $\Gamma(\mu)=\Cay(G,\mu)$ and $\Gamma(\mu_\ab)=\Cay(G_\ab,\mu_\ab)$. We refer to the rate-$1$ continuous-time random walk driven by $\nu$ as the random walk on $\Gamma(\nu)$.
For the unweighted graphs associated with the symmetric supports, set
\[
\Gamma\qrel{:=}\Cay(G,S_\sym)
\qquad\text{and}\qquad
\Gamma_\ab\qrel{:=}\Cay(G_\ab,S_{\ab,\sym}).
\]
Throughout, let $D:=\Diam(\Gamma_\ab)$ denote the diameter
of the Abelianization, and $k$ the number of distinct inverse pairs
$\{a,a^{-1}\}$ in $S_{\ab,\sym}$. Recall that
$m=|S_{\ab,\sym}|$. A self-inverse element contributes one pair, so $k\leq m\leq2k$.
\end{definition}

\medskip
We use standard notation and definitions for mixing times and cutoff, see e.g. \cite[\S 4 and \S 18]{wilmer2009markov}. The definitions of the various mixing times are given in Definition \ref{def:TV and ell p mixing times} in Section \ref{reduction_abelianization}, while the definition of cutoff (in total variation distance) is given below.

\begin{definition}[Total variation cutoff]
For a sequence of finite Markov chains $X^{(n)}$ with stationary distributions $\pi_n$, write
$d_n(t):=\max_x\|\P_x(X_t^{(n)}=\cdot)-\pi_n\|_{\TV}$.
The sequence exhibits \emph{total variation cutoff} if there exists a sequence of times $(t_n)_{n\in\mathbb{N}}$ such that, for every $\ep\in(0,1)$,
\[\lim_{n\to\infty}d_n((1-\ep)t_n)\qrel{=}1
\qquad\text{and}\qquad
\lim_{n\to\infty}d_n((1+\ep)t_n)\qrel{=}0.
\]
For the group walks considered here the distance is independent of the initial state.
\end{definition}

\mn
Throughout the paper, when considering a sequence of random walks $(X^{(n)})_{n\in\mathbb N}$ indexed by $n\in\mathbb N$, we suppress the superscript $(n)$ from the walks and all associated quantities whenever no confusion can arise. All asymptotic statements are understood along the sequence $n\to\infty$.

\mn
\textbf{Asymptotic notation.}
\phantomsection\label{par:asymptotic_notation}
We use standard asymptotic notation: $\ll$ and $o(\cdot)$ indicate smaller order; $\gg$ indicates larger order; $\lesssim$ and $O(\cdot)$ indicate an upper bound up to a constant factor; $\gtrsim$ indicates the corresponding lower bound; $\asymp$ indicates bounds in both directions up to constant factors; and $\eqsim$ indicates asymptotic equivalence. A subscript, as in $\lesssim_\varepsilon$, $\gtrsim_\varepsilon$, or $\asymp_\varepsilon$, indicates that the implicit constants may depend on $\varepsilon$, but not on the sequence index. The subscript may denote any parameter or collection of parameters.

\subsection{Main results}
We use the standing notation above and impose \eqref{eq:general_DC} only where explicitly stated.
The following theorem is our main result for general generating sets. It
compares the mixing and relaxation behavior of the random walk on
$\Gamma(\mu)$ with that of its projection to the Abelianization $G_\ab$,
namely the random walk on $\Gamma(\mu_\ab)$.

Throughout this paper, we define the relaxation time of the random walk on
$\Gamma(\mu)$, denoted by $t_{\rel}(\mu)$, via symmetrization:
\[t_{\rel}(\mu)\qrel{:=}t_{\rel}(\mu_{\sym}),
\]
where $t_{\rel}(\mu_\sym)$ is the inverse spectral gap of the walk driven by
$\mu_{\sym}$; see Definition~\ref{def:TV and ell p mixing times}(d) and the
discussion thereafter.
Analogously,
$t_{\rel}(\mu_\ab):=t_{\rel}(\mu_{\ab,\sym})$ denotes the relaxation time of the projected walk on
$\Gamma(\mu_\ab)$.
We write $t_{\mix}^{\ell^p}(\mu,\ep)$ and
$t_{\mathrm{sep}}(\mu,\ep)$ for the $\ell^p$ and separation mixing times,
respectively, of the walk on $\Gamma(\mu)$, and use analogous notation for
its projection to $G_\ab$.

\begin{thm}[Nilpotent to Abelian reduction]
\label{thm: mixing and spectral comparison with Abelianization}
Suppose that \eqref{eq:general_DC} holds. There is an absolute
$N_0$ such that, whenever $|G|\geq N_0$, the following identity holds:
\begin{equation}\label{e:gap given by Abelianization}
t_{\mathrm{rel}}(\mu)
\qrel{=}
t_{\mathrm{rel}}(\mu_\ab).
\end{equation}
Moreover, for every fixed $\ep\in(0,1)$ and
$0<\delta<\min\{\ep,1-\ep\}$, there exists
$N_1(\delta)$, depending only on $\delta$, such that, for every $p\in[1,\infty]$, the following comparisons\footnote{Our convention is $t_{\mix}^{\TV}(\nu,\ep_0)=t_{\mix}^{\ell^1}(\nu,2\ep_0)$ for any increment measure $\nu$. Thus, for $p=1$, \eqref{e:mixing time given by Abelianization} compares total variation mixing times at thresholds $\ep$ and $\ep-\delta$.}
 hold whenever
$|G|\geq N_1(\delta)$: \begin{equation}\label{e:mixing time given by Abelianization}
t_{\mathrm{mix}}^{\ell^p}(\mu_\ab,2\ep)
\qrel{\leq}
t_{\mathrm{mix}}^{\ell^p}(\mu,2\ep)
\qrel{\leq}
t_{\mathrm{mix}}^{\ell^p}(\mu_\ab,2\ep-2\delta).
\end{equation}
\begin{equation}\label{e:separation mixing time given by Abelianization}
t_{\mathrm{sep}}(\mu_\ab,\ep)
\qrel{\leq}
t_{\mathrm{sep}}(\mu,\ep)
\qrel{\leq}
t_{\mathrm{sep}}(\mu_\ab,\ep-\delta).
\end{equation}
\end{thm}

\begin{thm}[Relaxation-time bounds without (DC)]
\label{thm:unconditional_relaxation_bounds}
There exist absolute constants
$C_1,C_2>0$ such that
\begin{equation}\label{eq:relaxation_time_comparison}
t_{\rel}(\mu)
\qrel{\leq}
\max\left\{t_{\rel}(\mu_\ab),\;
\frac{C_1m}{\mu_*}\left(64^Lm+D\right)\right\},
\end{equation}
and
\begin{equation}\label{eq:weak_relaxation_comparison}
t_{\rel}(\mu)
\qrel{\leq}
\max\left\{t_{\rel}(\mu_\ab),\;
C_2m^2\left(\frac{64^L}{\mu_*}+\frac{m^2}{\mu_*^2}\right)\right\}.
\end{equation}
\end{thm}

\begin{remark}[Consequences of \eqref{e:gap given by Abelianization} for non-reversible walks]
\label{rmk:nonreversible_relaxation}
When $\mu$ is not symmetric, there are several possible notions of relaxation time. For a probability measure $\nu$ on a finite group $H$, let $P_\nu$ denote the transition matrix defined by $P_\nu(x,y):=\nu(x^{-1}y)$ for $x,y\in H$. Let $t_\rel^{\ev}(\mu)$ denote the reciprocal of the difference between $1$ and the largest real part of a nontrivial eigenvalue of $P_\mu$. Since the transition operator of an Abelian walk is normal, $t_\rel(\mu_\ab)=t_\rel^{\ev}(\mu_\ab)$. Moreover, the eigenvalues of $P_{\mu_\ab}$ are also eigenvalues of $P_\mu$, and $t_\rel^{\ev}(\mu)\leq t_\rel(\mu)$ by \eqref{eq:inequality_two_relaxation_times}. As explained in \S~\ref{reduction_abelianization}, these facts imply that $t_\rel(\mu)=t_\rel^{\ev}(\mu)$ whenever \eqref{e:gap given by Abelianization} holds. When these two notions coincide: (i) the familiar bounds $t_\rel(\mu)\lesssim t_\mix^\TV(\mu)\lesssim t_\rel(\mu)\log|G|$ hold, as in the reversible setting, and (ii) for any sequence of walks satisfying $t_\rel(\mu^{(n)})=t_\rel^{\ev}(\mu^{(n)})$, cutoff necessarily implies the product condition $t_\rel(\mu^{(n)})\ll t_\mix^\TV(\mu^{(n)})$, by \eqref{eq:necessity_product_condition}.
\end{remark}

An important consequence of Theorem~\ref{thm: mixing and spectral comparison with Abelianization} is that the cutoff problem on the nilpotent group reduces to that on its Abelianization. Under \eqref{eq:general_DC}, the walk on
$
\Gamma(\mu)
$
exhibits cutoff if and only if its projection
$
\Gamma(\mu_\ab)
$
does. A precise formulation, together with an equivalent product condition, is given in Theorem~\ref{thm: cutoff reduction to Abelianization} and proved in Section~\ref{sec:cutoff_abelianization}.

For the cutoff statements, consider a sequence $(G^{(n)},\mu^{(n)})$ in the standing setup, with $|G^{(n)}|\to\infty$ and \eqref{eq:general_DC} holding for every $n$. Superscripts $(n)$ denote the corresponding objects. Let $X^{(n)}$ be the walk driven by $\mu^{(n)}$, and $X^{\ab,(n)}$ its projection to $G_\ab^{(n)}$.

 \begin{thm}[A characterization of cutoff via the product condition]
        \label{thm: cutoff reduction to Abelianization}
Under the notation above and assumption \eqref{eq:general_DC}, the following are equivalent:
         \begin{enumerate}[label=(\roman*)]
         \item $(X^{(n)})_{n\in\mathbb{N}}$ exhibits cutoff;
         \item $(X^{\ab,(n)})_{n\in\mathbb{N}}$ exhibits cutoff;
         \item  The product condition $t_{\mathrm{rel}}(\mu_{\mathrm{ab}}^{(n)}) \ll t^{\TV}_{\mathrm{mix}}(\mu_{\mathrm{ab}}^{(n)})$ holds for the projected walk  $(X^{\ab,(n)})_{n\in\mathbb{N}}$ on the Abelianization;
         \item The seemingly weaker condition $t_{\mathrm{rel}}(\mu_{\mathrm{ab}}^{(n)}) \ll t^{\TV}_{\mathrm{mix}}(\mu^{(n)})$ holds.
        \end{enumerate}
 \end{thm}
\begin{remark}[Cutoff time and window]
\label{rmk:mixing_window}
Under \eqref{eq:general_DC}, for every fixed $\ep\in(0,1/2)$,
\begin{equation}\label{eq:nilpotent_mixing_window}
t^{\TV}_{\mix}(\mu^{(n)},\ep)-t^{\TV}_{\mix}(\mu^{(n)},1-\ep)
\qrel{\lesssim_\ep}
\sqrt{t^{\TV}_{\mix}(\mu_\ab^{(n)})\,t_{\rel}(\mu_\ab^{(n)})}.
\end{equation}
This follows from the same comparison argument and the diffusion-window bound
in \cite[Theorem~1]{salez2025cutoff}. A justification is provided at the end
of Section~\ref{sec:cutoff_abelianization}, where we show that the cutoff
times in Theorem~\ref{thm: cutoff reduction to Abelianization}\textup{(i)}
and \textup{(ii)} are asymptotically equivalent to that of
an associated weighted Brownian motion on a certain flat torus.
\end{remark}

\begin{remark}[Relaxing (DC)]\label{rmk:mixing_time_condition}\label{rmk:large_mixing}
\begin{enumerate}[label=(\roman*)]
\item\label{rmk:weaker_comparison_part} There is an absolute constant $C>0$
such that \eqref{e:gap given by Abelianization} holds whenever
\begin{equation}\label{eq:diameter_comparison_condition}
D\qrel{\geq}C\max\left\{8^L\sqrt{\mathfrak C},\;\frac{\mathfrak C}{m}\right\}
\qquad\text{where}\qquad
\mathfrak C\qrel{:=}\frac{m^3\log|G|}{\mu_*}.
\end{equation}
Moreover, for every $\ep\in(0,1)$ and
$0<\delta<\min\{\ep,1-\ep\}$, the comparisons
\eqref{e:mixing time given by Abelianization} and
\eqref{e:separation mixing time given by Abelianization} hold for all
$p\in[1,\infty]$ when \eqref{eq:diameter_comparison_condition} holds with
$C=C(\delta)$ depending only on $\delta$.
In Remark~\ref{rmk:geometric_comparison_conditions}, we show that
\eqref{eq:general_DC} implies
$D\gg\max\{8^L\sqrt{\mathfrak C},\mathfrak C/m\}$
and establish the preceding assertions.

\item\label{rmk:weaker_cutoff_part} Theorem~\ref{thm: cutoff reduction to Abelianization}
remains valid under the asymptotic condition
$D\gg\max\{8^L\sqrt{\mathfrak C},\mathfrak C/m\}$
in place of \eqref{eq:general_DC}.
A justification is given in Section~\ref{sec:cutoff_abelianization}.
\end{enumerate}
\end{remark}

\begin{conj}\label{op:nilpotent_product_condition}
Consider a sequence of finite nilpotent groups $G=G^{(n)}$
of diverging sizes, whose nilpotency steps are 
uniformly bounded. Suppressing the dependence
on $n$, suppose that $1/\mu_*\lesssim\log|G_{\mathrm{ab}}|$.
Then the product condition is necessary and sufficient for cutoff.
\end{conj}

For finite Abelian groups $G$,
Theorem~\ref{thm: cutoff reduction to Abelianization}
establishes the sufficiency of the product condition for cutoff
under \eqref{eq:general_DC}, which holds whenever
$1/\mu_*\leq c\log|G|/\log\log|G|$
for a sufficiently small absolute constant $c>0$.
In work in preparation, Blachar and the first- and third-named
authors \cite{blacharHermonSalez2026cutoff}
resolve Conjecture~\ref{op:nilpotent_product_condition}
in the Abelian setting; see \S~\ref{s:cutoff_criteria}. They also construct a sequence of
symmetric simple random walks on finite cyclic groups $G$
for which the product condition holds but cutoff fails,
with $|S|/\log|G|$ diverging arbitrarily slowly.
This example shows that, for the sufficiency of the product
condition for cutoff, the degree condition \eqref{eq:general_DC}
is optimal up to a factor of $\log\log|G|$ times an arbitrarily
slowly diverging factor.

In the nilpotent setting, Theorem~\ref{thm: cutoff reduction to Abelianization}
partially resolves Conjecture~\ref{op:nilpotent_product_condition}
by reducing cutoff to the Abelianization under \eqref{eq:general_DC}.
For $L=O(\log\log|G|)$, this condition forces
$m\lesssim\log|G_\ab|/\log\log|G|$, whereas the conjecture
allows $m\lesssim\log|G_\ab|$. 
Example~\ref{example:constant_factor_sharpness} shows that the implication
(iii) $\Rightarrow$ (i) in
Theorem~\ref{thm: cutoff reduction to Abelianization} can fail when
\eqref{eq:general_DC} is violated by even a constant factor,
indicating that closing this gap requires an approach beyond
the present nilpotent-to-Abelian reduction.

\medskip

We now state a consequence of
Theorem~\ref{thm: mixing and spectral comparison with Abelianization}
addressing the question of Diaconis and Saloff-Coste~\cite[Remark 2, p.~24]{diaconis1994moderate} 
concerning the dependence of mixing times on the choice of generators.

For rate-$1$ simple random walk on a graph $\Gamma$, we write
$t_{\rel}(\Gamma)$, $t_{\mix}^{\TV}(\Gamma,\varepsilon)$,
and $t_{\mix}^{\ell^p}(\Gamma,\varepsilon)$ for its relaxation time,
total-variation mixing time, and $\ell^p$-mixing time, respectively,
where $p\in[1,\infty]$.
Throughout, omitted mixing thresholds are $1/4$ for total variation
and $1/2$ for $\ell^p$.
\begin{corollary}\label{thm:independence of choice of generators}
Suppose that $G_\ab\cong\ZZ_q^r$ for some $q\geq2$, and that
$
r=r(G)\leq
\frac{\log|G_\ab|}{10\max\{\log\log|G|,L\}}.
$
Assume that $\mu_\ab$ is uniform on a symmetric generating set
of $G_\ab$ of minimum cardinality. Let
$
\mathbb T_r(q):=\Cay(\ZZ_q^r,\{\pm e_j:j\in[r]\})
$
denote the $r$-dimensional discrete torus of side length $q$.
There exists an absolute constant $N_0$ such that, whenever
$|G|\geq N_0$,
\[
t_{\rel}(\mu)\qrel{=}t_{\rel}(\mathbb T_r(q)).
\]
Moreover, for every fixed $\ep\in(0,2)$ and
$0<\delta<\min\{\ep,2-\ep\}$, there exists $N_1(\delta)$
such that, whenever $|G|\geq N_1(\delta)$, for every
$p\in[1,\infty]$,
\[
t_{\mix}^{\ell^p}(\mathbb T_r(q),\ep)
\qrel{\leq} t_{\mix}^{\ell^p}(\mu,\ep)
\qrel{\leq} t_{\mix}^{\ell^p}(\mathbb T_r(q),\ep-\delta).
\]
\end{corollary}
In particular, since $r(G)=r(G_\ab)$, these conclusions apply to simple random walk
on any symmetric generating set of $G$ of minimum cardinality.
Indeed, an automorphism of $\ZZ_q^r$ carries any symmetric
generating set of minimum cardinality to the standard generating
set $\{\pm e_j:j\in[r]\}$, so the result follows from
Theorem~\ref{thm: mixing and spectral comparison with Abelianization}.

\subsection{Geometric bounds on mixing times}

We next state mixing-time bounds for finite Abelian groups that underlie the
reduction to the Abelianization and extend to weighted, non-symmetric random
walks. For the Abelian results, we specialize the standing notation to
a finite Abelian group $G=A$ and use additive notation.
In particular, $\mu_\ab=\mu$, $S_\ab=S$,
$S_{\ab,\sym}=S_\sym$, and $\Gamma_\ab=\Gamma$.
Choose representatives $s_1,\ldots,s_k\in S$ of the
distinct inverse pairs in $S_\sym$, so that
\begin{equation}\label{eq:abelian_inverse_pairs}
S_\sym\qrel{=}S\cup(-S)\qrel{=}
\bigsqcup_{j=1}^k\{\pm s_j\},
\end{equation}
where $\{\pm s_j\}$ is a singleton if $s_j=-s_j$.
In this setting,
$\mu_*=\min_{s\in S_\sym}\mu_\sym(s)$.

Every element $g\in A$ can be written as
$g=\sum_{j=1}^k z_js_j$ for some $z=(z_1,\dots,z_k)\in\ZZ^k$.
Define
\begin{equation}\label{eq:quotient_Euclidean_norm}
\|g\|_{2,S}\qrel{=}\|g\|_{2,S_\sym}\qrel{:=}
\min\left\{
\|z\|_2:
z\in\ZZ^k,\ \sum_{j=1}^k z_js_j=g
\right\},
\end{equation}
and the associated \emph{$\ell^2$-diameter} by
\begin{equation}\label{eq:euclidean_diameter}
D_{\ell^2}(\Gamma)\qrel{:=}
\max_{g\in A}\|g\|_{2,S}.
\end{equation}
We also introduce the weighted versions
\begin{equation}\label{eq:weighted_quotient_Euclidean_norm}
\|g\|_{2,\mu}\qrel{=}\|g\|_{2,\mu_\sym}\qrel{:=}
\min\left\{
\left(\sum_{j=1}^k
\frac{z_j^2}{m\mu_\sym(s_j)}\right)^{1/2}:
z\in\ZZ^k,\ \sum_{j=1}^k z_js_j=g
\right\},
\end{equation}
and
\begin{equation}\label{eq:weighted_l2_diameter}
D_{\ell^2(\mu)}(\Gamma)\qrel{:=}
\max_{g\in A}\|g\|_{2,\mu}.
\end{equation}

Throughout the paper, we write $D_2:=D_{\ell^2}(\Gamma_\ab)$ and $D_2(\mu):=D_{\ell^2(\mu_\ab)}(\Gamma_\ab)$; in the present Abelian setting these are $D_{\ell^2}(\Gamma)$ and $D_{\ell^2(\mu)}(\Gamma)$, respectively. When $\mu_{\ab,\sym}$ is uniform on its support, the weighted and unweighted norms as well as their
diameters coincide.

Recall the standing notation $k$, $m$, and $D$ from
Definition~\ref{def:weighted_Cayley_graph}.

\begin{thm}[Geometric mixing time bounds for Abelian groups]
\label{thm:Geometric bounds on mixing for Abelian}
In the Abelian setup, the following statements hold:
\begin{itemize}

\item[(i)] \emph{Geometric bounds:} We have
\begin{equation}\label{e:universal geometric mixing time bounds 1}
\max\{t_{\rel}(\mu),D_2^2,D_2^2(\mu)\}
\lesssim t_{\mix}^{\TV}(\mu)\leq t_{\mix}^{\ell^\infty}(\mu)
\lesssim k\min\{t_{\rel}(\mu),D_2^2(\mu)\}\leq\frac{D_2^2}{\mu_*}.
\end{equation}
Moreover, there exist absolute constants $c,C>0$ such that,
for every $0<\varepsilon\leq 1/2$ satisfying
$\varepsilon D_2\geq C\log(1/\varepsilon)$,
\begin{equation}\label{eq:geometric_high_threshold_lower_bound}
t_{\mix}^{\TV}(\mu,1-\varepsilon)
\qrel{\geq}
c\frac{\varepsilon^2D_2^2}{\log(1/\varepsilon)}.
\end{equation}

\item[(ii)] \emph{Comparison between symmetric and non-symmetric RWs:}
For every $\varepsilon>0$,
\begin{equation}\label{eq:unconditional_symmetrization_mixing_times}
\begin{aligned}
\tfrac12 t_{\mix}^{\ell^\infty}(\mu,\varepsilon^2)
&\qrel{\leq}t_{\mix}^{\ell^2}(\mu,\varepsilon)
\qrel{=}t_{\mix}^{\ell^2}(\mu_\sym,\varepsilon)
\qrel{=}\tfrac12 t_{\mix}^{\ell^\infty}(\mu_\sym,\varepsilon^2),\\
\tfrac12 t_{\mix}^{\TV}(\mu_\sym,\varepsilon)
&\qrel{\leq}t_{\mix}^{\TV}(\mu,\varepsilon).
\end{aligned}
\end{equation}
Moreover, if
\begin{equation}\label{eq:stronger_comparison_condition}
k\qrel{\leq}\frac{\log|A|}{2\log\log|A|+\log(1/(m\mu_*))},
\end{equation}
then, for every fixed $\ep\in(0,1)$ and
$0<\delta<\min\{\ep,1-\ep\}$, there exists a finite
$N_0(\delta)$, depending only on $\delta$,
such that whenever $|A|\geq N_0(\delta)$,
\begin{equation}\label{eq:TV_symmetrization_mixing_times}
t_{\mathrm{mix}}^{\mathrm{TV}}
(\mu_{\mathrm{sym}},\ep+\delta)
\qrel{\leq}
t_{\mathrm{mix}}^{\mathrm{TV}}
(\mu,\ep)
\qrel{\leq}
t_{\mathrm{mix}}^{\mathrm{TV}}
(\mu_{\mathrm{sym}},\ep-\delta).
\end{equation}

\end{itemize}
\end{thm}

\begin{remark}
All of the bounds in \eqref{e:universal geometric mixing time bounds 1} can be saturated up to a constant factor; see the \hyperref[sec:sharpness of geometric bounds]{sharpness examples} at the end of Section~\ref{s:geometric_bound_mixing_times}.
\end{remark}

Related mixing-time bounds from earlier work are discussed in
Remark~\ref{rmk:abelian_geometric_proofs}.
The above Abelian bounds extend to nilpotent groups as follows.

\begin{corollary}[Geometric mixing-time bounds for nilpotent groups]
\label{thm: Geometric bounds on mixing for nilpotent}
In the nilpotent setup, assume \eqref{eq:general_DC}.
There exists an absolute constant $N_0$ such that,
whenever $|G|\geq N_0$,
\begin{equation}\label{eq:nilpotent_geometric_mixing_bounds}
\max\{t_{\rel}(\mu),D_2^2,D_2^2(\mu)\}
\lesssim t_{\mix}^{\TV}(\mu)
\leq t_{\mix}^{\ell^\infty}(\mu)
\lesssim k\min\{t_{\rel}(\mu),D_2^2(\mu)\}
\leq \frac{D_2^2}{\mu_*}.
\end{equation}
Moreover, for every fixed $\ep\in(0,1)$ and
$0<\delta<\min\{\ep,1-\ep\}$, there exists
$N_1(\delta)$, depending only on $\delta$, such that
whenever $|G|\geq N_1(\delta)$,
\begin{equation}\label{eq:nilpotent_TV_symmetrization}
t_{\mix}^{\TV}(\mu_\sym,\ep+\delta)
\qrel{\leq}
t_{\mix}^{\TV}(\mu,\ep)
\qrel{\leq}
t_{\mix}^{\TV}(\mu_\sym,\ep-\delta).
\end{equation}
\end{corollary}

Under \eqref{eq:general_DC}, Corollary~\ref{thm: Geometric bounds on mixing for nilpotent} resolves a problem suggested to us by Diaconis (private communication), concerning improvements to the parameter dependence in the mixing-time bounds of Diaconis and Saloff-Coste; see the moderate-growth discussion in Section~\ref{s:related moderate growth}.

Corollary~\ref{thm: Geometric bounds on mixing for nilpotent} follows by
combining Theorem~\ref{thm: mixing and spectral comparison with Abelianization}
with parts~(i) and~(ii) of
Theorem~\ref{thm:Geometric bounds on mixing for Abelian}. The details are
presented in Section~\ref{s:geometric_bound_mixing_times}.

\subsection{Organization of the paper}

Section~\ref{sec:discussion} discusses applications to random Cayley
graphs, examples demonstrating the sharpness of \eqref{eq:general_DC},
weaker degree conditions for conjugacy-invariant walks, and connections
to related work. We also state a conjecture bounding the relax\-ation-time
ratio $t_\rel(\mu)/t_\rel(\mu_\ab)$ in terms of the nilpotency step when
\eqref{eq:general_DC} fails.
Section~\ref{sec:geometry_Cayley} establishes quantitative word-length
bounds for iterated commutators and corresponding bounds on graph
distances and diameters in nilpotent Cayley graphs. These estimates
provide the geometric input for the reduction arguments in
Section~\ref{reduction_abelianization}.
There, we prove the nilpotent-to-Abelian reduction theorem
(Theorem~\ref{thm: mixing and spectral comparison with Abelianization})
and the relaxation-time bounds in
Theorem~\ref{thm:unconditional_relaxation_bounds}.
Section~\ref{s:Abelian  mixing} establishes geometric mixing bounds
for finite Abelian groups (Theorem~\ref{thm:Geometric bounds on mixing for Abelian})
and their extension to nilpotent groups
(Corollary~\ref{thm: Geometric bounds on mixing for nilpotent}).
Section~\ref{sec:weighted_RW} proves the cutoff criterion
in Theorem~\ref{thm: cutoff reduction to Abelianization}.
Finally, Appendix~\ref{app:biased_local_limit} proves
the one-dimensional local limit theorem.

\subsection*{Acknowledgements}
We thank Persi Diaconis for pointing out the potential applications of Theorem \ref{thm: mixing and spectral comparison with Abelianization}, as well as for sharing with us his insights concerning the general phenomenon of the mixing behavior of the walk being related to low-dimensional representations.
We also thank Tom Hutchcroft, Mathav Murugan, Allan Sly and Matt Tointon for useful discussions. The authors acknowledge support from NSERC (J.H.) and the ERC (J.S.).

\subsection*{Statement of AI use}
All fundamental mathematical ideas, proof strategies, and non-routine technical arguments are entirely the authors' own and were developed independently of AI assistance (see also Footnote~\ref{fn:submission_history}). AI tools were used solely to refine the exposition and assist with some routine mathematical work.

\section{Discussion of results and related work}\label{sec:discussion}

\subsection{Applications to random walks on random Cayley graphs}
\label{sec:applications}

Consider a sequence of finite nilpotent groups whose sizes tend to infinity and whose ranks and steps are uniformly bounded. We suppress the sequence index from the notation. Let $Z_1,\ldots,Z_k$ be independent uniform elements of $G$, where $1\ll\log k\ll\log|G|$.

With high probability, the $Z_j$ are distinct, and conditional on this event their unordered collection is a uniform $k$-element subset of $G$. Moreover, their projections onto $G_\ab$ belong to distinct inverse pairs with high probability, so $k$ agrees with our standing notation. Indeed, the probability of a coincidence up to inversion among the projected generators is at most $k^2/|G_\ab|=o(1)$, since Lemma~\ref{ab_size} gives $\log|G_\ab|\asymp\log|G|$.

Throughout this subsection, we consider the walk with increment measure
\[
\mu\qrel{=}\frac1{2k}\sum_{j=1}^k(\delta_{Z_j}+\delta_{Z_j^{-1}}).
\]
In the Abelian setup, cutoff for this walk is established in \cite{hermon2021cutoff}. Under \eqref{eq:general_DC}, Theorem~\ref{thm: mixing and spectral comparison with Abelianization} transfers this cutoff result to the nilpotent setup above.

A particularly useful application of this reduction concerns the cutoff
result of Hermon and Huang \cite{hermon2024cutoff} for random walks on
random Cayley graphs of finite nilpotent groups. It was proved in \cite{hermon2024cutoff} that, in this setup, without assuming \eqref{eq:general_DC}, the walk on $\Gamma(\mu)$ exhibits
cutoff with high probability at
\[t_*(k,G)
\qrel{:=}
\max\bigl\{t_{\mathrm{ent}}(k,|G_\ab|),\log_k|G|\bigr\},
\]
where $t_{\mathrm{ent}}(k,N)$ is the unique time at which the Shannon entropy of the rate-$1$ simple random walk on $\ZZ^k$ equals $\log N$; its asymptotics are well understood \cite{hermon2021cutoff}.

In the lower-degree regime $1\ll k\ll\log|G_\ab|$, the
results of  \cite{hermon2021cutoff} imply that, with high probability, the
walk on $\Gamma(\mu_\ab)$ exhibits cutoff and
\[t_{\mix}^{\TV}(\mu_\ab)
\qrel{\eqsim}
t_{\mathrm{ent}}(k,|G_\ab|)
\qrel{\eqsim}
\frac{k}{2\pi e}|G_\ab|^{2/k}.
\]
Hence, in the range of degrees covered by \eqref{eq:general_DC}, Theorems
\ref{thm: mixing and spectral comparison with Abelianization} and
\ref{thm: cutoff reduction to Abelianization} imply that, with high
probability, the walk on $\Gamma(\mu)$ also exhibits cutoff and
\[t_{\mix}^{\TV}(\mu)
\qrel{\eqsim}
t_{\mix}^{\TV}(\mu_\ab)
\qrel{\eqsim}
\frac{k}{2\pi e}|G_\ab|^{2/k}.
\]
This is the argument used in \cite{hermon2024cutoff} to establish the
lower-degree cutoff result and mixing-time asymptotics, by combining the
nilpotent-to-Abelian reduction under \eqref{eq:general_DC} with the
corresponding Abelian results.

Finally, suppose that, instead of taking $Z_1,\ldots,Z_k$ to be independent and uniform on $G$, we choose them so that their projections $G_2Z_1,\ldots,G_2Z_k$ are independent and uniform on $G_\ab$. Conditional on these projections, the lifts $Z_1,\ldots,Z_k$ may be chosen arbitrarily. With $\mu$ defined by the same formula as above, the conclusions obtained from our reduction theorem continue to hold under \eqref{eq:general_DC} in this more general quasi-random setting.

\subsection{Discussion of the degree condition \texorpdfstring{\eqref{eq:general_DC}}{(DC)}}
\label{s:DC}

A natural question is to what extent the degree condition
\eqref{eq:general_DC} can be relaxed while retaining some of the conclusions
of Theorem~\ref{thm: mixing and spectral comparison with Abelianization}.

Such a relaxation is indeed possible under additional structural assumptions
on the generating sets. Building on ideas developed in the present paper and
in \cite{hermon2024cutoff}, subsequent work by the
second-named author \cite{huang2026conjugacyinvariantrandomwalksnilpotent} establishes analogous
results for conjugacy-invariant random walks on nilpotent groups. In
particular, \cite[Theorems~1.1 and~1.3]{huang2026conjugacyinvariantrandomwalksnilpotent}
show that the dominance of the Abelianization persists under a substantially
weaker condition than \eqref{eq:general_DC}, which is independent of the
step $L$.

In general, however, there are obstructions to relaxing
\eqref{eq:general_DC}. The dihedral example below shows that the
relaxation-time ratio can grow exponentially with the step, even at rank $2$.
Building on this example, the construction on $\mathbb Z_p\times D_{2^{L+1}}$
in Example~\ref{example:constant_factor_sharpness} shows that
\eqref{eq:general_DC} is sharp up to an absolute constant factor for
Theorems~\ref{thm: mixing and spectral comparison with Abelianization}
and~\ref{thm: cutoff reduction to Abelianization}.
The canonical families in Examples~\ref{matrix_walk}
and~\ref{example:Free nilpotent group of rank r and step L} provide
obstructions with degree exponent arbitrarily close to $1$, even with
bounded step and rank.

\begin{exam}
\label{example:Dyhiadral}
Consider the dihedral group
\[
G:=D_{2^{L+1}}
=\langle r,s\mid r^{2^L}=\id=s^2,\ sr=r^{-1}s\rangle
\]
with increment measure
$\mu:=\Unif(\{s,rs\}).$
\end{exam}

Here $G_\ab\cong\ZZ_2^2$, $|G|=2^{L+1}$,
and $G$ has step $L$. Since $s$ and $rs$ are involutions
that generate $G$, the associated Cayley graph is connected
and 2-regular, and hence is a cycle of length $2^{L+1}$.
As $\mu$ is uniform on $\{s,rs\}$, the walk is simply
the simple random walk on this cycle. Consequently,
\[
t_\rel(\mu)
=\left(1-\cos\left(\frac{2\pi}{2^{L+1}}\right)\right)^{-1}
\asymp 4^L
\asymp |G|^2.
\]
On the other hand, the projected images of $s$ and $rs$ generate
$G_\ab\cong\ZZ_2^2$, and the corresponding Cayley graph
is a cycle of length $4$. Hence $t_\rel(\mu_\ab)=1$,
and therefore
\[
\frac{t_\rel(\mu)}{t_\rel(\mu_\ab)}
\asymp 4^L.
\]

Thus, the ratio between the relaxation times of the walk
and its Abelian projection can grow exponentially in the
step $L$, even at rank $r(G)=2$. In particular, without
\eqref{eq:general_DC}, one cannot expect a better dependence
on $L$ than the $4^{L-1}$ factor proposed in
Conjecture~\ref{conj:step_relaxation_comparison}.
For large $L$, this unavoidable exponential loss differs drastically from the exact equality of relaxation times in Theorem~\ref{thm: mixing and spectral comparison with Abelianization} under \eqref{eq:general_DC}.

The following conjecture refines \eqref{eq:weak_relaxation_comparison}
and suggests Example~\ref{example:Dyhiadral} is extremal.

\begin{conj}\label{conj:step_relaxation_comparison}
Let $p_*$ be the smallest prime divisor of $|G|$. Then
\[
\frac{t_\rel(\mu)}{t_\rel(\mu_\ab)}
\qrel{\leq}
\left(\frac{p_*}{p_*-1}\right)^{2L-2}
\qrel{\leq}
4^{L-1}.
\]
In particular, if $L\leq 100 p_*$, then
$t_\rel(\mu)\leq\bigl(1+O(L/p_*)\bigr)t_\rel(\mu_\ab)$.
\end{conj}

\medskip
\begin{exam}[Sharpness of \eqref{eq:general_DC} up to a constant factor]
\label{example:constant_factor_sharpness}
Let $p=p_n$ be a sequence of primes tending to infinity,
and let $d=d_n\asymp\frac{\log p_n}{\log\log p_n}$,
to be specified below.
For $2\leq L\lesssim\log\log p$, consider the nilpotent group
$G:=\mathbb Z_p\times D_{2^{L+1}}.$
Choose $Z_1,\ldots,Z_d$
independently and uniformly from $\mathbb Z_p$, and set
\[
S_1\qrel{:=}\{(\pm Z_j,\id):1\leq j\leq d\},
\qquad
S_2\qrel{:=}\{(0,s),(0,rs)\},
\]
where $r,s$ are as in Example~\ref{example:Dyhiadral}. 
Set $S:=S_1\cup S_2$, and consider the simple random
walk on $G$ with increment measure $\mu:=\Unif(S)$.
\end{exam}
To analyze this example, we recall basic facts about mixing and
relaxation times of product chains in the setting needed here.
Let $H=H_1\times H_2$ be a direct product of two finite groups,
and let $\nu_i$ be a generating probability measure on $H_i$.
Identifying $H_1$ with $H_1\times\{\id_{H_2}\}$ and $H_2$ with
$\{\id_{H_1}\}\times H_2$, we view
$\nu:=\eta\nu_1+(1-\eta)\nu_2$, for $\eta\in(0,1)$, as a probability
measure on $H$. 

The walk driven by $\nu$ is a \emph{product chain}: its coordinates
evolve independently according to $\nu_1$ and $\nu_2$, at rates $\eta$
and $1-\eta$, respectively. Thus (cf.\ \cite[Corollary~12.13]{wilmer2009markov})
\[
t_\rel(\nu)\qrel{=}
\max\left\{\frac{t_\rel(\nu_1)}{\eta},
\frac{t_\rel(\nu_2)}{1-\eta}\right\},
\]
and, for every $\varepsilon\in(0,1)$,
\[
\max\left\{\frac{t_\mix^\TV(\nu_1,\varepsilon)}{\eta},
\frac{t_\mix^\TV(\nu_2,\varepsilon)}{1-\eta}\right\}
\leq t_\mix^\TV(\nu,\varepsilon) \leq
\max\left\{\frac{t_\mix^\TV(\nu_1,\varepsilon/2)}{\eta},
\frac{t_\mix^\TV(\nu_2,\varepsilon/2)}{1-\eta}\right\}.
\]
The TV bounds follow by projection and the fact
that the TV distance between product measures is at most the sum of the
coordinate distances.

\medskip
We have $G_\ab\cong\mathbb Z_p\times\mathbb Z_2^2$,
and both elements of $S_2$ are involutions with distinct
images in the Abelianization. With high probability,
the $2d$ elements $\pm Z_1,\ldots,\pm Z_d$ are all
nonzero and distinct. On this event, $S$ and $S_\ab$
are symmetric, with
\[
|S|=|S_\ab|=m=2d+2,\qquad
k=d+2,\qquad
\mu_*=\frac{1}{2d+2},
\]
and both $\mu$ and $\mu_\ab$ are uniform on their respective
supports. Write
$
\nu_1:=\Unif(\{\pm Z_1,\ldots,\pm Z_d\}),
$ and $
\nu_2:=\Unif(\{s,rs\})
$
for the cyclic and dihedral increment measures. Under the coordinate subgroup identifications above,
\[
\mu=\frac{d}{d+1}\nu_1+\frac{1}{d+1}\nu_2.
\]
Thus, the two coordinates evolve independently at rates
$d/(d+1)$ and $1/(d+1)$, respectively, and the Abelian
projection has the same coordinate rates.

By \cite[Theorem~E]{hermon2023geometry} and
\cite[Theorem~A]{hermon2021cutoff}, with high probability,
$
t_\rel(\nu_1)\asymp p^{2/d}$ and 
$t_\mix^\TV(\nu_1)\asymp d p^{2/d}$,
and the cyclic walk driven by $\nu_1$ exhibits cutoff. The Cayley graph of $D_{2^{L+1}}$ generated by $s$ and $rs$ is a cycle
of length $2^{L+1}$. Hence, as in Example~\ref{example:Dyhiadral},
the rate-$1$ dihedral walk has relaxation and TV mixing times of order
$4^L$, while its Abelian projection has both times of order $1$.
Applying the product-chain estimates above to
$\mu$ and $\mu_\ab$, with $\eta=d/(d+1)$,
we obtain, with high probability,
\[
\begin{aligned}
t_\rel(\mu_\ab)&\qrel{\asymp}\max\{p^{2/d},d\},
&t_\rel(\mu)&\qrel{\asymp}\max\{p^{2/d},d4^L\},\\
t_\mix^\TV(\mu_\ab)&\qrel{\asymp}d p^{2/d},
&t_\mix^\TV(\mu)&\qrel{\asymp}\max\{d p^{2/d},d4^L\}.
\end{aligned}
\]

We illustrate the sharpness of \eqref{eq:general_DC}
for our results using this example.
Note that $|G_\ab|=4p$,
$\log\log|G|\eqsim\log\log p$, and $m\mu_*=1$.

\smallskip
\noindent (i) \emph{Relaxation times.}
Choose $d$ such that $p^{2/d}\asymp d$, so that
$d\eqsim 2\log p/\log\log p$, and take
$1\ll L\leq\log\log p$. Then
\[
\frac{t_\rel(\mu)}{t_\rel(\mu_\ab)}
\qrel{\asymp}4^L\longrightarrow\infty,
\qquad\text{while}\qquad
\frac{5m\max\{\log\log|G|,L\}}{\log|G_\ab|}
\longrightarrow  20.
\]
Thus, the equality $t_\rel(\mu)=t_\rel(\mu_\ab)$
in \eqref{e:gap given by Abelianization} can fail
when \eqref{eq:general_DC} is violated by a factor
of only $20+o(1)$.

\smallskip
\noindent (ii) \emph{Mixing times and cutoff.}
Fix $\varepsilon\in(0,1)$ and choose $C\geq2$ such that
$C^2\geq2/(\varepsilon\log4)$. Set
\[
d:=\left\lceil\frac{C\log p}{\log\log p}\right\rceil,
\qquad
L:=\lceil C\log\log p\rceil.
\]
Then $p^{2/d}\eqsim(\log p)^{2/C}$ and
$4^L\asymp(\log p)^{C\log4}$.
Since
\[
\frac{5m\max\{\log\log|G|,L\}}{\log|G_\ab|}
\longrightarrow 10C^2,
\]
the condition \eqref{eq:general_DC} is violated by
a factor of $10C^2+o(1)$.

With these choices,
\[
\frac{t_\mix^\TV(\mu)}{t_\mix^\TV(\mu_\ab)}
\qrel{\asymp}\frac{4^L}{p^{2/d}}
\qrel{\asymp}4^{(1-2/(C^2\log4))L}
\qrel{\gtrsim}4^{(1-\varepsilon)L}.
\]
The same lower bound holds for
$t_{\mathrm{sep}}(\mu,1/4)/t_{\mathrm{sep}}(\mu_\ab,1/4)$,
by comparability with TV mixing times at fixed thresholds
for reversible chains \cite[Lemmas~6.16 and~6.17]{wilmer2009markov}. Thus, both the $\ell^1$
mixing-time comparison
\eqref{e:mixing time given by Abelianization}
and the separation-time comparison
\eqref{e:separation mixing time given by Abelianization}
can fail.

The cutoff equivalence in
Theorem~\ref{thm: cutoff reduction to Abelianization}
can also fail: the projected walk exhibits cutoff,
while the original walk does not. Indeed, the projected
walk exhibits cutoff since its cyclic factor does,
while its $\mathbb Z_2^2$ factor mixes in time
$O(d)=o(dp^{2/d})$. In contrast,
\[
t_\rel(\mu)\qrel{\asymp}t_\mix^\TV(\mu)
\qrel{\asymp}d4^L,
\]
so the original walk fails the product condition
and hence does not exhibit cutoff.

\begin{exam}\label{matrix_walk}
Consider the group of $d\times d$ unit upper triangular matrices $G=\mathbb U_d(p)$ over $\ZZ_p$, for fixed $d\geq2$. Choose $R$ uniformly among the $k$-element subsets of $G$, and let
\[\mu_R\qrel{:=}\frac{1}{2k}\sum_{s\in R}(\delta_s+\delta_{s^{-1}})\qquad\text{and}\qquad
S\qrel{:=}\supp(\mu_R)\qrel{=}R\cup R^{-1}.
\]
We work in the regime $1\ll\log k\ll\log|G|$, using the sampling equivalence explained in Section~\ref{sec:applications}.
\end{exam}

The results of  \cite{hermon2024cutoff} on random Cayley
graphs of finite nilpotent groups include Example~\ref{matrix_walk} as a
special case; the latter was previously studied by Hermon and Olesker-Taylor \cite{hermon2019cutoff}. Since the results of
\cite{hermon2024cutoff} apply in the general nilpotent setting, we formulate
the discussion below at this level of generality, with
Example~\ref{matrix_walk} as a particular case.

As discussed above,
\cite{hermon2024cutoff} shows that, for sequences of nilpotent groups with bounded rank and step and $1\ll\log k\ll\log|G|$, the cutoff
time is asymptotic to
\[t_*(k,G)
\qrel{=}
\max\bigl\{t_{\mathrm{ent}}(k,|G_\ab|),\log_k|G|\bigr\}.
\]
Thus the cutoff time, as a function of $k$, depends on both $|G_\ab|$ and $|G|$.
In particular, write
$\rho:=\frac{\log k}{\log\log|G_\ab|}$ and assume that $\liminf_{|G|\to\infty}\rho>1$.
If
\[(1+\eta)\frac{\rho}{\rho-1}
\qrel{\leq}
\frac{\log|G|}{\log|G_\ab|},
\]
for some fixed $\eta>0$ and all sufficiently large $|G|$, then
$
t_*(k,G)\eqsim \log_k|G|.
$
Hence the mixing behavior is no longer governed by the Abelianization alone.
We restrict this comparison to regimes with a fixed asymptotic margin of dominance and make no assertion about the boundary regime.
For such $k$, the total variation mixing-time comparison in
Theorem~\ref{thm: mixing and spectral comparison with Abelianization}
in fact fails for most $k$-element subsets $R$ with increment measure $\mu_R$;
see \cite{hermon2024cutoff}. Moreover, for families of nilpotent groups with a fixed $L\geq 2$, for which
$
\frac{\log|G|}{\log|G_\ab|}
$
is arbitrarily large, the exponent
\[\frac{\log|G|}{\log|G|-\log|G_\ab|}
\qrel{=}
1+
\frac{\log|G_\ab|}{\log|G|-\log|G_\ab|}
\]
can be made arbitrarily close to $1$; see, e.g.,
Example~\ref{example:Free nilpotent group of rank r and step L}. Hence, in
general, one cannot expect the mixing-time comparison with the
Abelianization to extend much beyond the degree scale
$|S|\lesssim\log|G_\ab|$, even when $L$ is uniformly bounded. In this sense,
the degree condition \eqref{eq:general_DC} is nearly optimal.

Recall that $r$ denotes the rank of $G$ and $L$ its step. It follows from
Lemma~\ref{ab_size} that
$
\frac{\log |G|}{\log |G_{\ab}|}\leq 2r^L.
$
Consequently, \eqref{eq:general_DC} is implied by the stronger, but sometimes
more convenient, condition
\beq\label{eq:DC2}\tag{DC2}
m
\qrel{\leq}
\frac{\log|G|}
{2r^L\left(5\max\{\log\log|G|,L\}+\log(1/(m\mu_*))\right)}.
\eeq

The next example shows that the dependence on $r$ and $L$ in
Lemma~\ref{ab_size} is essentially sharp at the level of exponential growth
in $L$. Consequently, without additional information about the group, one
cannot expect to replace the factor $r^L$ in the stronger sufficient
condition \eqref{eq:DC2} by a quantity with substantially weaker dependence on $r$ and $L$.

\begin{exam}
\label{example:Free nilpotent group of rank r and step L}
Let $r\geq2$, and let $p>L$ be prime. Take $G=G(r,L,p)$ to be the free nilpotent group of exponent $p$, rank $r$ and step $L$, obtained from the free nilpotent Lie algebra over $\mathbb F_p$ by the Lazard correspondence. By Witt's formula (see, for example, \cite[p.~412]{magnus2004combinatorial}),
\[\frac{\log|G|}{\log|G_\ab|}
\qrel{\asymp}
\frac{1}{L}r^{L-1}.
\]
\end{exam}

\subsection{Related work}
\label{s:related}

\setcounter{secnumdepth}{3}
\subsubsection{Cutoff criteria in the Abelian setting}
\label{s:cutoff_criteria}
A substantial body of work studies cutoff for random walks on Abelian
groups, providing partial progress toward
Conjecture~\ref{op:nilpotent_product_condition} as well as results
in regimes beyond its scope, notably when
$d_*:=\max\{1/\mu(s):s\in\supp(\mu)\}\gtrsim\log|G|$.
Hermon and Olesker-Taylor \cite{hermon2021cutoff} establish cutoff for random generating sets under nearly optimal assumptions on the number of generators relative to the rank. For prescribed generators, earlier general criteria impose quantitative conditions stronger than the product condition, requiring $t_{\mix}^{\TV}/t_{\rel}$ to diverge at a sufficiently fast rate. The third-named author \cite{salez2023cutoff} obtained cutoff criteria for non-negatively curved Markov chains that apply, in particular, to Abelian Cayley walks. The authors and Pedrotti \cite{hermon2024concentration} subsequently bounded varentropy in terms of the number of generators, yielding the sufficient condition $t_{\rel}\sqrt{d_*}\ll t_{\mix}^{\TV}$ in the symmetric Abelian setting. Pedrotti and the third-named author \cite{pedrotti2025newcutoffcriterionnonnegatively} give a criterion involving the modified logarithmic Sobolev time, $t_{\mathrm{MLS}}\log\log d_*\ll t_{\mix}^{\TV}$. Their argument also yields the sufficient condition $t_{\rel}\log d_*\ll t_{\mix}^{\TV}$.
The earlier result \cite[Corollary~3]{salez2023cutoff}
also covers a complementary regime: the random walk on
an Abelian Cayley graph of degree $d$ on $G$ exhibits cutoff
whenever
$t_{\rel}\ll (\log |G|)^{1/2}/(\log d)^{3/2}$.
In particular, for Abelian Cayley expanders, where
$t_{\rel}=O(1)$, this condition reduces to
$\log d\ll(\log |G|)^{1/3}$.
Such expanders necessarily have $d\gtrsim\log |G|$, so this regime lies beyond \eqref{eq:general_DC}.

Under \eqref{eq:general_DC}, Theorem~\ref{thm: cutoff reduction to Abelianization} establishes
that the product condition is necessary and sufficient
for cutoff on Abelian groups,
and extends this characterization to the associated
nilpotent groups.

For random walks on finite Abelian groups, Blachar and
the first- and third-named authors
\cite{blacharHermonSalez2026cutoff} establish the more
general sufficient condition
\[
t_{\rel}\left(1\vee\log\left(
\frac{1}{\mu_*\,t_{\mix}^{\TV}}
\right)\right)
\qrel{\ll}t_{\mix}^{\TV}.
\]
They also show that
the logarithmic factor is optimal in order, even
for simple random walks on cyclic groups. In the regime $\mu_*^{-1}\lesssim\log|G|$, 
the logarithmic factor is bounded, and their sufficient
condition reduces to the product condition.

\subsubsection{\texorpdfstring{Related work on unit upper-triangular matrices}{Related work on unit upper-triangular matrices}}
\label{s:related UPT}

There are several results for unit upper-triangular matrix groups that are
closely related in spirit to our comparison with the Abelianization. Take $G=\mathbb{U}_d(q)$ to be the group of $d\times d$ upper-triangular matrices
over $\mathbb{Z}_q$ with all diagonal entries equal to $1$. Its Abelianization
is naturally identified with $\mathbb{Z}_q^{d-1}$ through the first
super-diagonal entries.

A canonical generating set for $\mathbb{U}_d(q)$ is
$
S:=\{I\pm E_{j,j+1}:1\leq j<d\},
$
where $E_{i,j}$ denotes the matrix with a $1$ in the $(i,j)$ entry and zeros
elsewhere. For simple random walk on $\Gamma$, the projected walk on the Abelianization is the product chain on
\[\Gamma_\ab
\qrel{=}
\Cay\left(
\mathbb{Z}_q^{d-1},
\{\pm e_i:i\in[d-1]\}
\right).
\]
For this walk, condition \eqref{eq:general_DC} holds, for all sufficiently
large $q$, whenever $d\leq c\log q$ for some fixed $c<1/10$. In this regime,
Theorem~\ref{thm: mixing and spectral comparison with Abelianization} gives
\[t_{\rel}(\Gamma)
\qrel{=}
t_{\rel}(\Gamma_\ab),
\]
together with the corresponding comparison of the mixing profiles.

Beyond the regime covered by \eqref{eq:general_DC}, several earlier results
exhibit a similar phenomenon. Stong \cite{stong1995random} showed that the
spectral gaps of $\Gamma$ and $\Gamma_\ab$ differ by at most an
absolute constant factor, uniformly in $q$ and $d$. For mixing times,
Nestoridi and Sly \cite{nestoridi2020random} proved that
\[t_{\mix}^{\TV}(\Gamma)
\qrel{=}
O\left(q^2d\log d+d^2q^{o(1)}\right).
\]
Since
$
t_{\mix}^{\TV}(\Gamma_\ab)\asymp q^2d\log d,
$
their result shows that the mixing times of the walk on $\Gamma$ and
its projection to the Abelianization are of the same order whenever
$d\leq q^{2-\eta}$ for any fixed $\eta>0$.

Nestoridi and Sly \cite{nestoridi2020random} build upon the coordinate estimates of Diaconis and Hough
\cite[Theorem~3 and the following remarks]{diaconis2021random}. For prime $q$, fixed matrix dimension and a fixed increment law satisfying their hypotheses, Diaconis and Hough show that an individual entry on the $\ell$-th super-diagonal mixes on the scale $q^{2/\ell}$. Nestoridi and Sly \cite[Lemma~22]{nestoridi2020random} track a bound of order $2^\ell q^{2/\ell}$, using rate-one clocks for each row.\footnote{Converting this clock normalization to total jump rate one introduces a factor of order $d$.} These estimates describe faster equilibration of individual entries farther from the first super-diagonal, which corresponds to the Abelianization; joint mixing of an entire super-diagonal is a separate question. An analogous phenomenon appears in our proof in Section~\ref{reduction_abelianization}, which gives quantitative upper bounds on the time required for
\[\left\|\P_{\pi_{G_\ell}}(X_t\in\cdot)-\P_{\pi_{G_{\ell+1}}}(X_t\in\cdot)\right\|_{\TV}
\]
to become small, where $\pi_H$ denotes the uniform measure on $H$. This measures equilibration in the layer $G_\ell/G_{\ell+1}$, without requiring the projected walk on $G/G_\ell$ to have mixed. For unit upper-triangular groups, this layer corresponds to the $\ell$-th super-diagonal. Remark~\ref{rmk:layer_mixing} records the quantitative bound.

For a more comprehensive account of the literature on random walks on
unit upper-triangular matrix groups, we refer the reader to
\cite[\S~1.3.2]{hermon2019cutoff}.

\enlargethispage{2pt}
\subsubsection{\texorpdfstring{Related work on Cayley graphs of moderate growth}{Related work on Cayley graphs of moderate growth}}
\label{s:related moderate growth}

For a symmetric generating set $S$, write $D_\Gamma:=\Diam(\Gamma)$. Then moderate growth can be characterized, up to changes of fixed parameters, by
\[|G|\qrel{\lesssim} |S|D_\Gamma^d
\]
for a fixed exponent $d$; this follows from Breuillard and Tointon \cite[Corollary~1.9]{breuillard2016nilprogressions}. This diameter formulation is easier to verify than the original ball-growth condition of Diaconis and Saloff-Coste \cite{diaconis1994moderate}, whose parameters $(\mathsf{A},d)$ are a multiplicative volume constant and a growth exponent. The equivalence need not preserve their numerical values.

For a symmetric walk, their quantitative bounds can be summarized as
\[
\frac{D_\Gamma^2}{\mathsf{A}^2\,16^d}
\qrel{\lesssim} t_{\mix}^{\TV}
\qrel{\lesssim} \frac{D_\Gamma^2}{\mu_{\min}}
\bigl[1+\log \mathsf{A}+d\log(d+1)\bigr],
\]
where $\mu_{\min}:=\min_{s\in\supp(\mu)}\mu(s)$ and $(\mathsf{A},d)$ are the
ball-growth parameters in the convention of \cite[Definition~5.1]{Nash}.
The lower bound follows from the spectral argument
in \cite[Theorem~3.1]{diaconis1994moderate}, while the upper bound follows
from \cite[Theorem~5.1 and Lemma~5.3]{Nash}, improving the $d^2$ dependence
in \cite{diaconis1994moderate} to $d\log(d+1)$. For nilpotent groups, $d$ can itself grow very rapidly with the degree and step.\footnote{For the groups $G(r,r,p)$ in Example~\ref{example:Free nilpotent group of rank r and step L}, equipped with their canonical generators, the degree is $|S|=2r$, Witt's formula gives $\log_p|G|\asymp r^{r-1}$ (see \cite[p.~412]{magnus2004combinatorial}), and Theorem~\ref{diam_bound} gives $D_\Gamma\asymp_r p$. Consequently, any bound $|G|\lesssim_r |S|D_\Gamma^d$ holding uniformly in $p$ requires $d\geq\log_p|G|\gtrsim r^{r-1}$.} Thus the familiar description $t^{\TV}_{\mix}\asymp D_\Gamma^2$ conceals substantial losses when these parameters vary. Improving this dependence was suggested to us by Diaconis (private communication). For simple random walks with both degree and step bounded, their results give uniform comparison with $D_\Gamma^2$ and exclude cutoff.

\setcounter{secnumdepth}{2}

\section{Geometry of Nilpotent Cayley Graphs}\label{Geo_Cayley}\label{sec:geometry_Cayley}

Throughout this section, $S$ is an arbitrary symmetric generating set. Thus $S_\sym=S$, $S_{\ab,\sym}=S_\ab$, and $\Gamma=\Cay(G,S)$; the notation $m$, $k$ and $D$ retains its standing meaning. The arguments are geometric and do not depend on the increment probabilities.

The main objective of this section is to establish quantitative estimates on the graph distance of elements of $G$; see Proposition \ref{prop:commutator_decomposition}. These estimates will be used later in the proof of our main comparison result, Theorem~\ref{thm: mixing and spectral comparison with Abelianization}. As a further consequence, they also yield, with little additional effort, a quantitative bound on the diameter of $\Cay(G,S)$, which we record in Theorem~\ref{diam_bound}. This has independent geometric interest: whenever the bound on $\Diam_S(G_2)$ is $o(\Diam_{S_\ab}(G_\ab))$, it gives $\Diam_S(G)=(1+o(1))\Diam_{S_\ab}(G_\ab)$. The proof of the reduction theorem uses the finer commutator estimates directly.

\medskip
 The definition of the Cayley graph of a group $G$ with generating set $S$ naturally extends to quotient groups of $G$. For a normal subgroup $H\trianglelefteq G$, define
\[S_H\qrel{:=}\{Hs : s \in S\}.
\]
Then the Cayley graph of $G/H$, denoted by $\Cay(G/H, S_H)$, consists of vertex set $G/H$ and edge set $\{ \{Hg,Hgs\}: g\in G, s\in S\}$.

Let $\dist_S (\cdot, \cdot)$ denote the graph distance on $\Cay(G, S)$. Similarly, let $\dist_{S_H}(\cdot,\cdot)$ denote the
graph distance on $\Cay(G/H, S_H)$. For a subgroup $H$ of $G$, we define the diameter of $H$ with respect to the graph distance $\dist_S (\cdot, \cdot)$ on $\Cay(G, S)$ by
\begin{equation}\label{diam_def}
\Diam_S(H) \qrel{:=} \max\{\dist_S(\id,h) : h \in H\}.
\end{equation}
For $H\trianglelefteq H'\trianglelefteq G$ with $H\trianglelefteq G$, define
\[\Diam_S(H'/H)\qrel{:=}
\max\{\dist_{S_H}(H,Hh'):h'\in H'\}.
\]
Equivalently,
$
\Diam_S(H'/H)=\Diam_{S_H}(H'/H),
$
and we use the notation $\Diam_S(H'/H)$ throughout for brevity.

 We have the following  triangle inequality in terms of the diameter of a group $H'$ and that of its subgroup $H$ and the quotient group $H'/H$.
 \begin{lemma}\label{triangle_diam}
For all $H\trianglelefteq H' \trianglelefteq G$  such that $H\trianglelefteq G$ the following holds:
\begin{align}\Diam_S(H')&\qrel{\le}\Diam_S(H'/H) +\Diam_S(H).
\label{e:diameterdecomposition2}
\end{align}
\end{lemma}

\begin{proof}
Let $h' \in H'$. By the definition of $\operatorname{Diam}_S(H'/H)$, there exist $s_1,\dots,s_n\in S$ with $n\leq \Diam_S(H'/H)$ such that
$Hh'=Hs_1\cdots s_n$, i.e., there exists $h\in H$ such that $h'=hs_1\cdots s_n$. Hence
$$\dist_S(\id,h')\qrel{=}\dist_S(\id,hs_1\cdots s_n)\qrel{\leq} \dist_S(\id,h)+n\qrel{\leq} \Diam_S(H) + \Diam_S(H'/H),
$$
which concludes the proof of \eqref{e:diameterdecomposition2}.
\end{proof}

Iterating \eqref{e:diameterdecomposition2} along the lower central series yields
$$\Diam_S(G_2)\qrel{\le}\sum_{i=2}^\LL\Diam_S(G_i / G_{i+1}).$$
Breuillard and Tointon \cite[Lemma 4.11]{breuillard2016nilprogressions} proved the diameter of $G_2$ is at most $C_{|S|,L}\Diam_S(G)^{1/2}$, where $C_{|S|,L}$ is a constant depending on the size of $S$ and $L$. In fact, they showed, for all $i\in [L]$, that $\Diam_S(G_i/G_{i+1})$ is at most $C_{|S|,L}\Diam_S(G)^{1/i}$. El-Baz and Pagano \cite{el2021diameters} later proved the same estimates using somewhat similar arguments.

It was established in Diaconis and Saloff-Coste~\cite{diaconis1994moderate} that there is no cutoff for simple random walks on finite nilpotent groups when both $|S|$ and the nilpotency step are bounded. Therefore, we also consider generating sets $S$ with $|S| \to \infty$, and it becomes necessary to quantify the dependence of the constant $C_{|S|,L}$ on $|S|$ and $L$, which was not addressed in~\cite{breuillard2016nilprogressions} or~\cite{el2021diameters}. Our approach for upper bounding $\Diam_S(G_2)$ adheres to the framework in El-Baz and Pagano \cite{el2021diameters}, but with considerably more attention devoted to quantifying the influence of $|S|$ as well as $L$.

Fix a set $R_\ab\subseteq S_\ab$ of representatives of the inverse pairs, so $|R_\ab|=k$ and
\beq\label{def_representatives_ab}
|R_\ab\cap\{x,x^{-1}\}|\qrel{=}1
\qquad\text{for every }x\in S_\ab.
\eeq

\begin{thm}\label{diam_bound}
For $2\leq i\leq L$, we have
\begin{equation}\label{diam_quotient}
\Diam_S(G_i / G_{i+1}) \qrel{\leq} k^i \cdot 2^{i+7} \left( 2^{2i} m +k \left( D/k \right)^{1/i}\right).
\end{equation}
Consequently,
\begin{align*}\Diam_S(G_2)&\qrel{\le}
                \sum_{i=2}^\LL
        \nonumber \Diam_S(G_i / G_{i+1})\\
        &\qrel{\le}
                \sum_{i=2}^\LL
        k^{i+1}\cdot 2^{i+8}
        \left( 2^{2i}+\left( D/k\right)^{1/i}\right).
        \end{align*}
\end{thm}

Before we turn to the proof of Proposition \ref{prop:commutator_decomposition}
 and Theorem \ref{diam_bound}, some preliminary results that will be useful are presented in the following subsection.

\subsection{Preliminaries}

We begin by recalling some standard notation and stating several properties of commutators. For $x,y\in G$ we write $[x,y] := x^{-1}y^{-1}xy = [y,x]^{-1}$ and $x^y := y^{-1}xy = x[x,y]$. Further observe that for $x,y,z\in G$, $[x,yz]=[x,z][x,y]^{z}=[x,z][x,y][[x,y],z]$. Define
        $\rho(x,y) := [x,y]$
        for $x,y \in G$ as the two-fold commutator, and inductively
\[\rho(x_1, ..., x_i) \qrel{:=} [\rho(x_1, ..., x_{i-1}), x_i]
\Quad{for}
        i\geq3
\Quad{and}
        x_1, ..., x_i \in G.
\]

These standard properties of commutators can be found in
\cite{dummit2004abstract}; we omit the proofs.
The following is a fairly well-known result following from an induction argument using the three subgroup lemma.

\begin{prop}
\label{strongly_lower_central}
The lower central series of a nilpotent group $G$ is a strongly central series, i.e., $[G_i,G_j]$ is a subgroup of $G_{i+j}$  for all $i,j\geq 1$.
\end{prop}

\begin{prop}
\label{p:bilinearity}
For $i\geq1$, the commutator map
$\phi:G\times G_i\to G_{i+1}/G_{i+2}$,
$\phi(g,h):=G_{i+2}[g,h]$, is bilinear and satisfies the inversion identity
below. Explicitly, for $x,x_1,x_2\in G$, $y,y_1,y_2\in G_i$, and
$\ell,j\in\ZZ$, we have
\begin{align*}
G_{i+2}[x,y]&\qrel{=}G_{i+2}[y,x]^{-1},\\
G_{i+2}[x,y_1y_2]&\qrel{=}G_{i+2}[x,y_1][x,y_2],\\
G_{i+2}[x_1x_2,y]&\qrel{=}G_{i+2}[x_1,y][x_2,y],\\
G_{i+2}[x^\ell,y^j]&\qrel{=}G_{i+2}[x,y]^{\ell j}.
\end{align*}
Moreover,
for $i \ge 2$ and $j \in [i]$, if $x_1,\ldots,x_i \in G$ and $y\in G$, then we have the following linearity in the $j$-th component, i.e.,
\begin{gather}G_{i+1} \rho(x_1, ..., x_{j-1}, x_j y,x_{j+1}, ..., x_i)
\qrel{=}
        G_{i+1} \rho(x_1, ..., x_i) \widehat x_{y,j}
\qrel{=}
        G_{i+1} \widehat x_{y,j} \rho(x_1, ..., x_i),
\label{e:linearityinfirstcoord}
\end{gather}
where $\widehat x_{y,j} := \rho(x_1, ..., x_{j-1}, y, x_{j+1}, ..., x_i)$, and so
\beq\label{e:bylinearity}
G_{i+1}\rho(x_1,x_2,\dots,x_i)\qrel{=}G_{i+1}\rho(y_1,y_2,\dots,y_i) \quad \text{ if } x_jy_j^{-1}\in G_2 \text{ for all $j\in[i]$}.
\eeq
\end{prop}

Let $R \subseteq S$ be such that $|\{s,s^{-1}\} \cap R| = 1$ for all $s \in S$. Let $S_1:=R$ and define,
for $i\geq 2$,
\beq\label{eq:def_S_i}
S_i\qrel{=}\{ \rho(x_1,x_2,\dots,x_i): x_j\in R \text{ for $j\in [i]$}\}.
\eeq

Throughout the paper, for $i\ge 1$, we use the notation
\[G_{i+1}S_i\qrel{:=}\{G_{i+1}s:s\in S_i\}
\]
to denote the image of $S_i$ in the quotient group $G_i/G_{i+1}$. Thus,
elements of $G_{i+1}S_i$ are understood as elements of $G_i/G_{i+1}$.

The following proposition can be proved by induction on $i$ using Proposition \ref{p:bilinearity}.
For a proof, see \cite{el2021diameters}.

\begin{prop}
\label{p:generatingG}
More generally, if $S$ is any symmetric subset and $G_2S_1$ generates $G_{\mathrm{ab}}$, then $G_{i+1}S_i$ generates the Abelian group $G_i/G_{i+1}$ for all $i\geq 1$. In particular, $S$ generates $G$ if and only if $G_2 S_1$ generates $G_\mathrm{ab}$.
\end{prop}

By Proposition~\ref{p:generatingG}, $G_{i+1}S_i$ generates $G_i/G_{i+1}$, so we may choose a subset
$\widetilde S_i\subseteq S_i$ such that
\beq\label{quotient_min_gen}
G_{i+1}\widetilde S_i
\text{ has minimum cardinality among all subsets of $G_{i+1}S_i$ that generate $G_i/G_{i+1}$.}
\eeq
By retaining only one representative from each coset, we may further assume that
$|\widetilde S_i|=|G_{i+1}\widetilde S_i|$, so that distinct elements of
$\widetilde S_i$ represent distinct cosets modulo $G_{i+1}$.

Recall from \eqref{def_representatives_ab} that $R_\ab$ contains one representative from each inverse pair in $S_\ab=G_2S$.
Our subsequent estimate on the size of the Abelianization $G_\ab$ depends explicitly on $|\widetilde S_i|$, which is why we keep track of this set separately. Under additional structural assumptions on $G$, one may be able to obtain upper bounds on $|\widetilde S_i|$ that improve upon the elementary estimate $|\widetilde S_i|\leq k^i$ proved below. In the general setting considered here, however, we will only use this elementary bound.

\begin{lemma}\label{ab_size}
For $1\leq i\leq L$, the sets $\widetilde S_i$ from \eqref{quotient_min_gen} satisfy
\[|\widetilde S_i|\qrel{\leq} k^i
\qquad\text{and}\qquad
|G_i/G_{i+1}|
\qrel{\leq}
|G_\ab|^{|\widetilde S_i|}.
\]
Consequently, if $S$ is chosen so that $|R|=r(G)$, then
\[|G|
\qrel{\leq}
|G_\ab|^{2r(G)^L}.
\]
\end{lemma}

\begin{proof}
The first two assertions are immediate for $i=1$. For these assertions, assume henceforth that $i\geq2$.
By Proposition~\ref{p:bilinearity}, the map
\[(G_\ab)^i\to G_i/G_{i+1},
\qquad
(G_2x_1,\ldots,G_2x_i)
\mapsto
G_{i+1}\rho(x_1,\ldots,x_i),
\]
is well defined and multilinear. Hence, up to inversion, every element of
$G_{i+1}S_i$ is represented by an $i$-tuple in $R_\ab^i$. Therefore
$G_i/G_{i+1}$ admits a generating set contained in $G_{i+1}S_i$ of
cardinality at most $k^i$, and consequently
$
|\widetilde S_i|\leq k^i.
$

\medskip
Write
$
G_{i+1}\widetilde S_i
=
\{G_{i+1}y_{i,j}:j\in[|\widetilde S_i|]\},
$
where
$y_{i,j}=\rho(s_{j_1},\ldots,s_{j_i})$
for some $(s_{j_1},\ldots,s_{j_i})\in R^i$.
Let $d_1$ denote the order of $G_2s_{j_1}$ in $G_{\ab}$.
By Proposition~\ref{p:bilinearity},
\[G_{i+1}y_{i,j}^{d_1}
\qrel{=}
G_{i+1}\rho(s_{j_1}^{d_1},s_{j_2},\ldots,s_{j_i})
\qrel{=}
G_{i+1},
\]
since $s_{j_1}^{d_1}\in G_2$. Therefore,
$
\operatorname{ord}_{G_i/G_{i+1}}(G_{i+1}y_{i,j})
\leq d_1\leq |G_{\ab}|
$ for each $1\leq j \leq |\widetilde S_i|$.
Since $G_{i+1}\widetilde S_i$ generates the Abelian group
$G_i/G_{i+1}$, it follows that
\[|G_i/G_{i+1}|
\qrel{\leq}
\prod_{j=1}^{|\widetilde S_i|}
\operatorname{ord}_{G_i/G_{i+1}}(G_{i+1}y_{i,j})
\qrel{\leq}
|G_{\ab}|^{|\widetilde S_i|}.
\]
Taking $S$ such that $|R|=r=:r(G)$, it then follows from $|\widetilde S_i|\leq r^i$ that
$$|G|\qrel{=}\prod_{i=1}^L |G_i/G_{i+1}|\qrel{\leq} |G_{\mathrm{ab}}|^{\sum_{i=1}^L r^i}\qrel{\leq} |G_\ab|^{2r^L}.
$$

\end{proof}

Lastly, we record lower bounds on the word and Euclidean diameters of the
Abelianization. Recall from Definition \ref{def:weighted_Cayley_graph} that $D:=\Diam(\Gamma_\ab)$ and from \eqref{eq:euclidean_diameter} that  $D_2:=D_{\ell^2}(\Gamma_\ab)$.
\begin{lemma}\label{diam_abelianization}
We have
\[
D\qrel{\gtrsim}|G_\ab|^{1/k}.
\]
Moreover, for every fixed $C>0$, if $k\leq C\log |G_\ab|$, then
\[
D\qrel{\gtrsim_C}k|G_\ab|^{1/k}
\qquad\text{and}\qquad
D_2\qrel{\gtrsim_C}\sqrt k\,|G_\ab|^{1/k}.
\]
\end{lemma}

\begin{proof}
Writing $G_\ab$ additively and $R_\ab=\{a_1,\ldots,a_k\}$, the map
$(z_1,\ldots,z_k)\mapsto\sum_{j=1}^kz_ja_j$ from $\ZZ^k$ to $G_\ab$
maps $B_k(D):=\{\bm z\in\ZZ^k:\|\bm z\|_1\leq D\}$ onto $G_\ab$. By
\cite[Lemma~E.2a]{hermon2018supplementary},
$
|G_\ab|\leq |B_k(D)|
\leq 2^{k\wedge D}\binom{D+k}{k}.
$
The bound $|B_k(D)|\leq(2D+1)^k$ immediately yields
$D\gtrsim |G_\ab|^{1/k}$.
The stronger bound under $k\leq C\log |G_\ab|$
follows similarly by applying
$\binom{n}{m}\leq(en/m)^m$
and treating the cases $D\geq k$ and $D<k$ separately.
Finally, the lower bound on $D_2$ follows from
$D\leq\sqrt{k}\,D_2$, which is a consequence
of $\|\bm z\|_1\leq\sqrt{k}\|\bm z\|_2$.
\end{proof}

\subsection{Quantitative graph-distance estimates}

Recall that $R\subseteq S$ satisfies
$
|\{s,s^{-1}\}\cap R|=1
$ for every $s\in S$.
Also fix
$\hat S_\ab\subseteq S$ containing one representative of each element of
$S_\ab$, so that $x\mapsto G_2x$ is a bijection from $\hat S_\ab$ onto
$S_\ab$.

\begin{prop}\label{prop:commutator_decomposition}
Fix $2\leq i\leq L$. For any $(h,x_2,\ldots,x_i)\in G\times R^{i-1}$, there exist
$N\in\mathbb N$ and elements
$u_{\alpha,j}=s_{\alpha,j}^{d_{\alpha,j}}$, where
$s_{\alpha,j}\in S$ and $d_{\alpha,j}\in\ZZ$, such that
\[G_{i+1}\rho(h,x_2,\ldots,x_i)
\qrel{=}
G_{i+1}\prod_{\alpha=1}^N
\rho(u_{\alpha,1},\ldots,u_{\alpha,i}).
\]
Moreover, the decomposition above can be chosen so that, upon expanding the
powers and commutators in the product on the right-hand side,
\[\prod_{\alpha=1}^N
\rho(u_{\alpha,1},\ldots,u_{\alpha,i})
\qrel{=}
z_1\cdots z_M,
\qquad z_1,\ldots,z_M\in S,
\]
where $M$ satisfies the following bounds:
\begin{enumerate}
\item[(i)]
If $h=s^b$ for some $s\in S$ and
$b\geq 1$, then
$
M\leq
2^{i+7}\left(2^{2i}+b^{1/i}\right).
$

\item[(ii)]
For each $h\in G$, there exists a decomposition of the form
\[G_{i+1}\rho(h,x_2,\dots,x_i)
\qrel{=}
G_{i+1}\prod_{x\in \hat S_\ab}
\rho\left(
x^{\ell_{(h,x_2,\dots,x_i)}(G_2x)},
x_2,\dots,x_i
\right),
\]
where $\ell_{(h,x_2,\dots,x_i)}\in \mathfrak F_\ab$, with
\[\mathfrak F_\ab
\qrel{:=}
\left\{
\ell:S_\ab\to\ZZ_{\geq 0}:
\sum_{y\in S_\ab}\ell(y)\leq D,\
\ell(y)\ell(y^{-1})\crel{=}0
\text{ whenever }y\neq y^{-1}
\right\}.
\]
After applying part~(i) to each nontrivial factor in this last display, the resulting product of commutators has an expansion of length $M$ satisfying
\beq\label{eq:decomposition_commutators}
M\qrel{\leq}
2^{i+7}
\left(
2^{2i}m
+
k
\left(
\frac{D}{k}
\right)^{1/i}
\right).
\eeq
\end{enumerate}
\end{prop}

\begin{proof}[Proof of Proposition \ref{prop:commutator_decomposition}~(i)] For an iterated commutator of the form
\[\rho(x_1^{n_1},\ldots,x_i^{n_i}),
\qquad x_1,\ldots,x_i\in S\quad\text{and}\quad
n_1,\ldots,n_i\in\ZZ,
\]
let
$
|\rho(x_1^{n_1},\ldots,x_i^{n_i})|
$
denote the number of elements of $S$ appearing after fully expanding the powers and commutators, without cancellation. Note this quantity is an upper bound on the corresponding graph distance.

The following simple estimate will play a major role in our proof. For $1\leq j\leq i$, $(x_1, ..., x_i) \in S \times R^{i-1}$ and $n \in \mathbb{N}$,
we have
\begin{equation}|\rho( x_1^{n},x_2^n, ..., x_j^n,x_{j+1}, ..., x_i)|
\qrel{\le}
        2^{i+2}n.
\label{e:bilinearity3}
\end{equation}
This follows by induction using the recursive estimate
$$|\rho(x_1, ..., x_i)|
\qrel{\le}
        2(|x_i|+|\rho(x_1, ..., x_{i-1})|),
$$
which is a simple consequence of the definition $\rho(x_1,\dots,x_i)=[\rho(x_1,\dots,x_{i-1}),x_i]$.

If $b=n^i$ for some $n\in\mathbb N$, then
\eqref{e:linearityinfirstcoord} gives
\[G_{i+1}\rho(s^b,x_2,\ldots,x_i)
\qrel{=}
G_{i+1}\rho(s^n,x_2^n,\ldots,x_i^n),
\]
whose expanded length is at most $2^{i+2}n=2^{i+2}b^{1/i}$ by
\eqref{e:bilinearity3}. For general $b\in\mathbb{N}$, we seek to decompose $b$ into a sum of powers $n^j$, with $j=i,i-1,\dots,2$, so as to write
$G_{i+1}\rho(s^b,x_2,\ldots,x_i)$ as a product of terms of the form
\[G_{i+1}\rho(s^n,x_2^n,\ldots,x_j^n,x_{j+1},\ldots,x_i),
\qquad j\in[i],\ n\in\mathbb N.
\]

To decompose $G_{i+1}\rho(s^b,x_2,\ldots,x_i)$,
we use a greedy procedure to express $b$ as a sum
of integer powers. Let $E$ denote the remaining
part of $b$, initially $E=b$.
For $j=i,i-1,\ldots,2$, we repeatedly subtract
$j$-th powers from $E$ as follows:
whenever $E\geq 4^{j^2}$, we choose
$a=\lfloor E^{1/j}\rfloor$ and replace $E$
with $E-a^j$. We refer to $a$ as the root
selected at this step, and associate with it
the commutator
\[
\rho(s^a,x_2^a,\ldots,x_j^a,x_{j+1},\ldots,x_i).
\]
After completing all stages, if $E>0$, we append
$\rho(s^E,x_2,\ldots,x_i)$.
By multilinearity modulo $G_{i+1}$, the product
of these commutators represents
$\rho(s^b,x_2,\ldots,x_i)$ modulo $G_{i+1}$.

\begingroup\emergencystretch=3em
To bound the total expanded length, we first estimate, at each stage $j$,
the sum of $\bigl|\rho(s^a,x_2^a,\ldots,x_j^a,x_{j+1},\ldots,x_i)\bigr|$
over the roots $a$ selected by this procedure, and then sum these bounds over all stages. Consider a step at stage $j$. Since
$E\geq4^{j^2}$, the selected root
$a=\lfloor E^{1/j}\rfloor$ satisfies $a\geq4^j$.
Moreover, since $E<(a+1)^j$, the new remainder
$E':=E-a^j$ satisfies
\[
0\qrel{\leq} E'<(a+1)^j-a^j
\qrel{=}\sum_{h=1}^j\binom{j}{h}a^{j-h}
\qrel{\leq} 2^j a^{j-1}
\]
\endgroup
for $a\geq1$.
If another step of the same stage follows, its root
$a_{\mathrm{next}}=\lfloor(E')^{1/j}\rfloor$
therefore satisfies
\[
a_{\mathrm{next}}
\qrel{\leq} (E')^{1/j}
\qrel{\leq} 2a^{1-1/j}
\qrel{\leq} \frac a2,
\]
where the last inequality follows from $a^{1/j}\geq4$. Thus, within each stage, the roots decrease by at least
a factor of two, so their sum is at most twice the
first root. In particular, the sum of the roots
at stage $i$ is at most $2b^{1/i}$.
For each subsequent stage $j<i$, the remainder at
the beginning of the stage is less than
$4^{(j+1)^2}$, since stage $j+1$ has terminated.
Hence the sum of the roots at stage $j\geq 2$ is at most
\[
2\cdot4^{(j+1)^2/j}
\qrel{=}2\cdot4^{j+2+1/j}
\qrel{\leq}64\cdot4^j.
\]
The final remainder is less than $4^4=256$. By \eqref{e:bilinearity3}, the fully expanded length therefore satisfies
\begin{align*}M&\qrel{\le}2^{i+2}\left(2b^{1/i}+64\sum_{j=2}^{i-1}4^j+256\right)\\
&\qrel{\le}2^{i+2}\left(2b^{1/i}+\frac{64}{3}4^i\right)
\qrel{\le}2^{i+7}\left(2^{2i}+b^{1/i}\right).
\end{align*}
Here the sum is empty when $i=2$. This proves part~(i).
\end{proof}

\begin{proof}[Proof of Proposition \ref{prop:commutator_decomposition}~(ii)]
By \eqref{e:bylinearity}, for any $h'\in G$ such that
$G_2h'=G_2h$,
\beq\label{Gh}
G_{i+1}\rho(h,x_2,\dots,x_i)
\qrel{=}
G_{i+1}\rho(h',x_2,\dots,x_i).
\eeq
Thus the quantity in \eqref{Gh} depends on $h$ only through its image
$G_2h\in G_\ab$.

Since $S_\ab$ is a symmetric generating set of $G_\ab$, there exists some function
$\ell:S_\ab\to\ZZ_{\geq0}$ satisfying $\sum_{G_2x\in S_\ab}\ell(G_2x)\leq D$ such that
\[G_2\prod_{x\in \hat S_\ab}x^{\ell(G_2x)}
\qrel{=}
G_2h.
\]
Indeed, take a shortest word in $S_\ab$ representing $G_2h$. Its length is
at most
$
D.
$
Since $G_\ab$ is Abelian, we may collect identical generators, and any
occurrences of both $G_2x$ and $G_2x^{-1}$ may be cancelled. Thus $\ell$
may be chosen so that
$
\ell(G_2x)\ell(G_2x^{-1})=0
$
whenever $G_2x\neq G_2x^{-1}$.

It follows from \eqref{Gh} and Proposition~\ref{p:bilinearity} that
\[G_{i+1}\rho(h,x_2,\dots,x_i)
\qrel{=}
G_{i+1}\prod_{x\in \hat S_\ab}
\rho\left(x^{\ell(G_2x)},x_2,\dots,x_i\right).
\]

For each $x\in\hat S_\ab$ with $\ell(G_2x)>0$, let $w_x$ be the product
of commutators supplied by part~(i), so that
\[
G_{i+1}w_x
\qrel{=}G_{i+1}\rho\bigl(x^{\ell(G_2x)},x_2,\ldots,x_i\bigr)
\quad\text{and}\quad
|w_x|\leq2^{i+7}\bigl(2^{2i}+\ell(G_2x)^{1/i}\bigr),
\]
where $|w_x|$ denotes the fully expanded length of this chosen product.
Thus the total expanded length satisfies
\[
\sum_{x\in\hat S_\ab:\ell(G_2x)>0}|w_x|
\qrel{\leq}2^{i+7}\left(2^{2i}m+
\sum_{G_2x\in S_\ab}\ell(G_2x)^{1/i}\right).
\]

By the definition of $\mathfrak F_\ab$, at most $k$ terms in the sum $\sum_{G_2x\in S_\ab}\ell(G_2x)^{1/i}$
 are nonzero. List the nonzero values of $\ell(G_2x)$, padding with zeros
if necessary to obtain $a_1,\ldots,a_k\geq0$. Then
$
\sum_{j=1}^{k} a_j
=
\sum_{G_2x\in S_\ab}\ell(G_2x)
\leq D.
$
Since $t\mapsto t^{1/i}$ is concave, Jensen's inequality gives
\[\frac{1}{k}
\sum_{j=1}^{k}a_j^{1/i}
\qrel{\leq}
\left(
\frac{1}{k}
\sum_{j=1}^{k}a_j
\right)^{1/i}
\qrel{\leq}
\left(\frac{D}{k}\right)^{1/i}.
\]
Therefore,
\[\max_{\ell\in \mathfrak F_\ab}
\sum_{G_2x\in S_\ab}\ell(G_2x)^{1/i}
\qrel{\leq}
k
\left(\frac{D}{k}\right)^{1/i}.
\]
Substituting this into the bound for $\sum_x|w_x|$
gives \eqref{eq:decomposition_commutators}.
\end{proof}

We are now ready to complete the proof of Theorem~\ref{diam_bound}.

\begin{proof}[Proof of Theorem \ref{diam_bound}]
The inequality $\Diam_S(G_2)\le\sum_{i=2}^\LL \Diam_S(G_i / G_{i+1})$  follows from inductively applying \eqref{e:diameterdecomposition2} with $H=G_{i+1}$ and $H'=G_i$ for $ 2\leq i\leq \LL$. It remains to prove \eqref{diam_quotient}.

Recall the definition of $\widetilde S_i$ from \eqref{quotient_min_gen}. We  write
$$G_{i+1}\widetilde S_i\qrel{=}\{G_{i+1}y_{i,j}: j\in [|\widetilde S_i|]\}$$ where each $y_{i,j}=\rho(s_{j_1}, s_{j_2},\dots,s_{j_i})$ for some $(s_{j_1},\dots,s_{j_i})\in R^i$. As $G_{i+1}\widetilde S_i$ generates $G_i/G_{i+1}$, for any $g\in G_i$, we can write
$$G_{i+1}g\qrel{=}
        G_{i+1}
        \prod_{j\leq| \widetilde S_i|}y_{i,j}^{n_{g,j}}
        $$
for some $(n_{g,j})_{j\leq |\widetilde S_i|}\subseteq \ZZ$. For each $y_{i,j}\in \widetilde S_i$, observe that
$$G_{i+1}y_{i,j}^{n_{g,j}}\qrel{=}G_{i+1}\rho(s_{j_1}^{n_{g,j}},s_{j_2},\dots,s_{j_i})\qrel{=}G_{i+1}\rho(h,s_{j_2},\dots,s_{j_i})$$
for any $h\in G$ such that $G_2h=G_2s_{j_1}^{n_{g,j}}$. Recall from Lemma \ref{ab_size} that $|\widetilde S_i|\leq k^i$. Applying  \eqref{eq:decomposition_commutators} with the triangle inequality immediately gives
\begin{align*}\dist_{S_{G_{i+1}}}(G_{i+1}, G_{i+1}g)&\qrel{\leq} \sum_{j\leq| \widetilde S_i|} \dist_{S_{G_{i+1}}}(G_{i+1}, G_{i+1}y_{i,j}^{n_{g,j}})\\
&\qrel{\leq} k^i \cdot 2^{i+7} \left( 2^{2i} m +k \left( D/k \right)^{1/i}\right),
\end{align*}
which then leads to \eqref{diam_quotient}.
\end{proof}

\section{Reduction to Abelianization---Proofs of Theorems \ref{thm: mixing and spectral comparison with Abelianization} and \ref{thm:unconditional_relaxation_bounds}}\label{reduction_abelianization}

In this section, we prove our reduction-to-Abelianization
theorems by combining the quantitative word-length bounds
for iterated commutators established in
Section~\ref{sec:geometry_Cayley}
with a comparison argument that exploits a series of normal
subgroups of $G$.
The underlying comparison framework is illustrated in
Fact~\ref{lem:normal_subgroup_comparison}, while
Proposition~\ref{interpolation_subgroups} incorporates
the finer estimates needed for the proofs.

We use the standing notation without assuming that $\mu$ is symmetric. The definitions and preliminary results below apply to arbitrary finite groups.
For any nonempty $H\subseteq G$, let $\pi_H$ denote the uniform probability measure on
$H$. In particular, write $\pi:=\pi_G$, and let $\id$ denote the identity
element of $G$.

\begin{definition}[Mixing and relaxation times]
        \label{def:TV and ell p mixing times}
For a function $f:G\to\mathbb C$, write
\[\|f\|_p\qrel{:=}\left(\sum_{x\in G}\pi(x)|f(x)|^p\right)^{1/p}\ (1\leq p<\infty)\qquad\text{and}\qquad
\|f\|_\infty\qrel{:=}\max_{x\in G}|f(x)|.
\]
For a signed measure $\sigma$ on $G$, set $\|\sigma\|_{p,\pi}:=\|\sigma/\pi\|_p$. Thus $\|\sigma\|_{\TV}:=\tfrac12\|\sigma\|_{1,\pi}=\tfrac12\sum_x|\sigma(x)|$.

\begin{enumerate}[label=(\alph*)]
\item The $\varepsilon$-$\ell^p$ \emph{mixing time} of the rate 1 random walk $X:=(X_t)_{t \ge 0}$ on $\Gamma(\mu)$ is defined to be
\[t_{\mathrm{mix}}^{\ell^p}(\mu,\varepsilon)\qrel{:=}\inf\left\{t:\left\|\mathbb{P}_{\id}(X_t = \cdot) - \pi\right\|_{p,\pi} \le \varepsilon \right\}, \]
where $\mathbb{P}_\id$ denotes the law of $X$ when $X_0=\id$. We write $t_{\mix}^{\ell^p}(\mu):=t_{\mix}^{\ell^p}(\mu,1/2)$ when no threshold is specified.

\item The $\ep$\emph{-total variation (TV) mixing time}, given by
$$t_{\mathrm{mix}}^{\mathrm{TV}}(\mu,\varepsilon)\qrel{:=}\inf\{t:\|\mathbb{P}_{\id}(X_t = \cdot) - \pi\|_{\mathrm{TV}} \le \varepsilon \},$$ satisfies $t_{\mathrm{mix}}^{\mathrm{TV}}(\mu,\varepsilon)=t_{\mathrm{mix}}^{\ell^{1}}(\mu,2\varepsilon)$. When $\varepsilon=1/4$ we omit it from the notation, i.e., write $t_{\mathrm{mix}}^{\mathrm{TV}}(\mu)=t_{\mathrm{mix}}^{\mathrm{TV}}(\mu,1/4)$.

\item For $\varepsilon \in (0,1]$ the $\varepsilon$-separation mixing time of the random walk on $\Gamma(\mu)$ is given by \[t_{\mathrm{sep}}(\mu,\varepsilon)\qrel{:=}\inf\{t:\min_{g \in G}|G|\cdot\mathbb{P}_{\id}(X_t = g) \ge 1-  \varepsilon \}.\]

\item
Let $P^\sym$ denote the transition matrix of the random walk driven by
$\mu_\sym$. We define the \emph{spectral gap} of $\mu$ by
\[\gamma(\mu)\qrel{:=}\gamma(\mu_\sym),
\]
where $\gamma(\mu_\sym)$ is the second smallest eigenvalue of
$I-P^\sym$, and define the \emph{relaxation time} by
\[t_{\rel}(\mu)\qrel{:=}t_{\rel}(\mu_\sym)\qrel{:=}\frac{1}{\gamma(\mu_\sym)}.
\]
When $\mu$ is symmetric,
$\mu=\mu_\sym$, recovering the usual definition.
\end{enumerate}
\end{definition}

\subsection{Non-reversible relaxation times}
\label{sec:nonreversible_relaxation}
For a generating probability measure $\nu$ on a finite group $H$, the eigenvalue relaxation time introduced in Remark~\ref{rmk:nonreversible_relaxation} can be written as
\[
t^{\ev}_{\rel}(\nu)\qrel{=}
\left(1-\max\{\Re(\lambda):\lambda\ne1\text{ is an eigenvalue of }P_\nu\}\right)^{-1}.
\]
This notion is natural because, if $d_{\TV}(t)$ denotes the total variation
distance to stationarity of the walk driven by $\nu$, then
$t^{\ev}_{\rel}(\nu)$ characterizes its asymptotic exponential decay rate:
\beq\label{eq:decay_rate_TV}
\lim_{t\to\infty}-\frac{\log d_{\TV}(t)}{t}
\qrel{=}\frac{1}{t_{\rel}^{\ev}(\nu)}\qquad\text{and}\qquad
2d_{\TV}(t)\qrel{\geq} e^{-t/t_{\rel}^{\ev}(\nu)}
\ (t\geq0).
\eeq
For the decay-rate identity, cf.\ \cite[Definition~1.3 and Exercise~1.3]{salez2025modernaspectsmarkovchains}; for the inequality, cf.\ \cite[Theorem~12.5 and Lemma~20.11]{wilmer2009markov}.

Moreover, let $(H_n)_{n\in\mathbb N}$ be a sequence of finite groups and,
for each $n$, let $X^{(n)}$ be the continuous-time random walk on $H_n$
driven by a probability measure $\nu^{(n)}$. By
\cite[Lemma~1.2]{salez2025modernaspectsmarkovchains}, if
$(X^{(n)})_{n\in\mathbb N}$ exhibits total variation cutoff, then necessarily
\beq\label{eq:necessity_product_condition}
\frac{t_{\mix}^{\TV}(\nu^{(n)})}
{t_{\rel}^{\ev}(\nu^{(n)})}
\to\infty.
\eeq

We now turn to comparing these two notions of relaxation time. Let $\nu_\sym$ denote the symmetrization of $\nu$, and write
$
P^{\sym}_{\nu}:=\frac{P_\nu+P_\nu^*}{2}.
$
By the variational characterization of the spectral gap,
\[\gamma(\nu)\qrel{=}\gamma(\nu_\sym)
\qrel{=}
\min_{f\perp \mathbf 1,\ f\neq 0}
\frac{\langle (I-P^{\sym}_{\nu})f,f\rangle}{\langle f,f\rangle}.
\]
If $P_\nu v=\lambda v$ with $\lambda\neq 1$, then $v\perp\mathbf 1$, and hence
\[\gamma(\nu)
\qrel{\leq}
\frac{\langle (I-P^{\sym}_{\nu})v,v\rangle}{\langle v,v\rangle}
\qrel{=}
1-\Re(\lambda).
\]
It follows that $\gamma(\nu)\leq1-\max_{\lambda\neq 1}\Re(\lambda)$,
and therefore
\beq\label{eq:inequality_two_relaxation_times}
t_{\rel}(\nu)\qrel{\geq} t^{\ev}_{\rel}(\nu).
\eeq
In general, the two notions need not coincide and can even be of different
orders for non-reversible Markov chains; see, e.g., \cite{Fill1991} and
\cite[Section~11.5]{wilmer2009markov}.

When $P_\nu$ is normal, however, $P_\nu v=\lambda v$ implies
$P_\nu^*v=\overline{\lambda}v$, and hence
$
P^{\sym}_{\nu}v=\Re(\lambda)v.
$
It follows that
$
\gamma(\nu)
=
1-\max_{\lambda\neq 1}\Re(\lambda),
$
and therefore
\beq\label{eq:equality_relaxation_times}
t_{\rel}(\nu)\qrel{=}t^{\ev}_{\rel}(\nu).
\eeq
A sufficient condition for $P_\nu$ to be normal is that $\nu$ be
conjugacy-invariant, i.e., $\nu(y^{-1}xy)=\nu(x)$ for all $x,y\in H$.
In particular, every random walk on an Abelian group has a normal transition
kernel. Moreover, if $\nu$ is symmetric, then $P_\nu$ is self-adjoint, and
hence normal.

We now return to our setting, with $X$ denoting the random walk on $G$
driven by $\mu$. Since $G_{\ab}$ is Abelian, the projected walk satisfies
$
t_{\rel}(\mu_{\ab})=t^{\ev}_{\rel}(\mu_{\ab}).
$
Moreover, since the walk on $G_{\ab}$ is a projection of the walk on $G$,
together with \eqref{eq:inequality_two_relaxation_times} we have
\begin{equation}\label{eq:relaxation_projection_chain}
t_{\rel}(\mu_{\ab})\qrel{=}t^{\ev}_{\rel}(\mu_{\ab})
\qrel{\leq}t^{\ev}_{\rel}(\mu)\qrel{\leq}t_{\rel}(\mu).
\end{equation}
If \eqref{e:gap given by Abelianization} holds, the two ends of this chain
coincide, so all four quantities are equal. This proves the equality of the
two notions of relaxation time asserted in
Remark~\ref{rmk:nonreversible_relaxation}, without any normality assumption
on $P_\mu$.

\subsection{Comparison principle}
We develop a comparison argument, similar in spirit to
the Dirichlet-form comparison method of Diaconis and
Saloff-Coste \cite{DSCcomparison}, that exploits the
structure of normal subgroups of $G$. The main result is Proposition~\ref{interpolation_subgroups}.
We begin with a preliminary result concerning the invariant subspaces associated with normal subgroups of $G$.

\begin{lemma}[Invariant subspaces associated with normal subgroups]
\label{lem:normal_invariant_subspaces}
Let $H\trianglelefteq G$ and define the averaging operator
\[
A_Hf(g)\qrel{:=}\frac1{|H|}\sum_{h\in H}f(gh), \qquad g\in G.
\]
Then $A_H$ commutes with $P_\mu$. In particular, the spaces
\[
\begin{aligned}
\mathcal I_H&\qrel{:=}\operatorname{ran}(A_H)\qrel{=}\{f:G\to\mathbb C:
f(gh)=f(g)\text{ for all }g\in G,\ h\in H\},\\
\mathcal Z_H&\qrel{:=}\operatorname{ker}(A_H)\qrel{=}\left\{f:G\to\mathbb C:
\sum_{h\in H}f(gh)=0\text{ for all }g\in G\right\},
\end{aligned}
\]
are invariant under $P_\mu$ and its continuous-time semigroup.
\end{lemma}

\begin{proof}
Since $H$ is normal in $G$, we have
\[
A_HP_\mu f(g)
\qrel{=}\sum_s\mu(s)\frac1{|H|}\sum_{h\in H}f(ghs)\\
\qrel{=}\sum_s\mu(s)\frac1{|H|}\sum_{h\in H}f(gsh)
\qrel{=}P_\mu A_Hf(g),
\]
where the second equality uses $s^{-1}Hs=H$.
Since $A_H$ commutes with $P_\mu$, its range
$\mathcal I_H$ and kernel $\mathcal Z_H$ are
invariant under $P_\mu$ and its continuous-time
semigroup.

\end{proof}

For $H\trianglelefteq G$, the equality $gH=Hg$ shows that left
$H$-invariance is equivalent to right $H$-invariance. We therefore refer
to functions in $\mathcal I_H$ simply as $H$-invariant.

To illustrate the comparison principle underlying the proofs of
Theorems~\ref{thm: mixing and spectral comparison with Abelianization}
and~\ref{thm:unconditional_relaxation_bounds},
we first present in Fact \ref{lem:normal_subgroup_comparison} a simpler version of the argument,
without the technical refinements required in
Proposition~\ref{interpolation_subgroups}.
The following fact is included for motivation and is not used
in the subsequent proofs. The case $H=G_2$ is particularly
helpful to keep in mind.

For a probability measure $\nu$ on $G$ and a real-valued function $f$ on $G$, write
\[
\mathcal E_\nu(f,f)\qrel{=}\mathcal E_{P_\nu}(f,f)
\qrel{:=}\langle f,(I-P_\nu)f\rangle_\pi
\qrel{=}\frac12\sum_{s\in G}\nu(s)\|f(\cdot s)-f\|_2^2.
\]
We use the same inner-product definition of $\mathcal E_P(f,f)$ for any
doubly stochastic transition matrix $P$ on $G$.

\begin{fact}[Comparison argument framework]
\label{lem:normal_subgroup_comparison}
Let $H$ be a nontrivial normal subgroup of $G$, and let
$\mathcal Z_H$ be as in
Lemma~\ref{lem:normal_invariant_subspaces}. Set
\[
\mu_{\min}\qrel{:=}\min_{s\in S}\mu(s),
\qquad
D(H)\qrel{:=}\Diam_{S_\sym}(H),
\]
where the diameter is measured in the ambient Cayley graph,
as in \eqref{diam_def}.
Define
$
f(g):=
\frac{\delta_{\id}(g)-\pi_H(g)}{\pi(g)}
$ for $g\in G$. Then $f\in\mathcal Z_H$.

Let $d_G^{\TV}(t)$ and $d_{G/H}^{\TV}(t)$ denote
the total variation distances from stationarity
of the walks on $G$ and $G/H$, respectively,
both started at the identity.
By projection, the triangle inequality, and Cauchy--Schwarz,
\[
\begin{aligned}
0\qrel{\leq}d_G^{\TV}(t)-d_{G/H}^{\TV}(t)
&\qrel{\leq}\bigl\|\P_{\id}(X_t=\cdot)-\P_{\pi_H}(X_t=\cdot)\bigr\|_{\TV}\\
&\qrel{=}\tfrac12\|P_tf\|_1\qrel{\leq}\tfrac12\|P_tf\|_2.
\end{aligned}
\]

By the standard path comparison for Dirichlet forms of two walks on the same group
\cite{DSCcomparison}, using paths of length at most $D(H)$,
\[
\|f\|_2^2\qrel{=}\mathcal E_{\pi_H}(f,f)
\qrel{\leq}\frac{D(H)^2}{\mu_{\min}}\mathcal E_{\mu_\sym}(f,f),
\qquad f\in\mathcal Z_H.
\]
Here the equality uses $f\in\mathcal Z_H$, since $P_{\pi_H}f=A_Hf=0$.
By Lemma~\ref{lem:normal_invariant_subspaces},
$P_tf\in\mathcal Z_H$ for all $t\geq0$. Thus,
\[
\frac{\mathrm d}{\mathrm dt}\|P_tf\|_2^2
\qrel{=}-2\mathcal E_{\mu_\sym}(P_tf,P_tf)
\qrel{\leq}-\frac{2\mu_{\min}}{D(H)^2}\|P_tf\|_2^2.
\]
Consequently,
\[
\|P_tf\|_2\qrel{\leq}e^{-\mu_{\min}t/D(H)^2}\|f\|_2.
\]
The same comparison also gives
\[
t_{\rel}(\mu)
\qrel{\leq}
\max\left\{
t_{\rel}(\mu_{G/H}),
\frac{D(H)^2}{\mu_{\min}}
\right\},
\]
where $\mu_{G/H}$ is the projection of $\mu$
onto $G/H$. Indeed, $L^2(G)$ decomposes into
the orthogonal invariant subspaces
$\mathcal I_H$ and $\mathcal Z_H$, and the
restriction to $\mathcal I_H$ corresponds
to the projected walk on $G/H$.\end{fact}

The Dirichlet-form comparison method of Diaconis and Saloff-Coste
\cite{DSCcomparison} compares the forms associated with two random-walk
kernels $P_\mu$ and $P_\nu$ on the same group $G$; see also
\cite[Corollary~13.24 and Remark~13.25]{wilmer2009markov}.
A highly unusual feature of our application is the choice $\nu=\pi_H$,
for which $P_\nu=P_{\pi_H}$ is reducible unless $H=G$. Its Dirichlet form
nevertheless equals $\|f\|_2^2$ on $\mathcal Z_H$, making the comparison
effective on this invariant subspace.

Proposition~\ref{interpolation_subgroups} refines this comparison by applying
it along a finer series of normal subgroups, interpolating between successive
terms of the lower central series. For consecutive subgroups $H\subseteq H'$,
we restrict the comparison to $\mathcal I_H\cap\mathcal Z_{H'}$.
The geometric bounds from Section~\ref{sec:geometry_Cayley} provide
representatives of the relevant quotient elements with word length at most
$\kappa_i$, where $\kappa_i$ is defined in Proposition~\ref{interpolation_subgroups}
and serves as a uniform upper bound on $\Diam_{S_\sym}(H'/H)$ over consecutive
subgroups $H\subseteq H'$ lying between $G_{i+1}$ and $G_i$.
This yields the decay rate $\mu_*/\kappa_i^2$.
The use of cyclic successive quotients is important here: it is what allows
Proposition~\ref{interpolation_subgroups} to use $\kappa_i$, rather than an
upper bound on $\Diam_{S_\sym}(G_i/G_{i+1})$.
Taking $H=G_{i+1}$ and $H'=G_i$, without interpolating through subgroups
with cyclic successive quotients, would lead to a stronger condition than
\eqref{eq:general_DC} and to weaker quantitative bounds in
Theorem~\ref{thm:unconditional_relaxation_bounds} and
Remark~\ref{rmk:large_mixing}.

The direct diameter comparison illustrated by
Fact~\ref{lem:normal_subgroup_comparison} yields versions of our main results
under a more restrictive degree condition, as in our 2024 arXiv submission
\cite{hermon2024cutoff}. The refinements developed here yield
\eqref{eq:general_DC}, which is sharp up to an absolute constant factor.
In particular, the degree condition now involves $m=|S_{\ab,\sym}|$,
rather than the size of the original generating set $S$. In the symmetric setting, this replaces $|S|$ by $|S_\ab|$.
The commutator constructions in Section~\ref{sec:geometry_Cayley} give
word-length bounds in terms of the projected generators and allow their
lifts to be interchanged without changing the represented element modulo
the relevant subgroup. Averaging over these choices according to their
probabilities under $\mu_\sym$ also replaces the dependence on $\mu_{\min}$
by dependence on the projected mass $\mu_*$. The finer comparison additionally
improves the dependence on the step, replacing $L\log\log|G|$ by
$\max\{L,\log\log|G|\}$ in the denominator of \eqref{eq:general_DC}.

\bigskip
We now turn to the preparations for Proposition~\ref{interpolation_subgroups}.
Let $X=(X_t)_{t\geq 0}$ denote the random walk on $\Gamma(\mu)$.
In what follows, let
\[X^{\ab}\qrel{:=}G_2X
\]
denote the projection of $X$ onto the Abelianization
$G_\ab=G/G_2$. Let $\P_{G_2}$ denote the law of the projected walk
$X^\ab$ started from $G_2\in G_\ab$. More precisely, for any $g\in G_2$,
\[\P_{G_2}(X_t^\ab=G_2x)
\qrel{:=}
\P_g(X_t\in G_2x).
\]
This definition is independent of the choice of $g\in G_2$.
The following lemma is an immediate consequence of the identity
\begin{equation}\label{eq:coset_uniform_lift}
\P_{\pi_{G_2}}(X_t=x)\qrel{=}\frac{1}{|G_2|}
\P_{G_2}(X_t^{\ab}=G_2x),\qquad x\in G,
\end{equation}
which shows that the density of $X_t$ started from $\pi_{G_2}$ is constant on each coset of $G_2$. We therefore omit the proof.

\begin{lemma}\label{projection_reduction_ellp}
For  $p\in [1,\infty]$ and  $t\geq 0$,
$$\bigl\|\P_{\pi_{G_2}}(X_t=\cdot) - \pi\bigr\|_{p,\pi}\qrel{=}\bigl\|\P_{G_2}(X^{\ab}_t=\cdot) - \pi_{G_\ab}\bigr\|_{p,\pi_{G_\ab}}.$$
\end{lemma}

\begin{lemma}\label{projection_reduction}
Let $H\subseteq H'$ be normal subgroups of $G$,
and define $\varphi=\varphi_{H',H}:G\to\mathbb R$ by
\beq\label{varphi}
\varphi(x)\qrel{=}\frac{\1\{x \in H\}}{|H|}
- \frac{\1\{x \in H'\}}{|H'|}.
\eeq
Then, for every $p\in[1,\infty]$ and $t\geq0$,
\[
\left\| \mathbb{P}_{\pi_H}(X_t = \cdot) - \mathbb{P}_{\pi_{H'}}(X_t = \cdot) \right\|_{p,\pi}
\qrel{\leq}|G|^{3/2} \cdot \|P_t \varphi\|_{2}.
\]
\end{lemma}

\begin{proof}

Write $\varphi=\varphi_{H',H}$ for notational simplicity. By translation invariance of the walk, for any $x\in G$,
\[
\begin{aligned}
\mathbb{P}_{\pi_H}(X_t = x)-\P_{\pi_{H'}}(X_t=x)
&\qrel{=}\sum_{h\in G} \varphi(h)\P_{\id}(X_t=h^{-1}x)\\
&\qrel{=}\sum_{y\in G}\varphi(xy^{-1})\P_{\id}(X_t=y).
\end{aligned}
\]
Since $H$ and $H'$ are normal in $G$, the function $\varphi$ is invariant under inversion and conjugation. That is,
$\varphi(xy^{-1})=\varphi(yx^{-1})=\varphi(x^{-1}y)$ for any $x,y\in G$.
Therefore,
\[
\begin{aligned}
\mathbb{P}_{\pi_H}(X_t = x)-\P_{\pi_{H'}}(X_t=x)
&\qrel{=}\sum_{y\in G}\varphi(x^{-1}y)\P_{\id}(X_t=y)\\
&\qrel{=}P_t\varphi(x^{-1}),
\end{aligned}
\]
which then implies
\beq\label{eqn:norm_equality}
\big\| \mathbb{P}_{\pi_H}(X_t = \cdot)-\P_{\pi_{H'}}(X_t=\cdot)\big\|_{p,\pi}\qrel{=}|G|\cdot \|P_t\varphi\|_{p}.
\eeq
Finally, the standard comparison of $L^p(\pi)$ norms on the finite space
$G$ gives
$
\|g\|_{p}\leq\|g\|_\infty\leq  |G|^{1/2}\|g\|_{2}$,
for any $p\in [1,\infty]$. Applying this with $g=P_t\varphi$ completes the proof.
\end{proof}

We now choose a chain of normal subgroups between $G_{i+1}$ and $G_i$
for each $i$, with cyclic successive quotients.
When $S$ is not symmetric, we define $S_i$ by applying
the construction \eqref{eq:def_S_i} to $S_\sym$.
Recall the definition of $\widetilde S_i\subseteq S_i$ from \eqref{quotient_min_gen}, which makes $G_{i+1}\widetilde S_i$ a minimal generating set of $G_i/G_{i+1}$. Let $m_i:=|\widetilde S_i|$ and write $\widetilde S_i:=\{ y_{i,j}: j\leq m_i\}$. For each $i\in[L]$ and $0\leq\ell\leq m_i$, define \beq\label{eq:def_H}
 H_{i,\ell}\qrel{:=}\langle G_{i+1},y_{i,1},\ldots,y_{i,\ell}\rangle,
 \eeq
 where $H_{i,0}:=G_{i+1}$. Since $G_{i+1}\widetilde S_i$ generates $G_i/G_{i+1}$, we have $H_{i,m_i}=G_i$. As $G_i/G_{i+1}$ is Abelian, it is not difficult to see that
$$
G_{i+1}\qrel{=}H_{i,0}\qrel{\trianglelefteq}H_{i,1}\qrel{\trianglelefteq}H_{i,2}\qrel{\trianglelefteq}\cdots\qrel{\trianglelefteq}H_{i, m_i}\qrel{=}G_i.
$$
We note that  each $H_{i,\ell}$ is normal in $G$. Indeed, for any $g\in G$ and $h\in H_{i,\ell}\subseteq G_i$, we have
$$
ghg^{-1}h^{-1}\qrel{\in}[G,G_i]\qrel{\subseteq}G_{i+1}\qrel{\subseteq}H_{i,\ell}.
$$
Hence $ghg^{-1}\in H_{i,\ell}$, and therefore
$
gH_{i,\ell}g^{-1}\subseteq H_{i,\ell}.
$
Applying the same argument with $g^{-1}$ in place of $g$ gives the reverse inclusion. Thus
$
gH_{i,\ell}g^{-1}=H_{i,\ell},
$
and hence $H_{i,\ell}\trianglelefteq G$.

\begin{prop}\label{interpolation_subgroups}
Fix $2\leq i\leq L$ and $0\leq \ell\leq m_i-1$, and let
$H'=H_{i,\ell+1}$ and $H=H_{i,\ell}$. Set
\[\kappa_i
\qrel{:=}
2^{i+7}m
\left(
2^{2i}
+
\left(\frac{D}{m}\right)^{1/i}
\right).
\]
For $\varphi=\varphi_{H',H}:G\to\RR$ defined as in \eqref{varphi}, we have
\[\|P_t\varphi\|_{2}
\qrel{\leq}
\|\varphi\|_{2}
\exp\left(
-\frac{\mu_*t}{\kappa_i^2}
\right).
\]
\end{prop}

\begin{proof}
For simplicity we will write $\kappa=\kappa_i$ throughout this proof. Let $P$ and $P^{\sym}$ denote the transition matrices associated with $\mu$
and $\mu_{\sym}$, respectively. Let $P^*$ denote the adjoint of $P$. Then
$P^\sym=(P+P^*)/2$.
Moreover, for every $f:G\to\RR$, since $\langle f, Pf\rangle_\pi=\langle f,P^*f\rangle_\pi$ it can be readily checked that
$$\EE_P(f,f)\qrel{=}\langle f,(I-P)f\rangle_\pi\qrel{=}\langle f,(I-P^{\sym})f\rangle_\pi\qrel{=}\EE_{P^{\sym}}(f,f).
$$
Thus, $P$ and $P^{\sym}$  induce the same quadratic Dirichlet form. In addition, let $\mathcal{T}\subseteq H'$ be a set of representatives of $H'/H$, i.e., $H'=\bigsqcup_{r\in \mathcal{T}}Hr$. Let $\widetilde P$ denote the
transition matrix associated with
$
\widetilde\mu
:=
\frac{1}{2}\Unif(\mathcal{T})+\frac{1}{2}\Unif(\mathcal{T}^{-1}).
$

With $A_{H'}$ as in Lemma~\ref{lem:normal_invariant_subspaces}, let
\[
\mathcal V\qrel{:=}(\mathcal I_H\cap\mathcal Z_{H'})\cap\RR^G
\qrel{=}\{H\text{-invariant }f\in\RR^G:A_{H'}f\equiv0\}.
\]

We first observe that $\varphi\in\mathcal V$. Indeed, $\varphi$ is $H$-invariant. Moreover, if $x\notin H'$,
then $xh'\notin H'$ for every $h'\in H'$, and hence
$A_{H'}\varphi(x)=0$. If $x\in H'$, then
\[A_{H'}\varphi(x)
\qrel{=}
\frac{1}{|H'|}
\left(
\sum_{y\in H}\frac{1}{|H|}
-
\sum_{y\in H'}\frac{1}{|H'|}
\right)
\qrel{=}0.
\]

We next show that $\widetilde P f=0$ for every $f\in\mathcal V$. Since
$H'=\bigsqcup_{r\in \mathcal T}Hr$, for every $H$-invariant function $f$,
\begin{align*}A_{H'}f(x)
&\qrel{=}\frac{1}{|H'|}\sum_{r\in \mathcal{T}}\sum_{h\in H}f(xhr)\qrel{=}\frac{1}{|H'|}\sum_{r\in \mathcal{T}}\sum_{h\in H}
f\bigl((xhx^{-1})xr\bigr)\qrel{=}\frac{1}{|\mathcal{T}|}\sum_{r\in \mathcal{T}}f(xr),
\end{align*}
where we used $H\trianglelefteq G$. Since $H$ is normal, $\mathcal{T}^{-1}$ is also
a complete set of representatives for the cosets of $H$ in $H'$, and the
same argument gives
$$A_{H'}f(x)\qrel{=}\frac{1}{|\mathcal{T}|}\sum_{r\in \mathcal{T}}f(xr^{-1}).
$$
Consequently, we can check that
$
\widetilde P f=A_{H'}f=0$
for every $f\in\mathcal V$.
It follows that
$$\EE_{\widetilde P}(f,f)\qrel{=}\langle f,(I-\widetilde P)f\rangle_{\pi}
\qrel{=}
\|f\|_{2}^2-\langle f,\widetilde P f\rangle_\pi
\qrel{=}
\|f\|_{2}^2
\qquad\text{for every }f\in\mathcal V.
$$

We claim the following comparison between
 Dirichlet forms of $\widetilde P$ and $P^\sym$ holds:
 \beq\label{eq:dirichlet_comparison}
 \EE_{\widetilde P}(f,f)
\qrel{\leq}
\frac{\kappa^2}{\mu_*}\EE_{P^\sym}(f,f), \qquad \forall f\in \mathcal{V}.
 \eeq
Assuming \eqref{eq:dirichlet_comparison}, for every $f\in\mathcal V$,
\beq\label{eqn:subspace_poincare}
\EE_{P^\sym}(f,f)
\qrel{\geq}
\frac{\mu_*}{\kappa^2}\|f\|_{2}^2.
\eeq

By Lemma~\ref{lem:normal_invariant_subspaces}, $\mathcal V$ is invariant under $P_t$.
Set $f_t:=P_t\varphi$. Since $\varphi\in\mathcal{V}$ it immediately follows that $f_t\in \mathcal{V}$. Applying
\eqref{eqn:subspace_poincare} to $f_t$ gives
\begin{align*}\frac{\mathrm d}{\mathrm dt}\|f_t\|_{2}^2
&\qrel{=}-2\EE_{P^\sym}(f_t,f_t)\qrel{\leq}
-\frac{2\mu_*}{\kappa^2}\|f_t\|_{2}^2.
\end{align*}
Gr\"onwall's inequality now yields
$
\|P_t\varphi\|_{2}^2
\leq
\|\varphi\|_{2}^2
\exp\left(
-\frac{2\mu_*t}{\kappa^2}
\right).
$
Taking square roots, we conclude with the desired inequality
\[\|P_t\varphi\|_{2}
\qrel{\leq}
\|\varphi\|_{2}
\exp\left(
-\frac{\mu_*t}{\kappa^2}
\right).
\]

\medskip

It remains to prove the Dirichlet comparison in
\eqref{eq:dirichlet_comparison}, which follows from a standard path comparison argument for Dirichlet forms; see, e.g., \cite{DSCcomparison}. As $S_\sym$ is a symmetric generating set of $G$, we perform the construction in \eqref{eq:def_S_i} and \eqref{eq:def_H} for $S_\sym$, according to which we have $H'=\langle H,y\rangle$ for some
$y=\rho(s_1,\dots,s_i)$ with $s_j\in R\subseteq S_\sym$ for each
$1\leq j\leq i$, where $R$ contains one representative from each inverse
pair of $S_\sym$, as in \eqref{eq:def_S_i}. This is because
$y\in\widetilde S_i\subseteq S_i$. Thus $H'/H$ is cyclic with generator $Hy$.
In particular, $s_2,\ldots,s_i\in R$, as required by
Proposition~\ref{prop:commutator_decomposition}~(ii).

Fix $r\in\mathcal T\cup\mathcal T^{-1}$. As $H'/H$ is cyclic, there exists some $b\in\mathbb N$ such that
\[Hr\qrel{=}Hy^b.
\]
By Proposition~\ref{p:bilinearity},
$
G_{i+1}y^b
=
G_{i+1}\rho(s_1^b,s_2,\ldots,s_i).
$
Applying Proposition~\ref{prop:commutator_decomposition}~(ii), we obtain
\beq\label{eq:r_representation}
G_{i+1}y^b
\qrel{=}
G_{i+1}\rho(s_1^b,s_2,\ldots,s_i)
\qrel{=}
G_{i+1}
\prod_{\alpha=1}^N
\rho(s_{\alpha,1}^{d_{\alpha,1}},\ldots,s_{\alpha,i}^{d_{\alpha,i}}),
\eeq
where $s_{\alpha,j}\in S_\sym$ and $d_{\alpha,j}\in\ZZ$.
Fix the formal expansion of the powers and commutators in the product
on the right-hand side, without cancelling any letters:
\beq\label{eq:expansion_to_letters}
\prod_{\alpha=1}^N
\rho(s_{\alpha,1}^{d_{\alpha,1}},\ldots,s_{\alpha,i}^{d_{\alpha,i}})
\qrel{=}
z_1\cdots z_M,
\qquad z_1,\ldots,z_M\in S_\sym,
\eeq
where
$
M\leq 2^{i+7}\left(2^{2i}m+m^{1-1/i}D^{1/i}\right)\leq \kappa.
$
Here we used \eqref{eq:decomposition_commutators} and
$k\leq m$.

For the tuple
$s=(s_{\alpha,j})_{\alpha\in[N],\,j\in[i]}$
appearing in \eqref{eq:r_representation}, let
\[\Omega_s
\qrel{:=}
\prod_{\alpha\in[N],\,j\in[i]}
\left(S_\sym\cap G_2s_{\alpha,j}\right).
\]
Thus an element
$
\tilde s
=
(\tilde s_{\alpha,j})_{\alpha\in[N],\,j\in[i]}
\in\Omega_s
$
satisfies $G_2\tilde s_{\alpha,j}=G_2s_{\alpha,j}$ for each $(\alpha,j)$. Equip $\Omega_s$ with the probability measure
\[Q_s(\tilde s)
\qrel{:=}
\prod_{\alpha\in[N],\,j\in[i]}
\frac{\mu_\sym(\tilde s_{\alpha,j})}
{\sum_{v\in S_\sym\cap G_2s_{\alpha,j}}\mu_\sym(v)}.
\]

Then define, for $\tilde s\in \Omega_s$,
\[\phi(\tilde s)
\qrel{:=}
\prod_{\alpha=1}^N
\rho\left(
\tilde s_{\alpha,1}^{d_{\alpha,1}},
\ldots,
\tilde s_{\alpha,i}^{d_{\alpha,i}}
\right).
\]

Since $G_2\tilde s_{\alpha,j}=G_2s_{\alpha,j}$ for every $(\alpha,j)$,
\eqref{e:bylinearity} applies to each commutator factor, including its
integer powers. Together with \eqref{eq:r_representation}, it gives
$G_{i+1}\phi(\tilde s)=G_{i+1}y^b$. As $G_{i+1}\subseteq H$ and
$Hr=Hy^b$, it follows that $H\phi(\tilde s)=Hr$. Hence, since
$H\trianglelefteq G$ and $f\in\mathcal V$ is $H$-invariant, we have
\beq\label{eq:invariance_ab_lift}
f(\cdot\phi(\tilde s))\qrel{=}f(\cdot r),\qquad \forall \tilde s\in \Omega_s.
\eeq

Fully expanding the powers and commutators in $\phi(\tilde s)$, in the same
fixed formal way as in \eqref{eq:expansion_to_letters}, without
cancellation, yields a sequence of the same length $M$:
\[\phi(\tilde s)
\qrel{=}
z_1(\tilde s)\cdots z_M(\tilde s),
\qquad
z_1(\tilde s),\ldots,z_M(\tilde s)\in S_\sym.
\]
At each position $h$, the letter $z_h(\tilde s)$ is a fixed coordinate
$\tilde s_{\alpha,j}$ or its inverse: the coordinate and sign depend only
on the formal expansion, not on $\tilde s$.
For fixed $\tilde s$, write $w_0=\id$ and
$w_h=z_1(\tilde s)\cdots z_h(\tilde s)$. Translation invariance of the
uniform measure gives
\[
\|f(\cdot w_{h-1}z_h(\tilde s))-f(\cdot w_{h-1})\|_2^2
\qrel{=}\|f(\cdot z_h(\tilde s))-f\|_2^2.
\]
Thus the prefix is removed before averaging over $\tilde s$; no
independence between that prefix and the letter is needed.
Using \eqref{eq:invariance_ab_lift}, followed by telescoping and the
Cauchy--Schwarz inequality, we obtain
\begin{align}\label{eq:telescoping}
\|f(\cdot r)-f\|_{2}^2\qrel{=}
\sum_{\tilde s\in\Omega_s}
Q_s(\tilde s)
\|f(\cdot\phi(\tilde s))-f\|_{2}^2\qrel{\leq}
M\sum_{h=1}^M
\sum_{\tilde s\in\Omega_s}
Q_s(\tilde s)
\|f(\cdot z_h(\tilde s))-f\|_{2}^2.
\end{align}
Fix $h\in[M]$. There exist $(\alpha,j)$ and
$\varepsilon_h\in\{\pm1\}$ such that
$
z_h(\tilde s)=(\tilde s_{\alpha,j})^{\varepsilon_h}.
$
Summing over all the other coordinates of $\tilde s$ gives
\begin{align}\label{eq:comparison_bound}
\sum_{\tilde s\in\Omega_s}
\nonumber Q_s(\tilde s)
\|f(\cdot z_h(\tilde s))-f\|_{2}^2
&\qrel{=}
\frac{1}{\sum_{s\in S_\sym\cap G_2s_{\alpha,j}}\mu_\sym(s)}
\sum_{s\in S_\sym\cap G_2s_{\alpha,j}}
\mu_\sym(s)
\|f(\cdot s^{\varepsilon_h})-f\|_{2}^2\\
&\qrel{\leq}
\frac{1}{\mu_*}
\sum_{s\in S_\sym}
\mu_\sym(s)
\|f(\cdot s)-f\|_{2}^2,
\end{align}
where the last inequality follows from
$\sum_{s\in S_\sym\cap G_2s_{\alpha,j}}\mu_\sym(s)\geq\mu_*$, together with the observation that $\|f(\cdot s^{-1})-f\|_{2}^2=\|f(\cdot s)-f\|_{2}^2$.

The last sum in \eqref{eq:comparison_bound} equals
$2\EE_{P^\sym}(f,f)$. The fixed-coordinate marginal is all that is
used here, even when several positions use the same coordinate.
Combining \eqref{eq:telescoping}, \eqref{eq:comparison_bound} and using
$M\leq\kappa$, we obtain
\[\|f(\cdot r)-f\|_{2}^2
\qrel{\leq}
\frac{\kappa^2}{\mu_*}
\sum_{s\in S_\sym}
\mu_\sym(s)
\|f(\cdot s)-f\|_{2}^2.
\]
Averaging over the probability measure $\widetilde\mu$ introduces no
additional factor, and the factor $1/2$ in the Dirichlet form gives
\begin{align*}\EE_{\widetilde P}(f,f)
&\qrel{=}
\frac12
\sum_r\widetilde\mu(r)
\|f(\cdot r)-f\|_{2}^2\\
&\qrel{\leq}
\frac{\kappa^2}{2\mu_*}
\sum_{s\in S_\sym}
\mu_\sym(s)
\|f(\cdot s)-f\|_{2}^2\qrel{=}
\frac{\kappa^2}{\mu_*}
\EE_{P^\sym}(f,f),
\end{align*}
which proves \eqref{eq:dirichlet_comparison}. The proof is thus completed.
\end{proof}

We conclude with an estimate on the quantities
$
\kappa_i=
2^{i+7}m
\left(
2^{2i}
+
\left(D/m\right)^{1/i}
\right),
$
which will be useful in the proof below.

\begin{lemma}\label{lem:kappa_endpoint_bound}
For $L\geq2$, we have
\[
\max_{2\leq i\leq L}\kappa_i^2
\qrel{=}\max_{i\in \{2,L\}}\kappa_i^2
\qrel{\lesssim}m^264^L+mD.
\]
\end{lemma}

\begin{proof}
Set $b:=\log(D/m)$ and consider
\[
f(x)\qrel{:=}2^x\left(2^{2x}+(D/m)^{1/x}\right)
\qrel{=}2^{3x}+\exp\left(x\log2+\frac bx\right).
\]
A direct calculation shows that $f$ is convex on $[2,\infty)$. Consequently, $f^2$ attains its maximum on $[2,L]$
at an endpoint. Since $\kappa_i=2^7mf(i)$, it follows that
$
\max_{2\leq i\leq L}\kappa_i^2
=\max_{i\in \{2,L\}}\kappa_i^2.
$

Moreover, for every $2\leq i\leq L$,
\[
2^{2i}\left(2^{2i}+e^{b/i}\right)^2
\qrel{\lesssim}64^i+2^{2i}e^{2b/i}
\qrel{\lesssim}64^i+e^b \qrel{\lesssim} 64^L+D/m,
\]
where the second-to-last inequality follows from the weighted
arithmetic--geometric mean inequality, using $2/i\leq1$. Multiplying both sides by $m^2$ completes the proof.
\end{proof}

\subsection{Proofs of Theorems \ref{thm: mixing and spectral comparison with Abelianization} and \ref{thm:unconditional_relaxation_bounds}}
The assertions are immediate when $L=1$; henceforth assume $L\geq2$.
Recall that $D=\Diam_{S_{\ab,\sym}}(G_\ab)$ and $D_2=D_{\ell^2}(\Gamma_\ab)$ where $\Gamma_\ab=\Cay(G_\ab,S_{\ab,\sym})$.

We first record lower bounds that do not require \eqref{eq:general_DC}.
Theorem~\ref{thm:Geometric bounds on mixing for Abelian}\textup{(i)},
$D^2\leq mD_2^2$, and $k\leq m$ give
\[
\frac{D^2}{m}
\qrel{\leq}D_2^2
\qrel{\lesssim}t_\mix^\TV(\mu_{\ab,\sym})
\qrel{\lesssim}m t_\rel(\mu_\ab).
\]
In particular,
\[
t_\rel(\mu_\ab)\qrel{\gtrsim}\frac{D_2^2}{m}
\qrel{\geq}\frac{D^2}{m^2}.
\]

Set $M:=\max\{\log\log|G|,L\}$. We now record the consequences of \eqref{eq:general_DC} needed below.
Under \eqref{eq:general_DC}, Lemma~\ref{diam_abelianization} and
$m/2\leq k\leq m$ yield
\begin{equation}\label{eq:DC_diameter_consequence}
D\qrel{\gtrsim}m|G_\ab|^{1/m}
\qrel{\geq}\frac{e^{5M}}{\mu_*}.
\end{equation}
Since $m\mu_*\leq1$ and $L\leq M$, this implies
$\frac{64^Lm^2}{mD}
\lesssim m\mu_*e^{L\log64-5M}
\leq e^{-(5-\log64)M}=o(1)$.
Hence, Lemma \ref{lem:kappa_endpoint_bound} yields
\begin{equation}\label{eq:kappa_bound}
\max_{2\leq i\leq L}\kappa_i^2\qrel{=}\max_{i\in\{2,L\}}\kappa_i^2\qrel{\lesssim}mD.
\end{equation}
Moreover, $m\leq\log|G|\leq e^M$, and hence
$
\frac{\mu_*D}{m^2\log|G|}
\gtrsim e^{2M}\to\infty,
$
i.e.,
$
\frac{\log|G|}{\mu_*}
\ll\frac{D}{m^2}.
$
Combining this with \eqref{eq:kappa_bound} and
Lemma~\ref{lem:kappa_endpoint_bound}, we obtain the following
estimate, which will be used in the mixing, separation, and
relaxation time comparisons:
\begin{equation}\label{eq:time_condition}
\frac{\log|G|}{\mu_*}\max_{2\leq i\leq L}\kappa_i^2
\qrel{=}
\frac{\log|G|}{\mu_*}\max_{i\in\{2,L\}}\kappa_i^2
\qrel{\ll}\frac{D^2}{m}.
\end{equation}

\bigskip

Since both the $\ell^p(\pi)$ distance and the separation distance are nonincreasing under projection, we immediately obtain
$$t_{\mathrm{mix}}^{\ell^{p}}(\mu,\varepsilon)\qrel{\geq}
t_{\mathrm{mix}}^{\ell^{p}}(\mu_{\mathrm{ab}},\varepsilon)
\qquad\text{and}\qquad
t_{\mathrm{sep}}(\mu,\varepsilon)\qrel{\geq}
t_{\mathrm{sep}}(\mu_{\mathrm{ab}},\varepsilon).
$$
It therefore remains to establish the upper bounds in \eqref{e:mixing time given by Abelianization} and \eqref{e:separation mixing time given by Abelianization}. The relaxation-time assertions \eqref{e:gap given by Abelianization}, \eqref{eq:relaxation_time_comparison}, and \eqref{eq:weak_relaxation_comparison} of both theorems will be proved separately.

\subsubsection{Proof of the $\ell^p$-mixing time comparison
\eqref{e:mixing time given by Abelianization}.}
\label{proof of mixing time}
We will prove that, for every $p\in[1,\infty]$,
$\varepsilon\in(0,1)$, and
$0<\delta<\min\{\varepsilon,1-\varepsilon\}$, we have $
t_{\mathrm{mix}}^{\ell^{p}}(\mu,2\varepsilon)
\leq
t_{\mathrm{mix}}^{\ell^{p}}(\mu_{\mathrm{ab}},2\varepsilon-2\delta)$
 whenever $|G|\geq N_1(\delta)$.

Since $G_{L+1}=\{\id\}$, the triangle inequality gives
\begin{align}\label{triangle_inequality_distance}
\nonumber &\bigl\|\P_{\id}(X_t=\cdot) - \pi\bigr\|_{p,\pi}\\
\nonumber\qrel{\leq} &\bigl\|\P_{\pi_{G_2}}(X_t=\cdot) - \pi\bigr\|_{p,\pi}+\sum_{i=2}^L
 \bigl\|\P_{\pi_{G_{i+1}}}(X_t=\cdot)-\P_{\pi_{G_i}}(X_t=\cdot)\bigr\|_{p,\pi}\\
 \qrel{\leq} &\bigl\|\P_{G_2}(X^{\ab}_t=\cdot) - \pi_{G_\ab}\bigr\|_{p,\pi_{G_\ab}}+\sum_{i=2}^L \sum_{\ell=0}^{m_i-1} \bigl\|\P_{\pi_{H_{i,\ell}}}(X_t=\cdot)-\P_{\pi_{H_{i,\ell+1}}}(X_t=\cdot)\bigr\|_{p,\pi}.
 \end{align}
Here the second inequality uses Lemma~\ref{projection_reduction_ellp}
and the refined subgroup series defined in \eqref{eq:def_H}.

Applying Lemma~\ref{projection_reduction} and
Proposition~\ref{interpolation_subgroups} with
$H'=H_{i,\ell+1}$ and $H=H_{i,\ell}$, together with
the bound $\|\varphi_{H',H}\|_2\leq |G|^{-1/2}$
from the definition \eqref{varphi}, we see that each
summand in \eqref{triangle_inequality_distance} is at most
$|G|\exp(-\mu_*t/\kappa_i^2)$.
Summing these bounds and using
$L\leq |G|$ and $\max_i m_i\leq |G|$, we obtain
\begin{align}\label{eq:lp_mixing_comparison_abelianization}
\bigl\|\P_\id(X_t=\cdot)-\pi\bigr\|_{p,\pi}
&\qrel{\leq}
\bigl\|\P_{G_2}(X_t^\ab=\cdot)-\pi_{G_\ab}\bigr\|_{p,\pi_{G_\ab}}
\nonumber\\
&\qquad+|G|^3\exp\left(
-\frac{\mu_*t}{\max_{2\leq i\leq L}\kappa_i^2}
\right).
\end{align}
Now take $t=t_{\mix}^{\ell^p}(\mu_\ab,2\ep-2\delta)$.
Since $\ep-\delta<1-\delta$, monotonicity and
\eqref{eq:geometric_high_threshold_lower_bound} give
\[
t\qrel{\geq}t_{\mix}^{\TV}(\mu_\ab,\ep-\delta)
\qrel{\geq}t_{\mix}^{\TV}(\mu_\ab,1-\delta)
\qrel{\gtrsim_\delta}D_2^2\qrel{\geq}D^2/m.
\]
The scale condition in \eqref{eq:geometric_high_threshold_lower_bound}
holds whenever $|G|$ is sufficiently large in terms of $\delta$,
since \eqref{eq:DC_diameter_consequence} gives $D_2\geq D/\sqrt m\to\infty$
uniformly under \eqref{eq:general_DC}.
Together with \eqref{eq:time_condition}, this makes the exponential error
in \eqref{eq:lp_mixing_comparison_abelianization} at most $2\delta$
whenever $|G|\geq N_1(\delta)$, uniformly in $\ep$ and $p$.
The projected distance is at most $2\ep-2\delta$, so the full distance
is at most $2\ep$.
\qed

\subsubsection{Proof of the separation mixing time comparison
\eqref{e:separation mixing time given by Abelianization}.}
For any $t\geq0$, we have
\[
\begin{aligned}
\max_{x\in G}\{1-|G|\P_{\id}(X_t=x)\}
&\qrel{\leq}
\max_{x\in G}\{1-|G|\P_{\pi_{G_2}}(X_t=x)\}\\
&\qquad
+|G|\max_{x\in G}
\{\P_{\pi_{G_2}}(X_t=x)-\P_{\id}(X_t=x)\}.
\end{aligned}
\]
Take $t=t_{\mathrm{sep}}(\mu_\ab,\ep-\delta)$.
By \eqref{eq:coset_uniform_lift},
\[
\begin{aligned}
\max_{x\in G}\{1-|G|\P_{\pi_{G_2}}(X_t=x)\}
&\qrel{=}
\max_{x\in G}
\{1-|G_\ab|\P_{G_2}(X_t^\ab=G_2x)\}\qrel{\leq}\ep-\delta.
\end{aligned}
\]
For the second term, applying the same telescoping argument as above
with $p=\infty$ gives
\begin{align*}
|G|\max_{x\in G}
\{\P_{\pi_{G_2}}(X_t=x)-\P_\id(X_t=x)\}
&\qrel{\leq}
\bigl\|\P_{\pi_{G_2}}(X_t=\cdot)
-\P_\id(X_t=\cdot)\bigr\|_{\infty,\pi}\\
&\qrel{\leq}
|G|^3\exp\left(
-\frac{\mu_*t}{\max_{2\leq i\leq L}\kappa_i^2}
\right).
\end{align*}
Since separation distance dominates total variation distance
(see \cite[Chapter~6]{wilmer2009markov}),
$t\geq t_{\mix}^{\TV}(\mu_\ab,\ep-\delta)
\geq t_{\mix}^{\TV}(\mu_\ab,1-\delta)$.
The same estimate as above makes the error at most $\delta$ after increasing
$N_1(\delta)$ if necessary. Thus, for $|G|\geq N_1(\delta)$,
we conclude as above that the full separation distance is at most $\ep$,
proving the upper bound in \eqref{e:separation mixing time given by Abelianization}.

\subsubsection{Proof of the relaxation time results \eqref{e:gap given by Abelianization}, \eqref{eq:relaxation_time_comparison}, and \eqref{eq:weak_relaxation_comparison}.}

Symmetrization preserves all quantities and assumptions in
\eqref{e:gap given by Abelianization}, \eqref{eq:relaxation_time_comparison},
and \eqref{eq:weak_relaxation_comparison}, so we may assume that $\mu$ is
symmetric. Taking exponential decay rates in
\eqref{eq:lp_mixing_comparison_abelianization} with $p=1$, whose derivation
does not require \eqref{eq:general_DC}, and using
\eqref{eq:decay_rate_TV} and \eqref{eq:equality_relaxation_times}, we obtain
\[
t_{\rel}(\mu)\qrel{\leq}
\max\left\{t_{\rel}(\mu_\ab),\frac1{\mu_*}
\max_{2\leq i\leq L}\kappa_i^2\right\}.
\]
Lemma~\ref{lem:kappa_endpoint_bound} now proves
\eqref{eq:relaxation_time_comparison}.

Under \eqref{eq:general_DC}, the relaxation lower bound established above
and \eqref{eq:time_condition} give
\begin{equation}\label{eq:relaxation_dominates_comparison_cost}
t_{\rel}(\mu_\ab)
\qrel{\gtrsim}\frac{D^2}{m^2}
\qrel{\gg}\frac1{\mu_*}\max_{2\leq i\leq L}\kappa_i^2,
\end{equation}
where we used $m\leq\log|G|$, which follows from \eqref{eq:general_DC}.
Together with the preceding comparison and
\eqref{eq:relaxation_projection_chain}, this proves
\eqref{e:gap given by Abelianization}.

Finally, the unconditional bound $D\lesssim m\sqrt{t_{\rel}(\mu_\ab)}$
and \eqref{eq:relaxation_time_comparison} imply that, if
$t_{\rel}(\mu)>t_{\rel}(\mu_\ab)$, then, for an absolute constant $C'>0$,
\[
\begin{aligned}
t_{\rel}(\mu)
&\qrel{\leq}\frac{C'm^2}{\mu_*}
\left(64^L+\sqrt{t_{\rel}(\mu_\ab)}\right)\\
&\qrel{\leq}\frac{C'64^Lm^2}{\mu_*}
+\frac{t_{\rel}(\mu)}2+\frac{(C')^2m^4}{2\mu_*^2}.
\end{aligned}
\]
Rearranging proves \eqref{eq:weak_relaxation_comparison} with
$C_2=\max\{2C',(C')^2\}$; the case
$t_{\rel}(\mu)=t_{\rel}(\mu_\ab)$ is immediate.

\begin{remark}[Geometric conditions for mixing-time comparison]
\label{rmk:geometric_comparison_conditions}
We now justify Remark~\ref{rmk:mixing_time_condition}\textup{(i)}.
Set $K:=\max_{2\leq i\leq L}\kappa_i^2$ and
$Q:=\max\{8^L\sqrt{\mathfrak C},\mathfrak C/m\}$.
By Lemma~\ref{lem:kappa_endpoint_bound}, the condition $D\geq CQ$ with
$C\geq1$ implies
\[
\frac{\mu_*D^2}{mK\log|G|}\qrel{\gtrsim}C.
\]
Applying Theorem~\ref{thm:Geometric bounds on mixing for Abelian}\textup{(i)}
and the usual reversible mixing upper bound to the projected symmetrized walk gives
$D_2^2\lesssim t_\mix^\TV(\mu_{\ab,\sym})
\lesssim t_\rel(\mu_\ab)\log|G_\ab|$. Hence
\[
t_\rel(\mu_\ab)
\qrel{\gtrsim}\frac{D_2^2}{\log|G_\ab|}
\qrel{\geq}\frac{D^2}{m\log|G_\ab|}
\qrel{\gtrsim}\frac{CK}{\mu_*}.
\]
For a sufficiently large absolute $C$, the unsimplified relaxation-time
comparison in the preceding proof and the projection lower bound yield
\eqref{e:gap given by Abelianization}.

For every $\ep\in(0,1)$ and $0<\delta<\min\{\ep,1-\ep\}$, the relevant
projected mixing times in \eqref{e:mixing time given by Abelianization} and
\eqref{e:separation mixing time given by Abelianization} are at least
$t_{\mix}^{\TV}(\mu_\ab,1-\delta)\gtrsim_\delta D_2^2\geq D^2/m$ by
\eqref{eq:geometric_high_threshold_lower_bound}.
The scale condition for this lower bound holds when $D\geq C(\delta)Q$,
after increasing $C(\delta)$, since $D_2\geq D/\sqrt m$ and
$Q/\sqrt m\geq8^Lm\sqrt{\log|G|/\mu_*}\gtrsim1$.
The exponential error in \eqref{eq:lp_mixing_comparison_abelianization}
is thus at most
$|G|^3\exp(-c_\delta C(\delta)\log|G|)$ for some $c_\delta>0$.
Choosing $C(\delta)$ large enough makes this at most $\delta$, uniformly
in $\ep$ and $p\in[1,\infty]$, proving the two mixing-time comparisons.

More generally, \eqref{eq:geometric_high_threshold_lower_bound} and the
same reversible mixing upper bound show that the comparisons hold
eventually, for every fixed choice of threshold parameters, whenever
\begin{equation}\label{eq:euclidean_comparison_condition}
D_2^2\qrel{\gg}\frac{\log|G|}{\mu_*}\max_{2\leq i\leq L}\kappa_i^2.
\end{equation}
The scale condition for the high-threshold lower bound holds eventually
under this assumption, since $\mu_*\leq1/m$ and $K\gtrsim m^2$.
Moreover, \eqref{eq:DC_diameter_consequence} implies
$D\gg m^2\log|G|/\mu_*$ and
$D^2\gg64^Lm^3\log|G|/\mu_*$, so \eqref{eq:general_DC} implies $D\gg Q$.

Finally, if $k\leq C\log|G_\ab|$ for some fixed $C>0$,
Lemma~\ref{diam_abelianization} gives
$D\gtrsim_C k|G_\ab|^{1/k}$.
Since $m/2\leq k\leq m$, a further sufficient condition for
$D\gg\max\{8^L\sqrt{\mathfrak C},\mathfrak C/m\}$ is
\[
|G_\ab|^{1/k}\qrel{\gg}
\max\{8^L\sqrt{\mathfrak{B}},\mathfrak{B}\},
\qquad
\mathfrak{B}\qrel{:=}\frac{k\log|G|}{\mu_*}.
\]

\end{remark}

\begin{remark}[Equilibration of successive layers]
\label{rmk:layer_mixing}
For $2\leq\ell\leq L$, let
$d_\ell(t):=\|\P_{\pi_{G_\ell}}(X_t\in\cdot)
-\P_{\pi_{G_{\ell+1}}}(X_t\in\cdot)\|_{\infty,\pi}$.
The telescoping argument in \eqref{triangle_inequality_distance}, applied
along $H_{\ell,j}$ using Lemma~\ref{projection_reduction} and
Proposition~\ref{interpolation_subgroups}, gives, without \eqref{eq:general_DC},
\[
d_\ell(t)\qrel{\leq}\frac{|G|^3}{2}
\exp\left(-\frac{\mu_*t}{\kappa_\ell^2}\right).
\]
Consequently, $d_\ell(t)\leq\varepsilon$ for $\varepsilon\in(0,1)$ whenever
$t\gtrsim\frac{m^2}{\mu_*}
\left(64^\ell+4^\ell(D/m)^{2/\ell}\right)\log(|G|/\varepsilon)$.
The exponent $2/\ell$ on $D/m$ is reminiscent of the individual-coordinate
estimates of Diaconis and Hough \cite{diaconis2021random}, discussed in
Section~\ref{sec:discussion}.
\end{remark}

\section{Mixing on Abelian Groups via Lattices---Proof of Theorem \ref{thm:Geometric bounds on mixing for Abelian}}
\label{s:Abelian  mixing}

\subsection{Preliminaries on lattices}\label{sec:lattice background}
A full-rank lattice in $\RR^d$ has the form $\Lambda=\mathbf{B}\ZZ^d$ for an
invertible matrix $\mathbf{B}\in \mathbb{R}^{d\times d}$; its determinant $\det(\Lambda):=|\det \mathbf{B}|$ is the
volume of a fundamental domain. All lattices considered below are assumed to have full rank;
we refer to \cite{rothvoss2023lattices} for background on lattices.

\medskip
Throughout this section, we follow the Abelian setup in
\eqref{eq:abelian_inverse_pairs}, with the choice of inverse-pair
representatives $S_\sym=\bigsqcup_{j=1}^k\{\pm s_j\}$.
 Let $\mathbf 0_A$ denote the identity element of $A$. We associate
with $(A,S_\sym)$ the full-rank lattice
\beq
\label{lattice_def}
\Lambda\qrel{=}\Lambda(A,S_\sym)
\qrel{:=}\left\{z\in\ZZ^k:\sum_{i=1}^k z_i s_i=\mathbf 0_A\right\}.
\eeq
The map $z+\Lambda\mapsto\sum_i z_is_i$ identifies $\ZZ^k/\Lambda$ with
$A$, so $\Lambda$ has full rank and $\det(\Lambda)=|A|$. Through this isomorphism, we identify the walk on $A$ with
the corresponding walk on $\ZZ^k/\Lambda$, where the increments
$\pm s_i$ correspond to $\pm e_i+\Lambda$.
This lattice representation plays a central role both in Hough's analysis
\cite{hough2017mixing} and in the arguments below.

\begin{definition}[Dual lattice]
The dual lattice of a full-rank lattice $\Lambda \subseteq\mathbb{R}^k$ is defined as
        \beq\Lambda^*\qrel{:=}\{x \in \mathbb{R}^k : \inprod{x, \lambda} \in \mathbb{Z}, \text{ for all }  \lambda \in \Lambda \},
        \eeq
        where $\inprod{x, \lambda}:=\sum_{i=1}^k x_i\lambda_i$. This satisfies $\mathrm{det}(\Lambda^*) \mathrm{det}(\Lambda)=1$.
\end{definition}

The quotient group $\Lambda^*/\ZZ^k$ is canonically isomorphic to
the dual group of $A=\ZZ^k/\Lambda$, namely, the group of all characters
of $A$. In particular,
$
\left|\Lambda^*/\ZZ^k\right|
=
\left|\ZZ^k/\Lambda\right|
=
|A|.
$
Each coset in $\Lambda^*/\ZZ^k$ has a unique representative in
$[-1/2,1/2)^k$. Accordingly, let
$$\mathcal R\qrel{:=}\Lambda^*\cap[-1/2,1/2)^k,
$$
so that $\mathcal R$ is a complete set of representatives of
$\Lambda^*/\ZZ^k$. For each $\lambda\in\mathcal R$, define the character
$\chi_\lambda:\ZZ^k/\Lambda\to\mathbb C$ by
\[
\chi_\lambda(x+\Lambda)
\qrel{:=}
e^{2\pi i\inprod{x,\lambda}},
\qquad x\in\ZZ^k.
\]
These characters form an orthonormal basis of $L^2(\pi_A)$
and are eigenfunctions of the transition operator $P_\mu$,
with corresponding eigenvalues
\[
b_\mu(\lambda)
\qrel{:=}\sum_{s\in S}\mu(s)\chi_\lambda(s),
\qquad \lambda\in\mathcal R.
\]

\begin{lemma}[Relaxation time via the dual lattice]
\label{lem:relaxation_time_asymp_weighted}
For $i\in[k]$, set $\mu_i:=\mu(\{\pm s_i\})$. Then
\[
t_{\rel}(\mu)^{-1}
\qrel{=}\min_{\lambda\in\mathcal R\setminus\{0\}}
\sum_{i=1}^k\mu_i\bigl(1-\cos(2\pi\lambda_i)\bigr)
\qrel{\asymp}\min_{\lambda\in\mathcal R\setminus\{0\}}
\sum_{i=1}^k\mu_i\lambda_i^2.
\]
\end{lemma}

\begin{proof}
Under the identification $A\cong\ZZ^k/\Lambda$,
each generator $s_i$ corresponds to $e_i+\Lambda$, so
$\chi_\lambda(s_i)=e^{2\pi i\lambda_i}$.
Grouping increments into inverse pairs gives
\begin{equation}\label{eq:abelian_eigenvalue_real_part}
1-\Re b_\mu(\lambda)
\qrel{=}
\sum_{i=1}^k\mu_i\bigl(1-\cos(2\pi\lambda_i)\bigr).
\end{equation}
Since $A$ is Abelian, $P_\mu$ is normal, and hence
\eqref{eq:equality_relaxation_times} yields
\[
t_{\rel}(\mu)^{-1}
\qrel{=}(t_{\rel}^{\ev}(\mu))^{-1}
\qrel{=}\min_{\lambda\in\mathcal R\setminus\{0\}}
\bigl(1-\Re b_\mu(\lambda)\bigr),
\]
proving the desired equality.
The asymptotic comparison then follows from
$8x^2\leq1-\cos(2\pi x)\leq2\pi^2x^2$
for $x\in[-1/2,1/2)$.
\end{proof}

For a full-rank lattice $\Lambda\subseteq\RR^k$ and $a>0$,
define the \emph{Gaussian mass} of its translate $\Lambda+x$ by
\[
\rho_a(\Lambda+x)
\qrel{:=}\sum_{z\in\Lambda}
e^{-\pi\|z+x\|_2^2/a^2},
\qquad x\in\RR^k,
\]
and write $\rho_a(\Lambda):=\rho_a(\Lambda+\mathbf{0})$.
Also, let
$\dist(x,\Lambda):=\min_{z\in\Lambda}\|x-z\|_2$.
We use the following estimate of Banaszczyk;
see also \cite[Lemma~1.3.10]{stephens2017gaussian}.

\begin{lemma}[Banaszczyk \cite{banaszczyk1993new}]
\label{shifted_Gaussian_mass}
For every full-rank lattice $\Lambda\subseteq\RR^k$,
$a>0$, and $x\in\RR^k$,
\[
e^{-\pi\dist(x,\Lambda)^2/a^2}\rho_a(\Lambda)
\qrel{\leq}
\rho_a(\Lambda+x)
\qrel{\leq}
\rho_a(\Lambda).
\]
\end{lemma}

\subsection{Mixing times and symmetrization}
\label{subsec:mixing_times_symmetrization}
For a walk with increment law $\mu$, write
\[
d_\mu^{\TV}(t)\qrel{:=}\|P_t^\mu(\mathbf 0_A,\cdot)-\pi_A\|_{\TV}
\qquad\text{and}\qquad
d_\mu^{(p)}(t)\qrel{:=}\|P_t^\mu(\mathbf 0_A,\cdot)-\pi_A\|_{p,\pi_A}.
\]
We first establish the elementary fact \eqref{eq:unconditional_symmetrization_mixing_times},
which compares the mixing times of a random walk with those of its
symmetrization and will be used repeatedly in the proof of
Theorem~\ref{thm:Geometric bounds on mixing for Abelian}~(i).

\begin{proof}[Proof of \eqref{eq:unconditional_symmetrization_mixing_times}]
Let $\mu_t$ denote the law of the continuous-time random walk driven by
$\mu$ at time $t$, and write $\check\mu_t(x):=\mu_t(-x)$.
Since $A$ is Abelian,
$\mu_t*\check\mu_t=(\mu_\sym)_{2t}$.
As convolution contracts total variation distance, we have
\[
d^{\TV}_{\mu_\sym}(2t)\qrel{\leq}d^{\TV}_\mu(t),
\]
which gives
$\tfrac12 t_{\mix}^{\TV}(\mu_\sym,\varepsilon)
\leq t_{\mix}^{\TV}(\mu,\varepsilon)$.

For the remaining comparisons, write $\nu_t:=\mu_t-\pi$.
Then
$\nu_t*\check\nu_t=(\mu_\sym)_{2t}-\pi$
and $\nu_t*\nu_t=\mu_{2t}-\pi$.
By Cauchy--Schwarz,
\[
\|\nu_t*\nu_t\|_\infty
\qrel{\leq}\sum_{x\in A}\nu_t(x)^2
\qrel{=}(\nu_t*\check\nu_t)(\mathbf 0_A).
\]
It follows that
\begin{equation}\label{eq:abelian_l2_linf_trace}
\begin{aligned}
\bigl(d_\mu^{(2)}(t)\bigr)^2
&\qrel{=}|A|(\nu_t*\check\nu_t)(\mathbf 0_A)
\qrel{=}\bigl(d_{\mu_\sym}^{(2)}(t)\bigr)^2\\
&\qrel{=}d_{\mu_\sym}^{(\infty)}(2t)
\qrel{\geq}d_\mu^{(\infty)}(2t).
\end{aligned}
\end{equation}
The third equality uses reversibility, while the last inequality follows
from the preceding bound. This yields the first line of
\eqref{eq:unconditional_symmetrization_mixing_times}, completing the proof.
\end{proof}

\subsection{Geometric bounds on mixing times}
\label{s:geometric_bound_mixing_times}
Recall the definitions of $D_2$ and $D_2(\mu)$ from \eqref{eq:euclidean_diameter} and \eqref{eq:weighted_l2_diameter}.
Under the lattice identification, \eqref{eq:euclidean_diameter} gives $D_2=\max_{x\in\ZZ^k}\dist(x,\Lambda)$.

\medskip
We begin with a lower bound on the total-variation mixing time
at threshold $1/4$. In the unweighted case, we obtain a more
quantitative bound at thresholds $1-\ep$, allowing $\ep$ to vary
with the group. The latter will be useful in the proof of
Theorem~\ref{thm: mixing and spectral comparison with Abelianization}.

\begin{lemma}[Geometric lower bounds on TV mixing]
\label{geometric_bound_TV}
We have
\[
t_{\mix}^{\TV}(\mu)
\qrel{\gtrsim}
\max\{D_2^2,D_2^2(\mu)\}.
\]
Moreover, there exist absolute constants $c,C>0$ such that,
for every $0<\ep\leq1/2$ satisfying
$\ep D_2\geq C\log(1/\ep)$,
\[
t_{\mix}^{\TV}(\mu,1-\ep)
\qrel{\geq}
c\,\frac{\ep^2D_2^2}{\log(1/\ep)}.
\]
\end{lemma}
\begin{proof}
Since symmetrization preserves both $D_2$ and $D_2(\mu)$,
\eqref{eq:unconditional_symmetrization_mixing_times}
allows us to assume that $\mu$ is symmetric,
at the cost of at most a factor of two in the mixing time.

Let $S=\bigsqcup_{i\in[k]}\{\pm s_i\}$ be the support of
$\mu$, and put $m:=|S|$. Recall that
$\mu_i=\mu(\{\pm s_i\})$ and $\sum_i\mu_i=1$.
Consider the continuous-time random walk $Z_t$ on $\ZZ^k$
with jumps $\pm e_i$, each at rate $\mu_i/2$.
Its projection onto $A$ under the homomorphism
$e_i\mapsto s_i$ has law $\mu_t$.
Besides the Euclidean norm, equip $\RR^k$ with the
weighted spatial norm
\[
\|z\|_{\mathrm{sp},\mu}^2
\qrel{:=}\sum_{i=1}^k\frac{z_i^2}{m\mu(s_i)}.
\]
Since $\E[Z_t(i)^2]=t\mu_i$, we have
\[
\begin{aligned}
\E\|Z_t\|_2^2
&\qrel{=}t\sum_{i=1}^k\mu_i\qrel{=}t,\\
\E\|Z_t\|_{\mathrm{sp},\mu}^2
&\qrel{=}\frac{t}{m}\sum_{i=1}^k\frac{\mu_i}{\mu(s_i)}
\qrel{=}t.
\end{aligned}
\]
The last equality uses symmetry and the partition of $S$ into inverse pairs.

We first prove the lower bound at threshold $1/4$.
Let $\|\cdot\|_*$ denote either $\|\cdot\|_{2,S}$
or $\|\cdot\|_{2,\mu}$, with corresponding diameter
$d\in\{D_2,D_2(\mu)\}$.
Let $\|\cdot\|_{**}$ denote the corresponding norm on $\RR^k$,
namely $\|\cdot\|_2$ or $\|\cdot\|_{\mathrm{sp},\mu}$, so that
$\E\|Z_t\|_{**}^2=t$.
Let $B:=\{a\in A:\|a\|_*<d/2\}$ be the open ball of radius $d/2$
centered at $\mathbf 0_A$. Choose $g\in A$ such that $\|g\|_*=d$.
By the triangle inequality, $B$ and $g+B$ are disjoint, since their centers
are at distance $d$. Hence $\pi_A(B)\leq1/2$.
Since the quotient norm of the projection of $Z_t$ is at most
$\|Z_t\|_{**}$, Markov's inequality gives
\[
\begin{aligned}
d_\mu^{\TV}(t)
&\qrel{\geq}\mu_t(B)-\pi_A(B)
\qrel{\geq}\P(\|Z_t\|_{**}<d/2)-\frac12\qrel{\geq}\frac12-\frac{4t}{d^2}.
\end{aligned}
\]
Applying this to both quotient norms yields
\[
t_{\mix}^{\TV}(\mu)
\qrel{\geq} \frac1{16}\max\{D_2^2,D_2^2(\mu)\}.
\]

We next refine the preceding argument to obtain the lower bound
at thresholds $1-\ep$. For $r>0$, let
$B(r):=\{x\in A:\|x\|_{2,S}\leq r\}$.
Let $x_0=\mathbf 0_A,x_1,\ldots,x_n=g$ be a path with
$\|g\|_{2,S}=D_2$. Since the quotient $\ell^2$ distance
from the identity changes by at most one along each edge,
we may select vertices $x_{i_j}$ along this path, where
$i_j:=\min\{i:\|x_i\|_{2,S}\geq j(2r+2)\}$,
for $0\leq j\leq\lfloor D_2/(2r+2)\rfloor$.
By the choice of $i_j$, we have
$j(2r+2)\leq\|x_{i_j}\|_{2,S}<j(2r+2)+1$.
This gives at least $\lfloor D_2/(2r+2)\rfloor+1$
vertices whose distances from the identity differ
pairwise by more than $2r$. Indeed, for any $j\neq \ell$,
the reverse triangle inequality gives
\[
\|x_{i_j}-x_{i_\ell}\|_{2,S}
\qrel{\geq}
\bigl|\|x_{i_j}\|_{2,S}-\|x_{i_\ell}\|_{2,S}\bigr|
\qrel{>}2r.
\]
Thus, the translates $x_{i_j}+B(r)$ are pairwise disjoint,
and it follows that
\begin{equation}\label{eq:l2_ball_volume_bound}
\pi_A(B(r))\qrel{\leq}\frac{2r+2}{D_2}.
\end{equation}

Put $u:=\log(8/\ep)$. We claim that, for some absolute
constant $C_0>0$,
\begin{equation}\label{eq:l2_lift_concentration}
\P\left(\|Z_t\|_2>C_0(\sqrt{tu}+u)\right)
\qrel{\leq}\frac{\ep}{4}.
\end{equation}
Indeed, the number of jumps $N_t$ is Poisson with mean $t$,
and a standard Poisson tail bound gives
$\P(N_t>2t+4u)\leq e^{-4u}$.
Conditional on $N_t=n$, the walk is a sum of $n$
independent symmetric unit-coordinate steps.
Since $\E[\|Z_t\|_2\mid N_t=n]\leq\sqrt n$
and changing one step changes the norm by at most $2$,
the bounded-differences inequality
\cite{mcdiarmid1989bounded} gives
\[
\P\left(\|Z_t\|_2>\sqrt n+\sqrt{2nu}
\,\middle|\,N_t=n\right)
\qrel{\leq}e^{-u}.
\]
Combining these estimates proves
\eqref{eq:l2_lift_concentration}.

Now suppose $\ep D_2\geq C\log(1/\ep)$, and take
$t=c\ep^2D_2^2/\log(1/\ep)$ and $r=\ep D_2/16$,
with $c$ sufficiently small and $C$ sufficiently large.
These choices ensure that
$C_0(\sqrt{tu}+u)\leq r$ and
$(2r+2)/D_2\leq\ep/4$.
Thus, by \eqref{eq:l2_lift_concentration} and
\eqref{eq:l2_ball_volume_bound},
\[
\P(\|Z_t\|_2\leq r)\qrel{\geq}1-\frac{\ep}{4},
\qquad
\pi_A(B(r))\qrel{\leq}\frac{\ep}{4}.
\]
Since the projection of $Z_t$ onto $A$ has law $\mu_t$,
we obtain
\[
\begin{aligned}
d_\mu^{\TV}(t)
&\qrel{\geq}\mu_t(B(r))-\pi_A(B(r))\qrel{\geq}\P(\|Z_t\|_2\leq r)-\pi_A(B(r))
\qrel{\geq}1-\frac{\ep}{2}.
\end{aligned}
\]
It follows that
\[
t_{\mix}^{\TV}(\mu,1-\ep)
\qrel{\geq}c\,\frac{\ep^2D_2^2}{\log(1/\ep)},
\]
after adjusting $c$ to account for the symmetrization
comparison.

\end{proof}

We now turn to the upper bound in
Theorem~\ref{thm:Geometric bounds on mixing for Abelian}~(i).
We begin by establishing spectral bounds for the $\ell^2$ and
$\ell^\infty$ distances. Recall that $b_\mu(\lambda)$ denotes
the eigenvalue associated with the character $\chi_\lambda$.
By the spectral decomposition, using transitivity;
cf.~\cite[Lemma~12.18(ii)]{wilmer2009markov},
\begin{equation}\label{eq:abelian_l2_spectral_decomposition}
\bigl(d_\mu^{(2)}(t)\bigr)^2
\qrel{=}
\sum_{\lambda\in\mathcal R\setminus\{0\}}
e^{-2t(1-\Re b_\mu(\lambda))}.
\end{equation}

Throughout the remainder of the paper, we write
\begin{equation}\label{eq:covariance_matrix}
\Sigma\qrel{:=}\operatorname{diag}(\mu_1,\ldots,\mu_k)
\quad\text{where}\quad \mu_j\qrel{:=}\mu(\{\pm s_j\}).
\end{equation}
Since $8x^2\leq1-\cos(2\pi x)\leq2\pi^2x^2$
for $x\in[-1/2,1/2)$, \eqref{eq:abelian_eigenvalue_real_part}
gives
\[
8\lambda^{\mathsf T}\Sigma\lambda
\qrel{\leq}1-\Re b_\mu(\lambda)
\qrel{\leq}2\pi^2\lambda^{\mathsf T}\Sigma\lambda.
\]
Consequently,
\begin{equation}\label{eq:abelian_l2_spectral_bound}
\sum_{\lambda\in\mathcal R\setminus\{0\}}
e^{-4\pi^2t\lambda^{\mathsf T}\Sigma\lambda}
\qrel{\leq}
\bigl(d_\mu^{(2)}(t)\bigr)^2
\qrel{\leq}
\sum_{\lambda\in\mathcal R\setminus\{0\}}
e^{-16t\lambda^{\mathsf T}\Sigma\lambda}.
\end{equation}
Moreover, by \eqref{eq:abelian_l2_linf_trace},
$d_\mu^{(\infty)}(t)\leq(d_\mu^{(2)}(t/2))^2$.
Thus,
\begin{equation}\label{eq:abelian_linf_spectral_bound}
d_\mu^{(\infty)}(t)
\qrel{\leq}
\sum_{\lambda\in\mathcal R\setminus\{0\}}
e^{-8t\lambda^{\mathsf T}\Sigma\lambda}.
\end{equation}

\begin{prop}\label{geometric_bound_infty}
For every increment measure $\mu$ in the Abelian setup,
\[
t_{\mix}^{\ell^\infty}(\mu)
\qrel{\lesssim} k\min\{t_{\rel}(\mu),D_2^2(\mu)\}.
\]
More precisely, for every $t>0$,
$
d_\mu^{(\infty)}(t)
\leq\exp\!\left(\frac{\pi^2mD_2^2(\mu)}{8t}\right)-1.
$
\end{prop}

\begin{proof}
By \eqref{eq:abelian_l2_linf_trace} and the invariance
of $D_2(\mu)$ under symmetrization, it suffices to
assume that $\mu$ is symmetric.

For a countable set $\mathcal L\subseteq\RR^k$, we introduce the following pair of dual Gaussian sums:
\[
F_\Sigma(\mathcal L)
\qrel{:=}\sum_{v\in\mathcal L}e^{-8t v^{\mathsf T}\Sigma v},
\qquad\text{and}\qquad
H_\Sigma(\mathcal L)
\qrel{:=}\sum_{v\in\mathcal L}
e^{-\pi^2v^{\mathsf T}\Sigma^{-1}v/(8t)}.
\]
By symmetry of $\ZZ^k$, for each $\lambda\in\mathcal R$,
\[
F_\Sigma(\ZZ^k+\lambda)
\qrel{=}
e^{-8t\lambda^{\mathsf T}\Sigma\lambda}
\sum_{z\in\ZZ^k}e^{-8tz^{\mathsf T}\Sigma z}
\cosh(16t z^{\mathsf T}\Sigma\lambda)
\qrel{\geq}
e^{-8t\lambda^{\mathsf T}\Sigma\lambda}F_\Sigma(\ZZ^k),
\]
where we used $\cosh(x)\geq1$. Since $\Lambda^*=\bigsqcup_{\lambda\in\mathcal R}
(\ZZ^k+\lambda)$,
applying \eqref{eq:abelian_linf_spectral_bound}, we obtain
\begin{equation}\label{eq:abelian_gaussian_sum_comparison}
\frac{F_\Sigma(\Lambda^*)}{F_\Sigma(\ZZ^k)}
\qrel{=}
\sum_{\lambda\in\mathcal R}
\frac{F_\Sigma(\ZZ^k+\lambda)}{F_\Sigma(\ZZ^k)}
\qrel{\geq}
\sum_{\lambda\in\mathcal R}
e^{-8t\lambda^{\mathsf T}\Sigma\lambda}
\qrel{\geq}
1+d_\mu^{(\infty)}(t).
\end{equation}

For either lattice $\mathcal L\in\{\Lambda^*,\ZZ^k\}$,
Poisson summation with Fourier kernel
$e^{-2\pi i\langle x,\xi\rangle}$
(see \cite[Theorem~4.16 and Lemma~4.18]{rothvoss2023lattices})
gives
\[
F_\Sigma(\mathcal L)
\qrel{=}
\frac{(\pi/(8t))^{k/2}}
{\det(\mathcal L)\sqrt{\det\Sigma}}
H_\Sigma(\mathcal L^*).
\]
Since $\det(\Lambda^*)=|A|^{-1}$, it follows that
\begin{equation}\label{eq:abelian_poisson_sum_ratio}
\frac{F_\Sigma(\Lambda^*)}{F_\Sigma(\ZZ^k)}
\qrel{=}
|A|\frac{H_\Sigma(\Lambda)}{H_\Sigma(\ZZ^k)}.
\end{equation}

For each coset $x+\Lambda\in\ZZ^k/\Lambda$, choose the
representative $x$ minimizing
$\sum_{j=1}^k x_j^2/\mu(s_j)$. Since $\mu_j\geq\mu(s_j)$,
we have
$
x^{\mathsf T}\Sigma^{-1}x
\leq\sum_{j=1}^k\frac{x_j^2}{\mu(s_j)}
\leq mD_2^2(\mu).
$
Writing $a=\sqrt{\pi/(8t)}$ and applying
Lemma~\ref{shifted_Gaussian_mass}, we obtain
\[
\begin{aligned}
H_\Sigma(\Lambda+x)
&\qrel{=}
\rho_{1/a}(\Sigma^{-1/2}\Lambda+\Sigma^{-1/2}x)\\
&\qrel{\geq}
e^{-\pi a^2\dist(\Sigma^{-1/2}x,\Sigma^{-1/2}\Lambda)^2}
\rho_{1/a}(\Sigma^{-1/2}\Lambda)\\
&\qrel{\geq}
e^{-\pi a^2 x^{\mathsf T}\Sigma^{-1}x}
H_\Sigma(\Lambda)\qrel{\geq}
e^{-\frac{\pi^2mD_2^2(\mu)}{8t}}H_\Sigma(\Lambda).
\end{aligned}
\]
Summing over the cosets of $\Lambda$ in $\ZZ^k$ gives
\[
H_\Sigma(\ZZ^k)
\qrel{\geq}
|A|e^{-\frac{\pi^2mD_2^2(\mu)}{8t}}H_\Sigma(\Lambda).
\]
Combining this with \eqref{eq:abelian_gaussian_sum_comparison}
and \eqref{eq:abelian_poisson_sum_ratio} yields
$
d_\mu^{(\infty)}(t)
\leq
\exp\left(\frac{\pi^2mD_2^2(\mu)}{8t}\right)-1.
$
Since $m\leq2k$, this proves $t_{\mix}^{\ell^\infty}(\mu)
\lesssim kD_2^2(\mu)$.

\medskip

For the upper bound in terms of relaxation time, consider the
weighted norm
\[
\|v\|_{2,\Sigma}^2\qrel{:=}v^{\mathsf T}\Sigma v
\qrel{=}\sum_{j=1}^k\mu_jv_j^2.
\]
Let $\lambda_*\in\mathcal R\setminus\{0\}$ be a shortest
nonzero vector in this norm, write
$\delta:=\|\lambda_*\|_{2,\Sigma}$, and define
\[
N(s)\qrel{:=}\bigl|\{\lambda\in\mathcal R\setminus\{0\}:
\|\lambda\|_{2,\Sigma}\leq s\delta\}\bigr|,
\qquad s\geq1.
\]

We first bound $N(s)$ by a packing argument.
The representatives in $\mathcal R$ are $\delta$-separated
under $\|\cdot\|_{2,\Sigma}$. Indeed, for distinct
$\lambda,\lambda'\in\mathcal R$, let $u\in\mathcal R$
be the nonzero representative of $\lambda-\lambda'$
modulo $\ZZ^k$. Since $|u_j|\leq|\lambda_j-\lambda_j'|$
by coordinatewise reduction,
\[
\|\lambda-\lambda'\|_{2,\Sigma}
\qrel{\geq}
\|u\|_{2,\Sigma}
\qrel{\geq}\delta.
\]
For $r>0$, write
$B_\Sigma(r):=\{v\in\RR^k:\|v\|_{2,\Sigma}<r\}$.
Thus the balls of radius $\delta/2$ centered at the
$N(s)$ representatives are pairwise disjoint and, by
the triangle inequality, contained in
$B_\Sigma((s+1/2)\delta)$.
Comparing volumes gives
\[
N(s)
\qrel{\leq}
\frac{\operatorname{Vol}(B_\Sigma((s+1/2)\delta))}
{\operatorname{Vol}(B_\Sigma(\delta/2))}
\qrel{=}(1+2s)^k.
\]
Cf.~\cite[Lemma~2.6 and the proof of Theorem~1.1]{hough2017mixing}
for a sharper counting estimate and its use in the analogous spectral argument
(see Remark~\ref{rmk:abelian_geometric_proofs}\textup{(i)} below).

Now set $\tau:=8t\delta^2$. Recalling that
$\lambda^{\mathsf T}\Sigma\lambda
=\|\lambda\|_{2,\Sigma}^2
=\sum_{j=1}^k\mu_j\lambda_j^2$, by
\eqref{eq:abelian_linf_spectral_bound} and integration
by parts,
\[
\begin{aligned}
d_\mu^{(\infty)}(t)
&\qrel{\leq}
\sum_{\lambda\in\mathcal R\setminus\{0\}}
e^{-8t\|\lambda\|_{2,\Sigma}^2}\\
&\qrel{=}
\int_{1^-}^\infty e^{-\tau s^2}\,dN(s)\qrel{=}
2\tau\int_1^\infty s e^{-\tau s^2}N(s)\,ds\\
&\qrel{\leq}
2\tau\int_1^\infty
s\exp\left(-\tau s^2+k\log(1+2s)\right)ds.
\end{aligned}
\]
Since $\log(1+2s)\leq s^2\log3$ for $s\geq1$,
whenever $\tau\geq2k\log3$, the last integral is at most
\[
2\tau\int_1^\infty
s e^{-(\tau-k\log3)s^2}\,ds
\qrel{=}
\frac{\tau}{\tau-k\log3}
e^{-(\tau-k\log3)}
\qrel{\leq}2e^{-\tau/2}.
\]
Taking $\tau=Ck$ for a sufficiently large absolute
constant $C$ gives
$t_{\mix}^{\ell^\infty}(\mu)\lesssim k\delta^{-2}$.
Finally, Lemma~\ref{lem:relaxation_time_asymp_weighted}
gives $t_{\rel}(\mu)\asymp \delta^{-2}$, completing
the proof.
\end{proof}

\begin{proof}[Proof of Theorem~\ref{thm:Geometric bounds on mixing for Abelian}\textup{(i)}]
Since the transition kernel on an Abelian group is normal,
\eqref{eq:decay_rate_TV} and
\eqref{eq:equality_relaxation_times} imply
$t_{\mix}^{\TV}(\mu)\gtrsim t_{\rel}(\mu)$.
Together with Lemma~\ref{geometric_bound_TV}, this
establishes the lower bounds. The upper bounds follow from
Proposition~\ref{geometric_bound_infty}, using
$D_2^2(\mu)\leq D_2^2/(m\mu_*)$ and $k\leq m$.
\end{proof}

\begin{proof}[Proof of Corollary~
\ref{thm: Geometric bounds on mixing for nilpotent}]
Under \eqref{eq:general_DC},
Theorem~\ref{thm: mixing and spectral comparison with Abelianization}
gives
\[
t_{\rel}(\mu)\qrel{=}t_{\rel}(\mu_\ab),
\qquad
t_{\mix}^{\ell^p}(\mu)
\qrel{\asymp}t_{\mix}^{\ell^p}(\mu_\ab),
\quad p\in\{1,\infty\}.
\]
Since $D_2$ and $D_2(\mu)$ are defined in terms of
the projected walk on $G_\ab$,
\eqref{eq:nilpotent_geometric_mixing_bounds}
follows from
Theorem~\ref{thm:Geometric bounds on mixing for Abelian}\textup{(i)}. To prove \eqref{eq:nilpotent_TV_symmetrization},
we combine
Theorem~\ref{thm:Geometric bounds on mixing for Abelian}\textup{(ii)}
with
Theorem~\ref{thm: mixing and spectral comparison with Abelianization},
applied to both $\mu$ and $\mu_\sym$,
using intermediate threshold shifts of $\delta/3$.
\end{proof}

\begin{remark}[Related upper bounds for Abelian walks and diffusions]
\label{rmk:abelian_geometric_proofs}
We discuss earlier upper bounds whose statements or proofs are
closely related to those of Proposition~\ref{geometric_bound_infty}.

\smallskip\noindent
\textup{(i)} For the uniform lazy symmetric random walk on $\ZZ_p$, where $p$ is prime,
with generating set $S=\{0,\pm s_1,\ldots,\pm s_d\}$
of size $2d+1$ and $d\leq\log p/\log\log p$,
Hough \cite[Theorems~1.1 and~1.4]{hough2017mixing}
proved analogous total-variation mixing-time upper bounds as in Theorem~\ref{thm:Geometric bounds on mixing for Abelian}~(i).

His analysis exploits the connection between
random walks on Abelian groups and lattice theory. We use this connection
to obtain the upper bounds of
Theorem~\ref{thm:Geometric bounds on mixing for Abelian}~(i) for general
Abelian groups without a degree restriction. In particular, our proof of
$t_{\mix}^{\ell^\infty}(\mu)\lesssim k\,t_{\rel}(\mu)$
provides a weighted, discrete analogue of Hough's packing argument for
Brownian motion modulo a lattice
\cite[proof of Theorem~1.1, especially (3.1)]{hough2017mixing}.

\smallskip\noindent
\textup{(ii)}
The bound $t_{\mix}^{\ell^\infty}(\mu)\lesssim kD_2^2(\mu)$
has an unweighted continuous analogue for Brownian motion on
$\RR^k/\Lambda$. For Brownian motion $B$ on $\RR^k$ with
covariance matrix $(t/k)I_k$ at time $t$, the corresponding continuous bound is
\[
t_{\mix}^{\ell^\infty}(B+\Lambda)
\qrel{\lesssim} k\operatorname{CR}(\Lambda)^2
\quad\text{where}\quad
\operatorname{CR}(\Lambda)
\qrel{:=}\sup_{x\in\RR^k}\dist(x,\Lambda)
\]
is the \emph{covering radius} of $\Lambda$.\footnote{
This bound also follows from the smoothing-parameter
estimate in \cite[Lemma~15]{alamati2018proofsystems} and a short calculation.
The bound $t_{\mix}^{\ell^\infty}(B+\Lambda)\lesssim k\,t_{\rel}(B+\Lambda)$ also follows
from \cite[Lemma~3.2]{micciancio2007worst} by the same short calculation.}
When all inverse pairs have cardinality $2$, $B+\Lambda$
is the continuous counterpart of simple random walk on $\ZZ^k/\Lambda$.
Both the discrete and continuous diameter bounds rely on
the comparison of Gaussian masses of shifted lattices
(Lemma~\ref{shifted_Gaussian_mass}).

\smallskip\noindent
\textup{(iii)}
A related upper bound in terms of the usual graph diameter $D$ follows from
the more general result of M{\"u}nch and the third-named author
\cite[Corollary~1]{munch2023mixing} for non-negatively curved Markov chains
(including random walks on Abelian groups). For symmetric $\mu$, their result yields $t_{\mix}^{\TV}(\mu)\lesssim D^2/\mu_*$.

\end{remark}

\subsubsection{Sharpness of geometric bounds}
\label{sec:sharpness of geometric bounds}

Let $\Gamma:=\Cay(A,S)$ with $S$ symmetric, and consider
the simple random walk on $\Gamma$. Theorem~\ref{thm:Geometric bounds on mixing for Abelian}~(i) gives
\[\max \left\{t_{\mathrm{rel}}(\Gamma),D_2^2\right\}
\qrel{\lesssim}
t_{\mix}^{\mathrm{TV}}(\Gamma)
\qrel{\lesssim}
|S|\min\left\{
t_{\mathrm{rel}}(\Gamma),
D_2^2
\right\}.
\]
We give examples illustrating the sharpness of the different terms in these bounds.

\medskip
\noindent
\textbf{Examples.}
\begin{enumerate}[label=(\alph*)]
\item
Let $A_0$ be any finite Abelian group of order $h$, and consider
\[A\qrel{=}A_0\oplus\ZZ_n
\]
with generating set
$S=\{ \pm (s,0),\pm(\id,1):s\in A_0\}$,
where $\id$ denotes the identity element of $A_0$.

\item
Let $A=\ZZ_2^{k-1}\oplus\ZZ_n$ with generating set
$S=\{\pm e_i:i\in[k]\}$. Suppose that $n^2\gtrsim k$.

\item
Let $A$ be a finite Abelian group with generating set
$S=\{\pm Z_i:i\in[k]\}$, where
$\{Z_i\}_{i\in[k]}\overset{\mathrm{iid}}{\sim}\Unif(A)$.
Suppose that $k-r(A)\asymp\log|A|$.
\end{enumerate}

Example~(a) gives a product construction with an arbitrary finite Abelian
factor. Example~(b) realizes the same scale using only coordinate generators.
In Examples~(a) and~(b),
$
t_{\mathrm{rel}}(\Gamma)
\asymp
t_{\mathrm{mix}}^{\mathrm{TV}}(\Gamma)
\asymp
|S|D_2^2
\asymp
kn^2.
$
Indeed, in Example~(b),
$
D_2^2\asymp k+n^2\asymp n^2,
$
where the last comparison follows from $n^2\gtrsim k$.

In Example~(c), results from~\cite{hermon2021cutoff} imply that, with high probability as $|A|\to\infty$,
$|S|\asymp\log|A|\asymp t_{\mathrm{mix}}^{\mathrm{TV}}(\Gamma)$
and $t_{\mathrm{rel}}(\Gamma)\asymp1$.
Theorem~\ref{thm:Geometric bounds on mixing for Abelian}~(i) gives
$D_2^2\lesssim\log|A|$, while Lemma~\ref{diam_abelianization}
gives the reverse bound, since the number of distinct inverse pairs is
$k\asymp\log|A|$ with high probability. Consequently, with high probability,
\[D_2^2\qrel{\asymp}\log|A|.
\]
Thus, in Example~(c),
\[|S|t_{\mathrm{rel}}(\Gamma)
\qrel{\asymp}
t_{\mathrm{mix}}^{\mathrm{TV}}(\Gamma)
\qrel{\asymp}
D_2^2.
\]

\subsection{Gaussian approximation and symmetrization}
\label{sec:comparison_nonsymmetric_rw}

For a  positive definite covariance matrix $\Sigma\in \mathbb{R}^{k\times k}$, let
$\varphi_\Sigma$ denote the density of
$\operatorname{Normal}_k(0,\Sigma)$.
Its periodization over $\Lambda$ is defined by
\[
\Theta(x,\Sigma;\Lambda)
\qrel{:=}\sum_{z\in\Lambda}\varphi_\Sigma(x+z).
\]
By the Poisson summation formula (see, e.g.,
\cite[Section~4.1.4]{rothvoss2023lattices}),
it admits the Fourier expansion
\begin{equation}\label{frequency_representation}
\Theta(x,\Sigma;\Lambda)
\qrel{=}\frac{1}{\det(\Lambda)}
\sum_{\lambda\in\Lambda^*}
e^{-2\pi^2\lambda^{\mathsf T}\Sigma\lambda}
e^{2\pi i\langle\lambda,x\rangle}.
\end{equation}
The Lebesgue measure of $\RR^k/\Lambda$ is $\det(\Lambda)$.
Normalizing it to have total mass one gives the Haar probability
measure, denoted by $\pi_{\RR^k/\Lambda}$.

Recall the covariance matrix
$\Sigma=\operatorname{diag}(\mu_1,\ldots,\mu_k)$
from \eqref{eq:covariance_matrix}.
Let $B^\Sigma$ denote the projection modulo $\Lambda$ of Brownian motion
on $\RR^k$, started at the origin with covariance matrix $t\Sigma$ at time $t$.
For $t>0$, the law of $B_t^\Sigma$ has density
$\Theta(\cdot,t\Sigma;\Lambda)$ with respect to Lebesgue measure
on $\RR^k/\Lambda$.
Its stationary distribution is the normalized Haar measure
$\pi_{\RR^k/\Lambda}$, and we write
\[
d_{B^\Sigma}^{\TV}(t)
\qrel{:=}\left\|
\P(B_t^\Sigma\in\cdot)-\pi_{\RR^k/\Lambda}
\right\|_{\TV}
\]
for its total variation distance from stationarity.

To approximate the random walk on $\ZZ^k/\Lambda$ by projected Brownian motion,
we will use a Gaussian approximation based on the following
one-dimensional local limit estimate, whose proof is given in
Appendix~\ref{app:biased_local_limit}.

In the following lemma, $\|h\|_{L^1(\RR)}:=\int_{\RR}|h(x)|\,dx$
denotes the $L^1$ norm with respect to Lebesgue measure.

\begin{lemma}[Local limit theorem for a biased walk on $\ZZ$]
\label{LLT_biased_walk}
Let $(S_t)_{t\geq 0}$ be the rate-$1$ continuous-time biased random walk
on $\ZZ$, started at $0$, that, at each jump, moves one step to the right with probability
$p$ and one step to the left with probability $1-p$. Let $\nu_t^p$
denote the law of $S_t$, and let
\[\phi_t(x)
\qrel{:=}
\frac{1}{\sqrt{2\pi t}}
\exp\left(-\frac{x^2}{2t}\right).
\]
Then there exists a universal constant $C<\infty$ such that, for every
$p\in[0,1]$ and $t\geq1$,
\[\left\|
\nu_t^p*\1_{[-1/2,1/2)}
-
\phi_t\bigl(\cdot-(2p-1)t\bigr)
\right\|_{L^1(\RR)}
\qrel{\leq}
Ct^{-1/2}.
\]
\end{lemma}

\begin{lemma}[Gaussian approximation]
\label{lem:weighted_gaussian_approximation}
For every $t>0$,
\begin{equation}\label{eq:mu_gaussian_comparison}
\left|d_\mu^{\TV}(t)-d_{B^\Sigma}^{\TV}(t)\right|
\qrel{\lesssim}
\sum_{i=1}^k(\mu_it)^{-1/2}.
\end{equation}
Since $\Sigma$ depends on $\mu$ only through $\mu_\sym$, this also gives
$|d_\mu^{\TV}(t)-d_{\mu_\sym}^{\TV}(t)|\lesssim\sum_{i=1}^k(\mu_it)^{-1/2}$.
\end{lemma}

\begin{proof}
If $\mu_it<1$ for some $i$, the claimed estimate is immediate:
its left-hand side is at most $1$, whereas
$\sum_i(\mu_it)^{-1/2}>1$.
We may therefore assume that $\mu_it\geq1$ for every $i$.
Let $Z_t\in\ZZ^k$ have independent coordinates, with the
$i$th coordinate jumping at rate $\mu_i$, to the right
with probability
$p_i:=\mu(s_i)/(\mu(s_i)+\mu(-s_i))$ and to the left
otherwise. Then the projected process $Z_t+\Lambda$ on $\ZZ^k/\Lambda$
has the same law as the random walk driven by $\mu$. This remains true when $s_i=-s_i$,
since both directions project to the same increment
with combined rate $\mu_i$.

Set $b_i:=\mu_i(2p_i-1)$ for $i\in[k]$ and
$\mathbf b:=(b_1,\ldots,b_k)$.
Let $U\sim\Unif([-1/2,1/2)^k)$ be independent of $Z_t$.
By contraction of total variation under projection onto
$\RR^k/\Lambda$, tensorization, and
Lemma~\ref{LLT_biased_walk}, we obtain
\begin{align*}
&\left\|\P(Z_t+U+\Lambda\in\cdot)
-\P(B_t^\Sigma+t\mathbf b\in\cdot)\right\|_{\TV}\\
&\quad\qrel{\leq}
\left\|\P(Z_t+U\in\cdot)
-\operatorname{Normal}_k(t\mathbf b,t\Sigma)\right\|_{\TV}
\qrel{\lesssim}\sum_{i=1}^k(\mu_it)^{-1/2}.
\end{align*}
Since $\Lambda\subseteq\ZZ^k$, adding $U$ preserves
the total variation distance from stationarity, so
$
d_\mu^{\TV}(t)
=\|\P(Z_t+U+\Lambda\in\cdot)
-\pi_{\RR^k/\Lambda}\|_{\TV}.
$
Moreover, translation invariance of Haar measure gives
$
\|\P(B_t^\Sigma+t\mathbf b\in\cdot)
-\pi_{\RR^k/\Lambda}\|_{\TV}
=d_{B^\Sigma}^{\TV}(t).
$
The triangle inequality therefore yields
\[
\left|d_\mu^{\TV}(t)-d_{B^\Sigma}^{\TV}(t)\right|
\qrel{\lesssim}
\sum_{i=1}^k(\mu_it)^{-1/2}.
\]

Finally, applying the same argument to the symmetrized walk $\mu_\sym$ yields
$|d_{\mu_\sym}^{\TV}(t)-d_{B^\Sigma}^{\TV}(t)| \lesssim
\sum_{i=1}^k(\mu_it)^{-1/2}$, and using the triangle inequality with \eqref{eq:mu_gaussian_comparison} proves the second assertion.
\end{proof}

\medskip
\begin{proof}[Proof of Theorem~
\ref{thm:Geometric bounds on mixing for Abelian}\textup{(ii)}]
The mixing time comparisons from
\eqref{eq:unconditional_symmetrization_mixing_times}
were established in Section \ref{subsec:mixing_times_symmetrization}. It remains to prove
\eqref{eq:TV_symmetrization_mixing_times} under
\eqref{eq:stronger_comparison_condition}.

Fix $\ep\in(0,1)$ and
$0<\delta<\min\{\ep,1-\ep\}$, and set
$
t_\pm:=t_{\mix}^{\TV}(\mu_\sym,\ep\pm\delta/2).
$
Both times are at least
$t_{\mix}^{\TV}(\mu_\sym,1-\delta/2)$. Since \eqref{eq:stronger_comparison_condition} implies
$k\leq\log|A|$ and $|A|^{1/k}\geq(\log|A|)^2$,
Lemma~\ref{diam_abelianization} gives
$D_2^2\gtrsim k|A|^{2/k}\to\infty$. Thus \eqref{eq:geometric_high_threshold_lower_bound} yields
$t_\pm\gtrsim_\delta D_2^2\gtrsim k|A|^{2/k}$.
Since each $\mu_i\geq\mu_*$,
Lemma~\ref{lem:weighted_gaussian_approximation} gives
\begin{align*}
\left|d_\mu^{\TV}(t_\pm)
-d_{\mu_\sym}^{\TV}(t_\pm)\right|
&\qrel{\lesssim}\sum_{i=1}^k(\mu_it_\pm)^{-1/2}\qrel{\lesssim_\delta}
\left(\frac{\mu_*|A|^{2/k}}{k}\right)^{-1/2}.
\end{align*}

It remains to show that this error tends to zero.
Indeed, \eqref{eq:stronger_comparison_condition} gives
$|A|^{1/k}\geq(\log|A|)^2/(m\mu_*)$.
Using $m\mu_*\leq1$, $m\leq2k$, and $k\leq\log|A|$,
we obtain
\[
\frac{\mu_*|A|^{2/k}}{k}
\qrel{\geq}\frac{(\log|A|)^4}{km^2\mu_*}
\qrel{\geq}\frac{(\log|A|)^4}{km}
\qrel{\geq}\frac{(\log|A|)^2}{2}.
\]
Consequently, the Gaussian approximation error at
$t_\pm$ is $O_\delta((\log|A|)^{-1})$.
For all sufficiently large $|A|$, depending only on
$\delta$, this error is less than $\delta/2$.
Since $t_{\mix}^{\TV}(\mu_\sym,\cdot)$ is nonincreasing,
we conclude that
\[
t_{\mix}^{\TV}(\mu_\sym,\ep+\delta)
\qrel{\leq} t_+
\qrel{\leq} t_{\mix}^{\TV}(\mu,\ep)
\qrel{\leq} t_-
\qrel{\leq} t_{\mix}^{\TV}(\mu_\sym,\ep-\delta),
\]
which proves \eqref{eq:TV_symmetrization_mixing_times}.
\end{proof}

\section{Cutoff and consequences}\label{sec:weighted_RW}\label{sec:cutoff_abelianization}
We retain the notation of Section~\ref{s:Abelian  mixing}.
In particular, recall the coordinate rates
$\mu_i:=\mu(\{\pm s_i\})$, the covariance matrix
$\Sigma:=\operatorname{diag}(\mu_1,\ldots,\mu_k)$,
and the projected Brownian motion $B^\Sigma$ on
$\RR^k/\Lambda$.

\subsection{Cutoff under the product condition}\label{sec:abelian_cutoff}

We begin with a cutoff criterion for random walks on finite
Abelian groups, which, together with
Theorem~\ref{thm: mixing and spectral comparison with Abelianization},
yields the corresponding criterion for nilpotent groups
in Theorem~\ref{thm: cutoff reduction to Abelianization}.

\begin{prop}[Cutoff criterion for Abelian walks]\label{prop:cutoff_lattice_product_condition}
Suppose that \eqref{eq:stronger_comparison_condition} holds, or, more
generally, that
\begin{equation}\label{eq:lattice_large_time_condition}
t_{\mix}^{\TV}(\mu_\sym)\qrel{\gg} \left(\sum_{i=1}^k\mu_i^{-1/2}\right)^2.
\end{equation}
Then the walk driven by $\mu$ exhibits total variation cutoff if and only if the product condition holds, namely,
$t_{\mix}^{\TV}(\mu)\gg t_{\rel}(\mu)$.
\end{prop}

\begin{proof}
The proof of Theorem~\ref{thm:Geometric bounds on mixing for Abelian}~(ii)
shows that \eqref{eq:stronger_comparison_condition} implies
\eqref{eq:lattice_large_time_condition}.
It therefore suffices to prove the result under the latter condition.

We first compare the mixing times of $\mu$ and $\mu_\sym$
with that of $B^\Sigma$. By Lemma~\ref{lem:weighted_gaussian_approximation}
and the assumption \eqref{eq:lattice_large_time_condition}, whenever
$t\asymp t_{\mix}^{\TV}(\mu_\sym)$,
we have, for $\nu\in\{\mu,\mu_\sym\}$,
\[
\left|d_\nu^{\TV}(t)-d_{B^\Sigma}^{\TV}(t)\right|
\qrel{\lesssim}\sum_{i=1}^k(\mu_it)^{-1/2}
=o(1).
\]
By the standard comparison of total variation mixing times at fixed
thresholds below $1/2$
(see \cite[Section~4.5]{wilmer2009markov}),
we have
$t_{\mix}^{\TV}(\mu_\sym,\alpha)
\asymp t_{\mix}^{\TV}(\mu_\sym)$
for $\alpha\in\{1/8,3/8\}$.
Applying the Gaussian approximation at these two mixing times
therefore gives
\[
t_{\mix}^{\TV}(\mu_\sym,3/8)
\qrel{\leq} t_{\mix}^{\TV}(\nu)
\qrel{\leq} t_{\mix}^{\TV}(\mu_\sym,1/8),
\qquad \nu\in\{\mu,B^\Sigma\},
\]
for all sufficiently large $|A|$.
Consequently,
\begin{equation}\label{eq:weighted_mixing_BM}
t_{\mix}^{\TV}(\mu)
\qrel{\asymp} t_{\mix}^{\TV}(\mu_\sym)
\qrel{\asymp} t_{\mix}^{\TV}(B^\Sigma).
\end{equation}

Next we compare the relaxation times. Recall that
$\mathcal R:=\Lambda^*\cap[-1/2,1/2)^k$
indexes the characters of $\ZZ^k/\Lambda$.
By Lemma~\ref{lem:relaxation_time_asymp_weighted},
$
t_{\rel}(\mu)^{-1}
\asymp
\min_{\lambda\in\mathcal R\setminus\{0\}}
\sum_{i=1}^k\mu_i\lambda_i^2.
$
On the other hand, the Fourier expansion
\eqref{frequency_representation} gives
\begin{equation}\label{eq:BM_relaxation_time}
t_{\rel}(B^\Sigma)^{-1}
\qrel{=}
2\pi^2\min_{\lambda\in\Lambda^*\setminus\{0\}}
\sum_{i=1}^k\mu_i\lambda_i^2.
\end{equation}
Since $\ZZ^k\subseteq\Lambda^*$ and reducing each coordinate
modulo $\ZZ$ does not increase
$\sum_{i=1}^k\mu_i\lambda_i^2$,
the minimum in \eqref{eq:BM_relaxation_time}
is attained either in $\mathcal R\setminus\{0\}$
or in $\ZZ^k\setminus\{0\}$.
As
$
\min_{\lambda\in\ZZ^k\setminus\{0\}}
\sum_{i=1}^k\mu_i\lambda_i^2
=\min_{i\in[k]}\mu_i,
$
we conclude that
\begin{equation}\label{eq:weighted_relaxation_comparison}
t_{\rel}(B^\Sigma)
\qrel{\asymp}
\max\left\{
t_{\rel}(\mu),\frac1{\min_{i\in[k]}\mu_i}
\right\}.
\end{equation}
Under the assumption
\eqref{eq:lattice_large_time_condition}, we have
\[
\frac{1}{\min_{i\in[k]}\mu_i}
\qrel{\leq}
\left(\sum_{i=1}^k\mu_i^{-1/2}\right)^2
\qrel{\ll}t_{\mix}^{\TV}(\mu_\sym)
\qrel{\asymp} t_{\mix}^{\TV}(\mu).
\]
Together with \eqref{eq:weighted_mixing_BM} and
\eqref{eq:weighted_relaxation_comparison},
this shows that the product conditions for $\mu$
and $B^\Sigma$ are equivalent.

Suppose now that $t_{\mix}^{\TV}(\mu)\gg t_{\rel}(\mu)$,
or equivalently, by the preceding argument,
$t_{\mix}^{\TV}(B^\Sigma)\gg t_{\rel}(B^\Sigma)$.
The change of coordinates
$x\mapsto\Sigma^{-1/2}x$ identifies $B^\Sigma$ with
standard Brownian motion on the flat torus
$\RR^k/(\Sigma^{-1/2}\Lambda)$, which has nonnegative
Bakry--\'Emery curvature. Thus the criterion in
\cite[Corollary~1]{salez2025cutoff} implies that
$B^\Sigma$ exhibits total variation cutoff.

It remains to transfer this cutoff to the discrete walks. Since $B^\Sigma$ exhibits cutoff,
\eqref{eq:weighted_mixing_BM} implies that for every fixed $\ep\in(0,1)$
\[
t_{\mix}^{\TV}(B^\Sigma,\ep)
\qrel{\asymp} t_{\mix}^{\TV}(\mu_\sym).
\]
Hence, for $\nu\in\{\mu,\mu_\sym\}$, $|d_\nu^{\TV}(t)-d_{B^\Sigma}^{\TV}(t)|=o(1)$
at these times. For every fixed $\ep\in(0,1)$ and
$0<\delta<\min\{\ep,1-\ep\}$,
Lemma~\ref{lem:weighted_gaussian_approximation} gives,
for $\nu\in\{\mu,\mu_\sym\}$,
\begin{equation}\label{eq:BM_TV_mixing_comparison}
t_{\mix}^{\TV}(B^\Sigma,\ep+\delta)
\qrel{\leq} t_{\mix}^{\TV}(\nu,\ep)
\qrel{\leq} t_{\mix}^{\TV}(B^\Sigma,\ep-\delta)
\end{equation}
for all sufficiently large $|A|$.
Since $B^\Sigma$ exhibits cutoff, these comparisons show
that both discrete walks exhibit cutoff.

Conversely, if the walk driven by $\mu$ exhibits cutoff,
the necessary product condition
\eqref{eq:necessity_product_condition} gives
$t_{\mix}^{\TV}(\mu)\gg t_{\rel}^{\ev}(\mu)$.
Since $A$ is Abelian, its transition operator is normal,
and \eqref{eq:equality_relaxation_times} yields
$t_{\rel}^{\ev}(\mu)=t_{\rel}(\mu)$.
This completes the proof.
\end{proof}

\begin{proof}[Proof of Theorem~\ref{thm: cutoff reduction to Abelianization}]
Under \eqref{eq:general_DC}, condition
\eqref{eq:stronger_comparison_condition} holds for
$A=G_\ab$ and the projected measure $\mu_\ab$,
since $k\leq m$ and
$\log\log|G_\ab|\leq\max\{\log\log|G|,L\}$.
Moreover, since $m\geq1$ and $m\mu_*\leq1$,
\eqref{eq:general_DC} implies
$
|G_\ab|\geq(\log|G|)^{5},
$
so in particular $|G_\ab|\to\infty$ as $|G|\to\infty$.
Thus Proposition~\ref{prop:cutoff_lattice_product_condition},
applied to $\mu_\ab$, gives the equivalence of
\textup{(ii)} and \textup{(iii)}.

The mixing-time comparisons in
Theorem~\ref{thm: mixing and spectral comparison with Abelianization}
give the equivalence of \textup{(i)} and \textup{(ii)},
as well as
$t_{\mix}^{\TV}(\mu)\asymp t_{\mix}^{\TV}(\mu_\ab)$.
The latter comparison also gives the equivalence of
\textup{(iii)} and \textup{(iv)}, completing the proof.
\end{proof}

\subsection{Further cutoff criteria}

We conclude by justifying the statements in
Remarks~\ref{rmk:mixing_window} and~\ref{rmk:large_mixing}.

\begin{proof}[Justification of Remark~\ref{rmk:mixing_window}]
Fix $\varepsilon\in(0,1/2)$ and suppress the sequence index.
Since $B^\Sigma$ is Brownian motion on a flat torus,
the window bound in \cite[Theorem~1]{salez2025cutoff},
together with
$t_{\rel}(B^\Sigma)\lesssim t_{\mix}^{\TV}(B^\Sigma)$,
gives
\[
t_{\mix}^{\TV}(B^\Sigma,\varepsilon/4)
-t_{\mix}^{\TV}(B^\Sigma,1-\varepsilon/4)
\qrel{\lesssim_\varepsilon}
\sqrt{t_{\rel}(B^\Sigma)t_{\mix}^{\TV}(B^\Sigma)}.
\]

Under \eqref{eq:general_DC},
condition \eqref{eq:stronger_comparison_condition} holds
for $A=G_\ab$ and the projected measure $\mu_\ab$.
In particular,
$|A|^{1/k}\geq(\log|A|)^2/(m\mu_*)$.
By Lemma~\ref{diam_abelianization}, and using
$m\mu_*\leq1$ and $m\leq2k\leq2\log|A|$,
\[
\frac{k^2}{\mu_*D_2^2}
\lesssim \frac{k}{\mu_*|A|^{2/k}}
\leq \frac{km^2\mu_*}{(\log|A|)^4}
\leq \frac{km}{(\log|A|)^4}
\lesssim \frac{1}{(\log|A|)^2}=o(1).
\]
Since Theorem~\ref{thm:Geometric bounds on mixing for Abelian}
\textup{(i)} gives
$t_{\rel}(\mu_\ab)\gtrsim D_2^2/k$,
we have $\mu_*t_{\rel}(\mu_\ab)\to\infty$.
Thus \eqref{eq:weighted_relaxation_comparison} and
$\min_i\mu_i\geq\mu_*$ yield
$t_{\rel}(B^\Sigma)\asymp t_{\rel}(\mu_\ab)$.
Moreover, \eqref{eq:weighted_mixing_BM} gives
$t_{\mix}^{\TV}(B^\Sigma)\asymp
t_{\mix}^{\TV}(\mu_\ab)$.

Finally, the preceding estimate, together with
\eqref{eq:geometric_high_threshold_lower_bound}
and Lemma~\ref{lem:weighted_gaussian_approximation},
gives the fixed-threshold mixing-time comparisons
between $B^\Sigma$ and $\mu_\ab$.
Combining these with
Theorem~\ref{thm: mixing and spectral comparison with Abelianization},
we obtain, for all sufficiently large $|G|$,
\[
t_{\mix}^{\TV}(\mu,\varepsilon)
-t_{\mix}^{\TV}(\mu,1-\varepsilon)
\qrel{\leq}
t_{\mix}^{\TV}(B^\Sigma,\varepsilon/4)
-t_{\mix}^{\TV}(B^\Sigma,1-\varepsilon/4).
\]
Combining these bounds proves \eqref{eq:nilpotent_mixing_window}.
\end{proof}

\begin{proof}[Justification of Remark~
\ref{rmk:large_mixing}\textup{(ii)}]
Under the asymptotic condition in Remark~\ref{rmk:large_mixing}\textup{(ii)},
we have $D\gg\mathfrak C/m=m^2\log|G|/\mu_*$. Since $m\mu_*\leq1$,
\[
\frac{\mu_*D^2}{m^3}
\qrel{\gg}\frac{m^2(\log|G|)^2}{m\mu_*}
\qrel{\geq}m^2(\log|G|)^2\quad\longrightarrow\quad\infty.
\]
Consequently, the geometric lower bound in
Theorem~\ref{thm:Geometric bounds on mixing for Abelian}
\textup{(i)} gives
\[
t_{\mix}^{\TV}(\mu_{\ab,\sym})
\qrel{\gtrsim} D_2^2
\qrel{\geq} \frac{D^2}{k}
\qrel{\gg}\frac{k^2}{\mu_*}
\qrel{\geq} \left(\sum_{i=1}^k\mu_i^{-1/2}\right)^2.
\]
Thus Proposition~\ref{prop:cutoff_lattice_product_condition}
applies to the projected walk.
The comparisons in
Remark~\ref{rmk:mixing_time_condition}\textup{(i)}
then transfer the cutoff characterization to the original walk,
as in the proof of
Theorem~\ref{thm: cutoff reduction to Abelianization}.
\end{proof}

\appendix
\section{A local limit theorem for a biased walk}\label{app:biased_local_limit}
\begin{proof}[Proof of Lemma~\ref{LLT_biased_walk}]
Set
$
\beta:=2p-1.
$
Let $U\sim\Unif[-1/2,1/2)$ be independent of $S_t$, and let $f_t$ denote the density of
$\frac{S_t-\beta t+U}{\sqrt t}$. Since
$\nu_t^p*\1_{[-1/2,1/2)}$ is the density of $S_t+U$, a change of
variables gives
\[\left\|
\nu_t^p*\1_{[-1/2,1/2)}
-
\phi_t(\cdot-\beta t)
\right\|_{L^1(\RR)}
\qrel{=}
\|f_t-\phi_1\|_{L^1(\RR)},
\]
where $\phi_1$ is the density function of the standard Normal distribution.
It therefore suffices to show that
\[\|f_t-\phi_1\|_{L^1(\RR)}
\qrel{\lesssim} t^{-1/2}.
\]

Let $h_t:=f_t-\phi_1$. It follows from Cauchy--Schwarz that
$$\|h_t\|_{L^1(\RR)}
\qrel{\leq}
\left(\int_{\RR}\frac{dx}{1+x^2}\right)^{1/2}
\left(\int_{\RR}(1+x^2)|h_t(x)|^2\,dx\right)^{1/2}.
$$
Thus it is enough to prove
\beq\label{eq:weighted_L2_biased}
\int_{\RR}(1+x^2)|h_t(x)|^2\,dx
\qrel{\lesssim} t^{-1}.
\eeq

We use the Fourier transform convention
$\widehat g(\theta):=\int_{\RR} e^{i\theta x}g(x)\,dx$ for $g\in L^1(\RR)$.
By Plancherel's theorem,
$$\int_{\RR}|h_t(x)|^2\,dx
\qrel{=}
\frac{1}{2\pi}\int_{\RR}|\widehat h_t(\theta)|^2\,d\theta
\qrel{=}
\frac{1}{2\pi}
\int_{\RR}
\left|\widehat f_t(\theta)-e^{-\theta^2/2}\right|^2\,d\theta.
$$
Moreover, since
$
\widehat{x h_t}(\theta)
=
-i\,\partial_\theta \widehat h_t(\theta),
$
another application of Plancherel's theorem gives
$$\int_{\RR}x^2|h_t(x)|^2\,dx
\qrel{=}
\frac{1}{2\pi}
\int_{\RR}
\left|\partial_\theta \widehat h_t(\theta)\right|^2\,d\theta
\qrel{=}
\frac{1}{2\pi}
\int_{\RR}
\left|
\partial_\theta \widehat f_t(\theta)
+\theta e^{-\theta^2/2}
\right|^2\,d\theta.
$$
To prove \eqref{eq:weighted_L2_biased}, it therefore suffices to establish
\beq\label{eq:fourier_H1}
\int_{\RR}
\left(
\left|\widehat f_t(\theta)-e^{-\theta^2/2}\right|^2
+
\left|
\partial_\theta \widehat f_t(\theta)+\theta e^{-\theta^2/2}
\right|^2
\right)\,d\theta
\qrel{\lesssim} t^{-1}.
\eeq

\medskip

For the rate-$1$ continuous-time biased random walk, we may write
$
S_t=N_t^+-N_t^-,
$
for two independent random variables $N_t^+\sim\operatorname{Pois}(pt)$ and $N_t^-\sim\operatorname{Pois}((1-p)t)$.
Hence,
\begin{align}\label{eq:fourier_smoothed_biased}
\nonumber\widehat f_t(\theta)&\qrel{=}\E\left[e^{i\theta \left(\frac{S_t+U-\beta t}{\sqrt{t}}\right)}\right]\qrel{=}\E\left[e^{ i\theta (S_t-\beta t)/\sqrt{t} }\right]\cdot\E\left[e^{i\theta U/\sqrt{t}}\right]\\
\nonumber&\qrel{=}e^{
t\left(
\cos(\theta/\sqrt{t})-1
+i\beta(
\sin(\theta/\sqrt{t})
-(\theta/\sqrt{t})
)
\right)
} \cdot \frac{\sin(\theta/(2\sqrt t))}{\theta/(2\sqrt t)}\\
&\qrel{=:} e^{E_t(\theta)} a_t(\theta),
\end{align}
where $a_t(\theta):=\frac{\sin(\theta/(2\sqrt t))}{\theta/(2\sqrt t)}$ and $E_t(\theta):=t(\cos(\theta/\sqrt t)-1+i\beta(\sin(\theta/\sqrt t)-\theta/\sqrt t))$.

For $j\in\ZZ$, let $I_j:=\left\{\sqrt t(2\pi j+u):-\pi\leq u\leq\pi\right\}$ so that $\mathbb{R}=\cup_{j\in \ZZ} I_j$.
For $\theta=\sqrt t(2\pi j+u)$ with $|u|\leq\pi$, we have $1-\cos u\geq 2u^2/\pi^2$ and thus \eqref{eq:fourier_smoothed_biased} leads to
\beq\label{eq:fourier_f_bound}
\left|\widehat f_t(\theta)\right|
\qrel{=}
|a_t(\theta)|
e^{t(\cos u-1)}\qrel{\leq} |a_t(\theta)|e^{-\frac{2tu^2}{\pi^2}}.
\eeq

\noindent\textit{Case 1: $j\neq 0$.} In this case, $a_t(\theta)=\frac{(-1)^j\sin(u/2)}{\pi j+u/2}$ satisfies
$$|a_t(\theta)|\qrel{\lesssim}\frac{|u|}{|j|}\qquad\text{ and }\qquad|a_t'(\theta)|\qrel{\lesssim}\frac{1}{\sqrt t\,|j|}.$$
Note also $|E'_t(\theta)|\lesssim \sqrt{t}|u|$ and hence
$$|\partial_\theta \widehat f_t(\theta)|\qrel{=}|a'_t(\theta)+a_t(\theta)E'_t(\theta)| e^{\operatorname{Re}E_t(\theta)}
\qrel{\lesssim}
\left(
\frac{1}{|j|\sqrt t }
+
 \frac{\sqrt{t} u^2}{|j|}
\right)e^{-2tu^2/\pi^2}.
$$
Combining these derivative estimates with \eqref{eq:fourier_f_bound} gives
\begin{align*}\int_{I_j}
\left(
|\widehat f_t(\theta)|^2
+
|\partial_\theta \widehat f_t(\theta)|^2
\right)\,d\theta
&\qrel{\lesssim} \sqrt{t} \int_{-\pi}^{\pi} \left( \frac{u^2}{j^2} +\frac{1}{j^2t}+\frac{tu^4}{j^2}  \right)e^{-\frac{2tu^2}{\pi^2}}\,du\qrel{\lesssim}  \frac{1}{j^2t},
\end{align*}
where the last inequality uses the fact $\int_{-\pi}^{\pi}|u|^a e^{-ctu^2}\,du\lesssim t^{-(a+1)/2}$ for $a\geq 0$.
Summing over $j\neq0$ yields
\beq\label{eq:alias_bound}
\sum_{j\neq0}
\int_{I_j}
\left(
|\widehat f_t(\theta)|^2
+
|\partial_\theta \widehat f_t(\theta)|^2
\right)\,d\theta
\qrel{\lesssim} t^{-1}.
\eeq
Moreover, since $|\theta|\geq \pi\sqrt t$ on $\bigcup_{j\neq0}I_j$, for any $c<\pi^2$
\[\int_{\bigcup_{j\neq0}I_j}
\left(
e^{-\theta^2}
+
\theta^2e^{-\theta^2}
\right)\,d\theta
\qrel{\leq}
\int_{|\theta|\geq\pi\sqrt t}
(1+\theta^2)e^{-\theta^2}\,d\theta
\qrel{\lesssim} e^{-ct},
\]
which, together with \eqref{eq:alias_bound}, proves
\beq\label{eq:nonzero_k_bound}
\int_{\cup_{j\neq 0} I_j}
\left(
\left|\widehat f_t(\theta)-e^{-\theta^2/2}\right|^2
+
\left|
\partial_\theta \widehat f_t(\theta)+\theta e^{-\theta^2/2}
\right|^2
\right)\,d\theta \qrel{\lesssim} t^{-1}.
\eeq

\medskip\noindent\textit{Case 2: $j=0$.}
It remains to consider the central interval $I_0$. We first consider the region
$
\sqrt t/2\leq|\theta|\leq\pi\sqrt t.
$
Since $1/2\leq |\theta|/\sqrt t\leq\pi$, we have
$
\cos(\theta/\sqrt t)-1
\leq
-\bigl(1-\cos(1/2)\bigr)=:-c_0<0$ and thus $|e^{E_t(\theta)}|\leq e^{-c_0t}$.
Moreover, note that in this case
\[|a_t(\theta)|\qrel{\lesssim}1,
\qquad
|a_t'(\theta)|\qrel{\lesssim} t^{-1/2}\qquad\text{and}\qquad
|E_t'(\theta)|\qrel{\lesssim}\sqrt t.
\]
It then follows that $|\widehat f_t(\theta)|
\lesssim e^{-c_0t}$ and
$
|\partial_\theta\widehat f_t(\theta)|
\lesssim \sqrt t\,e^{-c_0t}.
$
Since the length of this interval is $O(\sqrt t)$, it follows that
\begin{align}\label{eq:I0_bound}
\nonumber&\int_{\sqrt t/2\leq|\theta|\leq\pi\sqrt t}
\left(
\left|\widehat f_t(\theta)-e^{-\theta^2/2}\right|^2
+
\left|
\partial_\theta \widehat f_t(\theta)+\theta e^{-\theta^2/2}
\right|^2
\right)\,d\theta\\
\qrel{\lesssim} &
\int_{\sqrt t/2\leq|\theta|\leq\pi\sqrt t}
\left(
|\widehat f_t(\theta)|^2
+
|\partial_\theta \widehat f_t(\theta)|^2
+
(1+\theta^2)e^{-\theta^2}
\right)\,d\theta
\qrel{\lesssim} e^{-ct}
\end{align}
for some $c\in (0,\min\{c_0,1/4\})$.

We finally turn to the region $|\theta|\leq\sqrt t/2$. Recall that $\theta=\sqrt{t}u$ when $j=0$. Taylor expansion gives, uniformly in
$p\in [0,1]$,
$$a_t(\theta)\qrel{=}1+O\left(\frac{\theta^2}{t}\right)\qquad\text{and}\qquad a_t'(\theta)\qrel{=}
\frac{1}{\sqrt t}\,
\frac{u\cos(u/2)-2\sin(u/2)}{u^2}\qrel{=}O\left(\frac{|\theta|}{t}\right).$$
We decompose
\beq\label{eq:fourier_f_decomposition}
\widehat f_t(\theta)-e^{-\theta^2/2}
\qrel{=}
(a_t(\theta)-1)e^{E_t(\theta)}
+
\left(
e^{E_t(\theta)}-e^{-\theta^2/2}
\right).
\eeq
Using $|e^{E_t(\theta)}|\leq e^{-2\theta^2/\pi^2}$, the first term satisfies
\[|(a_t(\theta)-1)e^{E_t(\theta)}|
\qrel{\lesssim}
\frac{\theta^2}{t}e^{-2\theta^2/\pi^2}.
\]
For the second term, using
$
e^z-e^w
=
(z-w)\int_0^1 e^{(1-s)w+sz}\,ds,
$
with $z=E_t(\theta)$ and $w=-\theta^2/2$, we obtain
$$\left|e^{E_t(\theta)}-e^{-\theta^2/2}
\right|
\qrel{\lesssim} \left|E_t(\theta)+\frac{\theta^2}{2}\right| \cdot \int_0^1 e^{ -(1-s)\frac{\theta^2}{2}+st(\cos(\theta/\sqrt{t})-1)}\,ds.$$
Since $-(1-s)\frac{\theta^2}{2}+st(\cos(\theta/\sqrt{t})-1)\leq \theta^2(-(1-s)/2-2s/\pi^2)\leq -2\theta^2/\pi^2$ for every $s\in[0,1]$,
\begin{align*}\left|
e^{E_t(\theta)}-e^{-\theta^2/2}
\right|
&\qrel{\leq} \left|E_t(\theta)+\frac{\theta^2}{2}\right|e^{-2\theta^2/\pi^2}\\
&\qrel{\leq}
\left(
\left|t\left(\cos(\theta/\sqrt{t})-1\right)+\frac{\theta^2}{2}\right|
+
|\beta t|\cdot \left|\sin(\theta/\sqrt{t})-(\theta/\sqrt{t})\right|
\right)e^{-2\theta^2/\pi^2}\\
&\qrel{\lesssim}
\left(
\frac{\theta^4}{t}
+
\frac{|\theta|^3}{\sqrt t}
\right)e^{-2\theta^2/\pi^2}.
\end{align*}
Therefore, \eqref{eq:fourier_f_decomposition} yields
\beq\label{eq:I0_first_fourier_error}
\left|
\widehat f_t(\theta)-e^{-\theta^2/2}
\right|
\qrel{\lesssim}
\left(
\frac{\theta^2}{t}
+
\frac{|\theta|^3}{\sqrt t}
+
\frac{\theta^4}{t}
\right)e^{-2\theta^2/\pi^2}
\qrel{\lesssim}
t^{-1/2}(1+|\theta|^4)e^{-2\theta^2/\pi^2}.
\eeq

In a similar spirit, we decompose
\[\partial_\theta \widehat f_t(\theta)+\theta e^{-\theta^2/2}
\qrel{=}
a_t'(\theta)e^{E_t(\theta)}
+
(a_t(\theta)-1)E_t'(\theta)e^{E_t(\theta)}
+
(E_t'(\theta)+\theta)e^{E_t(\theta)}
+
\theta\left(e^{-\theta^2/2}-e^{E_t(\theta)}\right).
\]
Using the same Taylor estimates as above, together with
\[|a_t'(\theta)|\qrel{\lesssim} \frac{|\theta|}{t},
\qquad
|E_t'(\theta)|\qrel{\lesssim} |\theta|\qquad\text{and}\qquad
|E_t'(\theta)+\theta|
\qrel{\lesssim}
\frac{\theta^2}{\sqrt t}+\frac{|\theta|^3}{t},
\]
we obtain
\[\left|
\partial_\theta \widehat f_t(\theta)+\theta e^{-\theta^2/2}
\right|
\qrel{\lesssim}
t^{-1/2}(1+|\theta|^5)e^{-2\theta^2/\pi^2}.
\]
This, together with \eqref{eq:I0_first_fourier_error}, yields
\[\int_{|\theta|\leq\sqrt t/2}
\left(
\left|\widehat f_t(\theta)-e^{-\theta^2/2}\right|^2
+
\left|\partial_\theta \widehat f_t(\theta)+\theta e^{-\theta^2/2}\right|^2
\right)\,d\theta
\qrel{\lesssim} \frac{1}{t}.
\]
Combining this with \eqref{eq:nonzero_k_bound} and \eqref{eq:I0_bound} proves \eqref{eq:fourier_H1} and hence completes the proof.
\end{proof}

\begingroup

\renewcommand{\bibliofont}{\fontsize{9}{10}\selectfont}
\bibliographystyle{abbrv}
\bibliography{Nilpotent_references_2026-10-19_v8}
\endgroup

\end{document}